\documentclass[english]{article}
\usepackage{custom_tex}

\title{
Inference in Randomized Least Squares and PCA \\
via Normality of Quadratic Forms}

\author
{Leda Wang,
Zhixiang Zhang,
and Edgar Dobriban\footnote{
Author information:
LW: School of Mathematics,  
Univ. of Science and Technology of China,
\texttt{wangleda@mail.ustc.edu.cn}.
ZZ: Department of Mathematics, 
University of Macau,
\texttt{zhixz@wharton.upenn.edu}.
ED: Department of Statistics and Data Science,  
University of Pennsylvania,
\texttt{dobriban@wharton.upenn.edu}.
}}

\begin{document}

\maketitle
\begin{abstract}
Randomized algorithms can be used to speed up the analysis of large datasets.
In this paper, we develop a unified methodology for statistical inference via randomized sketching or projections in two of the most fundamental problems in multivariate statistical analysis:
 least squares and PCA. The methodology applies to fixed datasets---i.e., is data-conditional---and the only randomness is due to the randomized algorithm. We propose statistical inference methods for a broad range of sketching distributions, such as the subsampled randomized Hadamard transform (SRHT), Sparse Sign Embeddings (SSE) and CountSketch, sketching matrices with i.i.d.~entries, and uniform subsampling. 
To our knowledge, no comparable methods are available 
for SSE and for SRHT in PCA. 
Our novel theoretical approach rests on showing the asymptotic normality of certain quadratic forms.   As a contribution of broader interest, we show central limit theorems for quadratic forms of the SRHT,  relying on a novel proof via a dyadic expansion that leverages the recursive structure of the Hadamard transform. Numerical experiments using both synthetic and empirical datasets support the efficacy of our methods, and in particular suggest that sketching methods can have better computation-estimation tradeoffs than recently proposed optimal subsampling methods. 
  
\end{abstract}

\tableofcontents
\section{Introduction}
\label{intro}

Randomized algorithms, including
Monte Carlo methods, stochastic optimization, as well as randomized sketching and random projections, can be used to speed up the analysis of large datasets, and have a broad range of applications  
\citep{spall2005introduction,owen2019monte,
vempala2005random,halko2011algorithm, mahoney2011randomized,woodruff2014sketching,Lee:Ng:2020, cannings2021random}. 
At the same time, randomization leads to additional variability, which must be controlled.
By viewing the numerical quantity of interest as an unknown parameter, recent works have developed methods for statistical inference based on randomized sketching algorithms in both least squares regression, see e.g., 
\cite{lopes2018error,lopes2019improving,lopes2020error,wang2018optimal,wang2021optimal,yu2022optimal,ahfock2021statistical,bartan2023distributed,lee2022least,zhang2023framework,ma2022asymptotic}, etc, 
and in principal component analysis (PCA),
see e.g., 
\cite{lopes2020error}; as discussed more in \Cref{relw}.

At the moment, developing methods for statistical inference via randomized sketching algorithms requires a case-by-case analysis for each specific problem and sketching distribution.
In particular, the above works study 
two of the possibly most fundamental problems in multivariate statistics---linear regression and PCA---separately and with different methods.
Moreover, various works study different specific distributions of the sketching or projection matrices, such as Gaussian sketching 
\citep[e.g.,][etc]{lopes2020error, ahfock2021statistical,bartan2023distributed},
projections with i.i.d.~entries \citep[e.g.,][etc]{lopes2018error},
subsampling \citep[e.g.,][etc]{wang2018optimal,wang2021optimal,yu2022optimal,lopes2020error,ma2022asymptotic},
CountSketch \citep[e.g.,][etc]{ahfock2021statistical}, and
Hadamard sketching \citep[e.g.,][etc]{lopes2019bootstrap,ahfock2021statistical}.
It is unclear whether each problem requires a completely different approach and methods, or a unifying approach exists.

In this work, we develop a
unifying approach 
for statistical inference for randomized sketching methods in those two fundamental multivariate statistical problems,
least squares and PCA.
Our key insight is that both problems can be reduced to proving the asymptotic normality of certain quadratic forms of the sketching matrices.
We develop a theoretical framework (see \Cref{qfgreatfig1} for an illustration)
that makes only a few assumptions on the data, 
and develops methods for statistical inference in randomized least-squares and PCA  if appropriate 
quadratic forms are normal.
Then we apply our general theory to develop methods for a broad set of 
sketching distributions (matrices with i.i.d.~entries, subsampled randomized Hadamard transforms, sparse sign embeddings and CountSketch, uniform orthogonal sketching and uniform subsampling).
We also 
study the computational cost of our methods in detail, 
and show empirically that our sketching methods have a favorable speed-accuracy tradeoff compared to optimal subsampling methods.

In more detail,
we consider a sequence of deterministic datasets $(\mD_n)_{n\ge 1}$, following the framework of \cite{zhang2023framework}.
We are interested in a parameter $\theta_n = \theta_n(\mD_n) \in \R^d$, for some fixed positive integer $d$, which is a deterministic function of the data.
We consider the setting where the data is 
so large that we cannot access it directly to perform least squares or PCA computations.
We instead observe the output $\mA_m(\mD_n,S_{m,n})$ of a known \emph{randomized algorithm}  $\mA_m$, based on an auxiliary source of randomness $S_{m,n}$ whose distribution $Q_{m,n}$ is chosen by the user.
However, the random variable  $S_{m,n}$ may have a large size and is not assumed to be observed.
Having observed $ Z_{m,n}$, we are interested in statistical inference for $\theta_n$.

\begin{figure}[ht]
    \centering
    \includegraphics[width=0.55\linewidth]{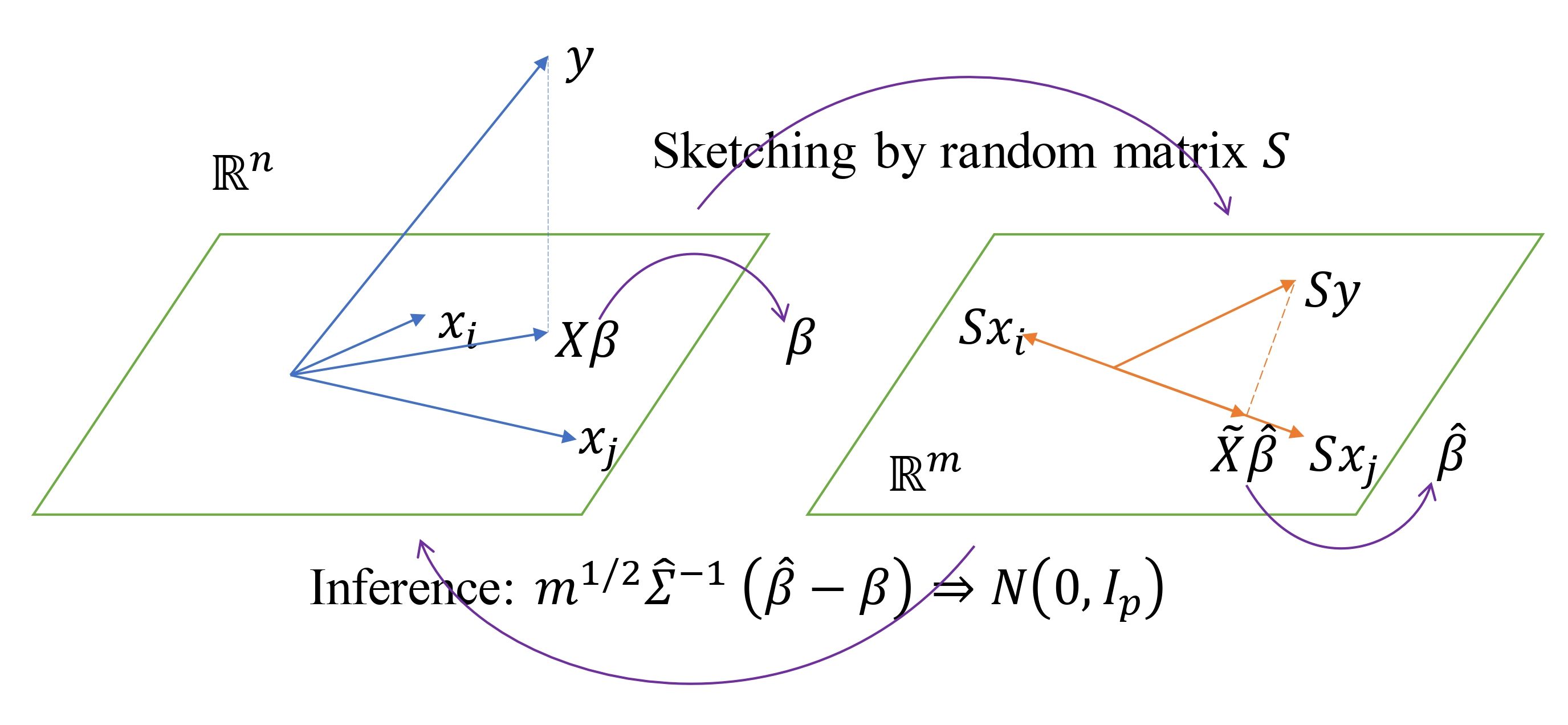}
    \caption{An illustration of 
    the geometry of the sketched least squares estimator,  highlighting our inferential objective.}
    \label{qfgreatfig1}
\end{figure}

In the  \emph{least squares problem}---see  \Cref{qfgreatfig1} for an illustration---the unobserved dataset is a deterministic $n\times p$ matrix $X_n$---where often $n$ is the sample size and $p$ is the number of features---and a deterministic $n\times 1$ vector $y_n$.
We are interested in the deterministic least squares parameter
 $\beta_n= (X_n^\top X_n)^{-1} X_n^\top y_n$.
For an $m\times n$ random sketching matrix $S_{m,n}$, 
the sketched data is
$(\tX_{m,n}, \ty_{m,n}) = (S_{m,n}X_n,S_{m,n}y_{n})$.
We consider two possible estimators of $\beta_n$, the 
\emph{sketch-and-solve} (or, complete sketching) least squares and \emph{partial sketching} estimators \citep[e.g.,][etc]{sarlos2006improved,drineas2006sampling},  defined, respectively as:
\begin{equation}\label{ss} 
\hbs=\left(\tX_{m,n}^\top \tX_{m,n}\right)^{-1}\tX_{m,n}^\top \ty_{m,n}  \quad \text{and} \quad \hbp=\left(\tX_{m,n}^\top \tX_{m,n}\right)^{-1}X_n^\top y_{n}.\end{equation}
The observed data for sketch-and-solve regression is 
$(\tX_{m,n}, \ty_{m,n}) = (S_{m,n}X_n,S_{m,n}y_{n})$, 
and for partial sketching, it is $S_{m,n}X_n$ and 
$X_n^\top y_{n}$. 
In some cases, $X_n^\top y_{n}$ can be computed relatively efficiently, and using it may improve accuracy.

In the  \emph{principal component analysis problem}, 
the unobserved dataset is the $n\times p$ matrix $X_n$, while 
the observed dataset is
$\tX_{m,n} = S_{m,n}X_n$.
Our targets of inference are the singular values  of $X_n$---or equivalently the eigenvalues of $X_n^\top X_n$---and the right singular vectors of $X_n$.
We consider the corresponding singular values and vectors of $\tX_{m,n}$ as estimators. 

\begin{figure} 
    \centering
    \includegraphics[width=1\textwidth]{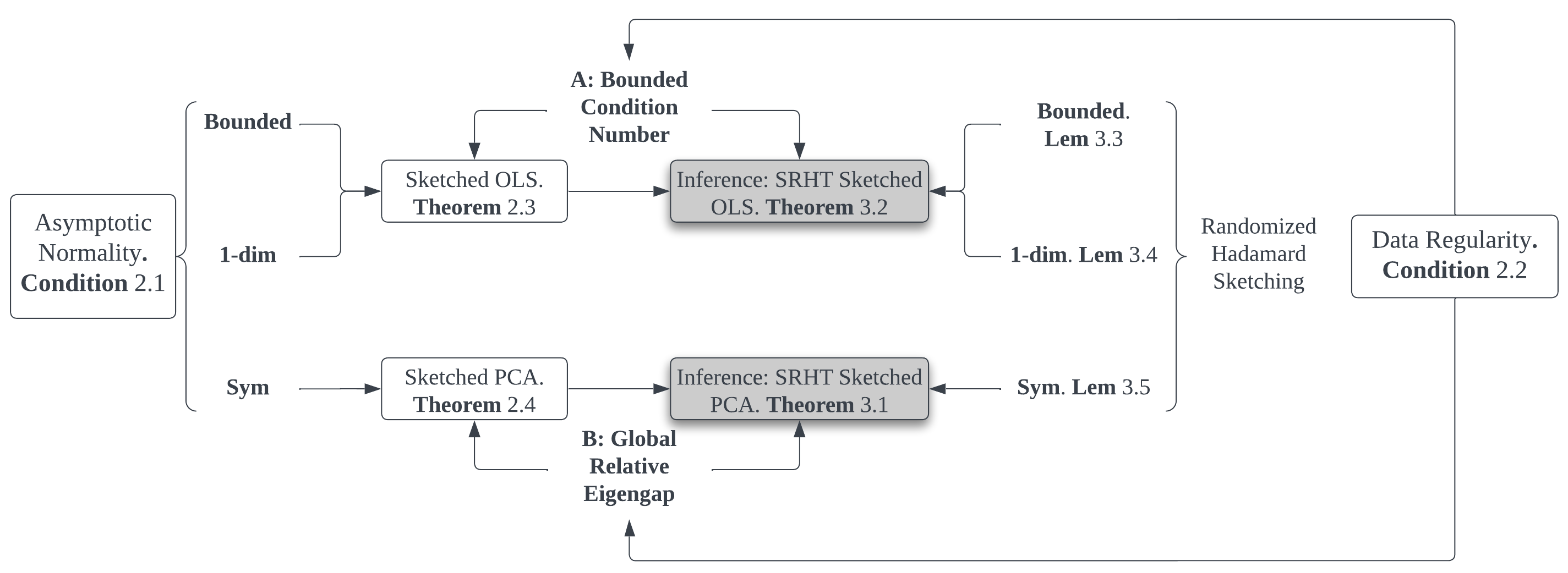}
    \caption{A flowchart illustrating our theoretical framework.
    This graph shows the connections between the various conditions  and results to develop methods for   statistical inference in  least squares (LS, also referred to as ordinary least squares or OLS) and PCA using SRHT (Subsampled Randomized Hadamard Transform) sketching. In particular, note that all components of Condition \ref{generalquadratic forms} are used directly in our key general inference results.}
    \label{SRHT_flowchart}
\end{figure}
\subsection{Contributions}
\begin{enumerate}
    \item {\bf Unifying approach to statistical inference in randomized least squares (LS) and PCA.} We propose a unifying approach for statistical inference in randomized least squares and PCA (Section \ref{secframework}). 
    Our approach reduces this to showing the asymptotic normality of certain sequences of quadratic forms.
    See Figure \ref{SRHT_flowchart}
    for the logical flow of ideas and the connections between the various results in our framework. 
    The methodology applies to fixed datasets---i.e., is data-conditional---under mild conditions such as a bounded condition number.
    In particular, we do not impose any statistical model (such as i.i.d.-ness or a linear model) 
    on the data.
    
    Specifically, we propose methods for inference in 
    sketched least squares (Theorem \ref{lemgeneols}), 
    under the asymptotic normality of quadratic forms from Condition \ref{generalquadratic forms}.
    We also develop methods for inference in  sketched PCA (Theorem \ref{infPCAgeneral}), when certain 
    matrix-valued quadratic forms converge to a multivariate Gaussian distribution. 
    See Table \ref{sumasyPCA} for a summary.
    For completeness, we also show how our results can be used for inference for population parameters when given an i.i.d.~sample from a population.
    
    \item 
    {\bf  Inference methods for a broad range of sketching distributions.} 
    We then propose statistical inference methods for a broad range of sketching distributions, such as 
    the subsampled 
randomized Hadamard transform (SRHT, Section \ref{srht}),
Sparse Sign Embeddings (SSE) and CountSketch (Section \ref{cs-sse}),
sketching matrices with i.i.d.~entries (Section \ref{iid}), 
uniform partial orthogonal or Haar sketches (Section \ref{haar}), 
and uniform subsampling (Section \ref{unif}).
For several of these, either no comparable methods have been available before (e.g., for SSE, and for SRHT in PCA), or our results allow for broader conditions on the data (e.g., for i.i.d.~sketches, and for SRHT in LS).

      \item {\bf Technical contributions: CLTs for quadratic forms.}
      As a technical contribution of broader interest, we show novel central limit theorems for quadratic forms of  sketching matrices. 
      This includes the SRHT (Lemmas \ref{lemgeneralqfsrht} and \ref{lemsymqfsrht}), which relies on a novel proof via a dyadic expansion, leveraging the Hadamard transform.
        It also includes SSE and CountSketch  (Lemmas \ref{lemgeneralqfcs} and \ref{lemsymqfcs}), which relies on 
        the analysis of a weighted degenerate U-statistic of order two via
        delicate moment calculations.
        CLTs for random quadratic forms are an important tool in random matrix theory, see  e.g. \citep{bai2008central}; and thus we think our results are of broader interest.
      
    \item {\bf Computational analysis and empirical experiments.} 
    We analyze and compare the computational cost of various sketching methods (Section \ref{comp}). Additionally, we conduct numerical experiments using both synthetic and empirical data (Section \ref{secsimulation}), which support the efficacy of our methods. 
    A comparison of our methods with certain optimal subsampling approaches (\Cref{csk}) supports that sketching methods achieve a better empirical time-accuracy tradeoff.
    Our experiments are reproducible with code provided at \url{github.com/Futwangalerda/normality-QF-sketching-apps}. 
    \end{enumerate}
\subsection{Related Work}
\label{relw}
Sketching and random projection methods have been  studied extensively in multiple domains \citep[etc]{vempala2005random, li2006very, halko2011finding, mahoney2011randomized, woodruff2014sketching, drineas2016randnla, martinsson2020randomized, cannings2021random}.
They have been applied to various problems, 
including 
PCA and SVD \citep[e.g.,][etc]{frieze2004fast,halko2011finding, yang2021reduce,zhang2022perturbation},
linear and 
ridge regression \citep{lu2013faster,chen2015fast, wang2017sketched, dobriban2019asymptotics, liu2019ridge, lacotte2020limiting, derezinski2021sparse}, 
two sample testing \citep{lopes2011more,srivastava2016raptt},
 testing point null hypotheses in non-parametric regression \citep{liu2019sharp},
testing in single-index models \citep{liu2022random}, nonparametric testing \citep{liu2019sharp}, and convex optimization \citep{pilanci2015randomized,pilanci2016iterative,pilanci2017newton}, etc.

Works such as \cite{lopes2018error, ahfock2021statistical, lee2022least, bartan2023distributed} have developed methods for statistical inference in least squares regression.
\cite{ahfock2021statistical} 
study the statistical properties of
complete, partial, and certain one-step corrected sketching estimators in least squares.
They develop methods for finite-sample exact inference for Gaussian sketching.
For Hadamard sketching and the CountSketch, they show the asymptotic normality of the sketched estimators, using a technical approach that is different from ours. 
Compared to \cite{ahfock2021statistical}, we develop a unifying framework that includes both PCA and least squares, study the setting where the sketch dimension diverges, and consider a broader range of sketching distributions. 

Bootstrap methods for inference using sketching have been studied separately for least squares 
\citep{lopes2018error} and PCA \citep{lopes2020error}, while we develop a unified approach.
Moreover, they require certain convergence conditions for the unobserved data, while we only need some regularity conditions (such as a bounded condition number), by using subsequence arguments.
\cite{lopes2020error} considers  
Gaussian sketching methods and subsampling in PCA, again under certain convergence conditions.

\cite{zhang2023framework} propose a framework for statistical inference via randomized algorithms.
They consider arbitrary algorithms, only assuming that their output has a limiting error distribution around the true quantity/parameter of interest, without knowing what the distribution is.
They develop a range of methods for statistical inference in this context, including sub-randomization---generalizing subsampling \citep{politis1999subsampling}---multi-run plug-in and multi-run aggregation.
When the limiting distribution 
belongs to a known parametric family with unknown parameters, they consider inference based on estimating those parameters.
Our setting belongs to this latter case, 
and our main contribution is to show how to derive the limiting distributions in a unifying way in sketching for linear regression and PCA.

\cite{wang2018optimal,wang2021optimal,yu2022optimal} investigate optimal subsampling for logistic regression,  quantile regression models as well as for more general maximum quasi-likelihood estimators. 
They establish consistency and asymptotic normality of weighted subsampling estimators, 
and derive optimal subsampling probabilities that minimize the asymptotic mean squared error, as well as methods to estimate these probabilities.
These approaches impose a regression model on the data, while we consider the data as fixed. 
\cite{ma2022asymptotic} study 
weighted sampling estimators in a fixed-data setting, where the sample size is fixed. In contrast, we consider a growing sample size. 

\subsection{A Concrete Result}

To give the reader a concrete example of the type of results that we can obtain, we state here a method for inference in PCA using the Subsampled Randomized Hadamard Transform (SRHT) \citep{ailon2006approximate}, see \Cref{srht}.
In \Cref{PCAsrht}, we consider an $m\times n$, $m\le n$, 
SHRT sketching matrix $S_{m,n}$, 
data matrices $(X_n)_{n\ge  1}$
satisfying an eigengap condition, 
with vanishing leverage scores (i.e., 
row norms of their left singular matrices).
We show that 
for any $i\in[p]$,
as $m,n\to \infty$ such that $m/\log^2 n\to\infty$ with $p$ fixed, 
we have for the unobserved data eigenvalues $\Lambda_{n,i}=\lambda_i(X_n^\top X_n)$
and the observed sketched eigenvalues $\hat{\Lambda}_{m,n,i}=\lambda_i(X_n^\top S_{m,n}^\top S_{m,n} X_n)$ that with $\gamma_n = m/n$,
\begin{equation}\label{infPCAvalsrht}
\sqrt{\frac{m}{3(1-\gamma_n)}}\hat{\Lambda}_{m,n,i}^{-1}\left(\hat{\Lambda}_{m,n,i}-\Lambda_{n,i}\right)\Rightarrow \N(0,1).
\end{equation}
Letting $z_{1-\alpha/2}$ be the normal $1-\alpha/2$ quantile for some $\alpha\in(0,1)$,
this result can be used to form  confidence intervals $[\hat{\Lambda}_{m,n,i}(1-z_{1-\alpha/2}\sqrt{3(1-\gamma_n)/m})  , \hat{\Lambda}_{m,n,i}(1+z_{1-\alpha/2}\sqrt{3(1-\gamma_n)/m})]$
with $1-\alpha$ asymptotic coverage
for the data eigenvalues $\Lambda_{n,i}$ based on the sketched eigenvalues $\hat{\Lambda}_{m,n,i}$.

\subsection{Notation and Definitions}
\label{not}

We present a self-contained summary 
of notations needed in our work.
The reader may skip to \Cref{secframework} 
and refer back to this section as needed.

{\bf Basics.}
For a positive integer $p\ge 1$, 
we let $[p] = \{1,\ldots,p\}$.
For a vector $a=(a_1,\ldots,a_n)^\top\in\R^n$, we 
write $\|a\|_q=(\sum_{i=1}^n |a_i|^q)^{1/q}$
for its $\ell_q$ norm,
where $1\le  q<\infty$. 
For $q=\infty$, we denote $\|a\|_\infty=\max_{1\le  i\le  n}|a_i|$. 
For any $n\ge 1$, we define $\nsp$ to be the unit sphere in Euclidean space $\R^n$. 
For any real numbers $x$ and $y$, 
we define ${\lceil x\rceil}$ as the smallest integer greater than or equal to $x$, and $x\wedge y=\min\{x,y\}$.
We use the symbol ``$\mrd$" to denote the differential of a function.  
For an event $A$, $I\{A\}$ denotes its indicator function, which equals to unity if $A$ happens and to zero otherwise. 

{\bf Convergence.}
For two positive sequences $\{a_n\}_{n \ge 1},\{b_n\}_{n \ge 1}$, we write $a_n=O(b_n)$ if there exists a positive constant $C$ such that $a_n / b_n \le C$ and we write $a_n=o(b_n)$ if $a_n / b_n \rightarrow 0$. 
Moreover, $O_P(\cdot)$ and $o_P(\cdot)$ have similar meanings to $O(\cdot)$ and $o(\cdot)$ respectively, but hold asymptotically with probability tending to one.
For two positive sequences $\{a_n\}_{n \ge 1},\{b_n\}_{n \ge 1}$, we write $a_n=\Theta(b_n)$ if $a_n=O(b_n)$ and $b_n=O(a_n)$.
We also denote weak convergence of a sequence of probability measures $\{Q_n\}_{n\ge 1}$ to
a probability measure $Q$ by $Q_n \Rightarrow Q$.

{\bf Matrices.} 
Let $\Sp$ be the set of $p\times p$ positive definite matrices.
For a vector $(v_1,\ldots, v_n)^\top$, we define $\diag(v_1,\ldots, v_n)$ as the diagonal matrix with diagonal entries $v_i$ for $i\in [n]$.
For positive integers $m, n$,
and for a matrix $A\in\R^{m\times n}$, we denote the operator norm of $A$ as $\|A\|$.
If further $m\le n$, the Stiefel manifold $W_{m,n}$
is the set of ordered orthonormal $m$-tuples of vectors in $\R^n$.
Recall that for square matrices $A, B$, 
$A\preceq B$ denotes that $B-A$ is positive semi-definite.
For two matrices $A,B$ of the same size $k\times$, $A\odot B$ is Hadamard product where $(A\odot B)_{ij}=A_{ij}B_{ij}$ for all  $(i,j)\in[k]\times[l]$.
For two matrices $A,B$,
$A\otimes B$ denotes the tensor product.
For a sequence of matrices $\{A_n\}_{n\ge 1}$ with fixed dimensions, we write $A_n = O(1)$ if $\|A_n\| = O(1)$.

For $i\in [p]$,
we let $E_{ii}$ be a $p\times p$ matrix with 
the $(i,i)$-th entry equal to unity, 
and other entries equal to zero.
For positive integers $k,l$,
we often denote  the  columns 
of 
a $k\times l$ matrix $A$ by $(a_j)_{j\in[l]}$. 
To study limiting distributions of matrices, it will be convenient to work with their vectorizations:

\begin{definition}[Vectorization]\label{vi}
We define $\vec\colon\R^{k\times l}\to\R^{kl}$ as the vectorization operator listing the columns of a matrix placing the leftmost one on top, i.e.,
for a $k\times l$ matrix $A$,
$\vec(A)= (a_1^\top,\ldots,a_p^\top)^\top$. 
\end{definition}

Vectorization provides an identification of $[kl]$ and $[k]\times [l]$---whose elements we denote by either $(i,j)$, $(ij$), or $ij$---so that $\vec(A)_{(i,j)}=A_{ij}$.
Given indices $(i,j) \in [k] \times [l]$
of a matrix, we refer to the 
corresponding indices in the vectorization of the matrix as follows:

\begin{definition}[Vectorized Indexing of a Matrix]\label{mii}
For $(i,j) \in [k] \times [l]$,
$\sigma_{k,l}(i,j) \in [kl]$ is defined as the index of $A_{ij}$ in the vectorization $\vec(A)$, 
so that 
$\vec(A)_{\sigma_{k,l}(i,j)}= \vec(A)_{(i,j)}=A_{ij}$.
\end{definition}

Similarly, we can represent the indices in $[p(p+1)/2]$ as $\{(11), (12),\ldots,(1p), (22),\ldots, (2p), \ldots, (pp)\}$, which are the indices in the upper triangular part of a $p\times p$ matrix.
When dealing with matrices that have one or both dimensions of size $kl$ or $p(p+1)/2$,
we may use the above forms of indexing for those dimensions.

When working with symmetric matrices, we will often need to extract their diagonals and above-diagonal components.
To apply this operation to a vectorized $p\times p$ matrix, we define the following upper triangulation operator.
Recall that for $i,j\in \R$, $\delta_{ij}$ is the Kronecker delta which equals unity of $i=j$ and zero otherwise.


\begin{definition}[Upper Triangulation Operator]\label{uto}
For every $p\in\NN_{>0}$, 
we consider the upper triangulation operator $D_p: \R^{p^2}\to\R^{p(p+1)/2}$, 
which in a matrix form takes values
$(D_p)_{(ij),(kl)}=\delta_{ik}\delta_{jl}$, $i,j,k,l\in[p]$, and $k\le  l$. 
\end{definition}

We will require the following 
definition to describe covariance matrices 
of vectorized $p\times p$ matrices.

\begin{definition}[Absolutely Symmetric Matrix]\label{abs}
For a positive integer $p$,
we say that a $p^2\times p^2$ positive semi-definite matrix $G$ is absolutely symmetric if $G_{(ij),(kl)} = G_{(ij),(lk)} = G_{(ji),(kl)} = G_{(ji),(lk)}$ for all $i,j,k,l\in[p]$, and 
with the  the upper triangulation operator $D_p$ from Definition \ref{uto},
$D_p^\top G D_p$ is positive definite.     
\end{definition}

For an absolutely symmetric matrix $G$, there is a one-to-one correspondence between $D_p^\top G D_p$ and $G$. Indeed,  although $D_p^\top G D_p$ is a principal submatrix of $G$, it contains all distinct elements of $G$.

We will also need the following two special matrices.

\begin{definition}[Special Matrices $Q_p$ and $P_p$]\label{spmx}
For a positive integer $p$,
we denote $Q_p=\vec(I_p)\vec(I_p)^\top$.
Moreover, 
we denote the  $p^2\times p^2$ 
matrix representing matrix transposition in the vectorized matrix space by $P_p$, 
such that
$(P_p)_{ij,kl}=\delta_{il} \delta_{jk}$ for all $i,j,k,l\in[p]$. 
\end{definition}

\section{Unifying Approach via Normality of Quadratic Forms}\label{secframework}

In this section, we present the core results 
of our unifying approach 
for inference in randomized PCA and LS
via normality of quadratic forms.

\subsection{Unifying Conditions}

Notice that the sketch-and-solve least squares solution from \eqref{ss} 
can be written as 
\begin{equation*}
   \hbs = (X_n^\top S_{m,n}^\top S_{m,n} X_n)^{-1} X_n^\top S_{m,n}^\top S_{m,n} y_n.
\end{equation*}
Therefore, we expect that
\emph{quadratic forms} 
$a_n^\top S_{m,n}^\top S_{m,n} b_n$ for appropriate vectors $a_n,b_n$ play a crucial role in determining the distribution of this estimator.
Similar observations can be made about the partial sketching estimator as well as about the singular values and vectors of the sketched data.
Our analysis confirms this intuition.

In addition, if we
let the SVD\footnote{Here, $U_n$ denotes a $n\times p$ partial orthogonal matrix satisfying $U_n^\top U_n = I_p$, $L_n= \diag(\ell_{n,1},\ldots,\ell_{n,p})$, where $\ell_{n,1}\ge \ldots\ge \ell_{n,p}>0$ are the singular values of $X_n$ in non-decreasing order, 
and $V_n$ is a $p\times p$ orthogonal matrix containing the right singular vectors of $X_n$.} of $X_n$ be $X_n = U_n L_n V_n^\top$, then
\begin{equation}\label{geneformucomplete0}
\begin{aligned}
   \hbs& =  V_n L_n^{-1} (U_n^\top S_{m,n}^\top S_{m,n} U_n)^{-1} U_n^\top S_{m,n}^\top S_{m,n} y_n.
\end{aligned}\end{equation}
Hence, 
due to the form of $U_n^\top S_{m,n}^\top S_{m,n} y_n$, 
$U_n$ is expected to play a special role.
We introduce
an $n\times (n-p)$ orthogonal complement $U_{n,\perp}$ 
of $U_n$, 
so that the matrix $[U_n, U_{n,\perp}]$ forms an $n\times n$ orthogonal matrix.
We let $\ep_n=y_n-X_n \beta_n$ be the true residuals.
We also let $\bar\ep_n=\ep_n/\|\ep_n\|$ 
be the normalized residuals,
and
$\widebar{X_n\beta_n} = X_n\beta_n/\|X_n\beta_n\|$, 
be the normalized fitted values.
It will be assumed that $\ep_n\neq 0$ and $\widebar{X_n\beta_n}\neq 0$, respectively, whenever these are used.
We also let $\Lambda_n = L_n^2$.

The normality conditions of quadratic forms are 
required to hold only for sequences of vectors $({a_n, \ta_n})_{n\ge 1}$ such that $a_n, \ta_n\in\nsp$ for all $n$ and $({a_n, \ta_n})_{n\ge 1}$ satisfy 
certain conditions---such as certain forms of delocalization, as seen informally in Table \ref{sumqf}---that depend on the sketching matrix $S_{m,n}$. 
Typically, the more  ``randomness" $S_{m,n}$ has, the weaker these conditions. For instance, Gaussian sketches require weak conditions (but are computationally impractical), while uniform subsampling requires stronger conditions.
To handle this in a general way,
we introduce abstract conditions called Condition $\mathfrak{p}$, $\mathfrak{P}$, and $\mathfrak{P}'$. These will depend on the specific distributions of sketching matrices used, and will be specified in each case in \Cref{spec}.
Since we are aiming to develop a theory with near-optimal yet unified conditions, our next condition lists several types of properties that we may require of quadratic forms.
While this condition may seem complicated, each component is used crucially in our framework; see \Cref{SRHT_flowchart} for an illustration.

\begin{condition}[Asymptotic Normality of Quadratic Forms]\label{generalquadratic forms}
Consider the asymptotic regime where $n\to \infty$, $m=m_n\to\infty$ with $m<n$, while 
$p$ is fixed.
Let $(S_{m,n})_{m,n\ge 1}$ be a sequence of $m\times n$ sketching matrices, and $U_n$ be 
the $n\times p$ left singular matrix of $X_n$. 
Let $\bar\ep_n=\ep_n/\|\ep_n\|$, 
$\widebar{X_n\beta_n} = X_n\beta_n/\|X_n\beta_n\|$, assuming they are well-defined when used.
We consider a sequence of positive scalars $(\tau_{m,n})_{m,n\ge  1}$, such that
$\tau_{m,n}/m\to 0$.
Consider 
Conditions $\mathfrak{p}$, $\mathfrak{P}$,
which refer to sequences of vectors $({a_n, \ta_n})_{n\ge 1}$ such that $a_n, \ta_n\in\nsp$ for all $n$ and $({a_n, \ta_n})_{n\ge 1}$, 
and 
Condition $\mathfrak{P}'$
for the sequence of matrices $(U_n)_{n\ge 1}$. 
Consider the following boundedness condition:
\begin{compactitem}
\item {\bf \textup{\texttt{Bounded}}: Bounded Quadratic Form Under Condition $\mathfrak{p}$}.
For any sequence $({a_n, \ta_n})_{n\ge 1}$ 
 of vector-pairs
such that $a_n, \ta_n\in\nsp$ for all $n$ and $({a_n, \ta_n})_{n\ge 1}$ satisfy Condition $\mathfrak{p}$, we have 
\begin{equation*}\label{boundquadform}
\sqrt{m} \tau_{m,n}^{-1/2}\left(a_n^\top S_{m,n}^\top S_{m,n} \ta_n - a_n^\top \ta_n\right)=O_P(1).
\end{equation*}
\end{compactitem}

Moreover, consider the following asymptotic normality conditions for quadratic forms:
\begin{compactitem}
\item {\bf \textup{\texttt{1-dim}}: One-dimensional Quadratic Form Under Condition $\mathfrak{P}$}. For any sequences of vectors $({a_n, \ta_n})_{n\ge 1}$ such that $a_n, \ta_n\in\nsp$ for all $n$,
and $({a_n, \ta_n})_{n\ge 1}$ satisfy Condition $\mathfrak{P}$,
there exists a sequence
$\left(\sigma_n^2(a_n, \ta_n)\right)_{n\ge 1}$ uniformly bounded away from zero and infinity, such that
\begin{equation*}\label{1dimquadform}
\sqrt{m} \tau_{m,n}^{-1/2} \sigma_n^{-1}(a_n, \ta_n) \left(a_n^\top S_{m,n}^\top S_{m,n} \ta_n - a_n^\top \ta_n\right)\Rightarrow \N(0, 1).
\end{equation*}
Moreover, 
for all $n\ge 1$,  
there is $M_n=M_n(\bar \ep_n)$
and 
$\Mp=\Mp(\widebar{X_n\beta_n})\in\R^{p\times p}$,
also possibly depending on $U_n$,
such that for any $b_n\in\psp$ 
with $(U_nb_n, \bar{\ep}_n)_{n\ge 1}$,
$(U_nb_n, \widebar{X_n\beta_n})_{n\ge 1}$,
satisfying Condition $\mathfrak{P}$, $\sigma_n^2$ has the form
\begin{align}
    \label{defMn}\sigma_n^2(U_nb_n, \bar{\ep}_n)&=b_n^\top M_n(\bar \ep_n) b_n;\qquad
    \sigma_n^2(U_nb_n, \widebar{X_n\beta_n})=b_n^\top \Mp(\widebar{X_n\beta_n})b_n.
\end{align}
As a special case, for $\alpha\in\{0,1\}$, 
consider Condition {\bf \textup{\texttt{1-dim'}}}--- the special case of Condition  {\bf \textup{\texttt{1-dim}}}---where 
$\sigma_n^2(a_n, \ta_n)=1+(\alpha+1)(a_n^\top \ta_n)^2$ for all $n\ge 1$ and $({a_n, \ta_n})_{n\ge 1}$ satisfying Condition $\mathfrak{P}$.

\item {\bf \textup{\texttt{Sym}}: Symmetric Matrix Quadratic Form Under Condition $\mathfrak{P}'$}. For $(U_n)_{n\ge 1}$ satisfying Condition $\mathfrak{P}'$, there exists 
a sequence of $p^2\times p^2$ absolutely symmetric matrices $(G_n)_{n\ge 1}$ as per Definition \ref{abs}, 
such that there is $0<c<C<\infty$
for which 
the upper triangulation operator $D_p$ from Definition \ref{uto} satisfies
$cI_{p(p+1)/2}\preceq D_p^\top G_n D_p\preceq CI_{p(p+1)/2}$, 
and vectorizing as per Definition \ref{vi},
\begin{equation}\label{symmatform-1}
\sqrt{m} (\tau_{m,n}D_p^\top G_n D_p)^{-1/2} D_p^\top \vec(U_n^\top S_{m,n}^\top S_{m,n} U_n - I_p) \Rightarrow \N(0, I_{p(p+1)/2}).
\end{equation}
As a special case, consider \textup{Condition} {\bf \texttt{Sym-1}}, under which there exists 
a $p^2\times p^2$ absolutely symmetric matrix $G$ as per Definition \ref{abs}, such that
\begin{equation}\label{symmatform}
\sqrt{m} \tau_{m,n}^{-1/2} \vec(U_n^\top S_{m,n}^\top S_{m,n} U_n - I_p) \Rightarrow \N(0, G).
\end{equation}

As an even more special case, 
we consider \textup{Condition} {\bf \textup{\texttt{Sym-1'}}} 
for $\alpha\in\{0,1\}$,
such that for $P_p$ and $Q_p$ from Definition \ref{spmx},
$G=I_{p^2}+P_p+\alpha Q_p$.

\end{compactitem}
\end{condition}

It is not hard to see that under Condition \ref{generalquadratic forms} \textup{\texttt{1-dim'}}, we can take
$M_n=I_p$.
Moreover, 
in most cases---but not always---we
take 
$\mathfrak{p}=\mathfrak{P}$
and 
use the results in Condition \ref{generalquadratic forms} \textup{\texttt{1-dim}} to derive Condition \ref{generalquadratic forms} \textup{\texttt{Bounded}}. 
However,
Condition \ref{generalquadratic forms} \textup{\texttt{Bounded}} is 
more general than Condition \ref{generalquadratic forms} \textup{\texttt{1-dim}}.


Moreover, we will work under some
mild conditions on the singular values of the data matrix.
In particular, our theory and methods do not require that the sequence of data matrices converge in any way.

\begin{condition}[Condition Number and Eigengap]\label{spectralcondition}
For a sequence 
$(X_n)_{\ge 1}$
of $n\times p$ data matrices, 
where $\ell_{n,1}\ge \ldots\ge \ell_{n,p}>0$ are the singular values of $X_n$, 
consider the following conditions:
\begin{compactitem}
\item {\bf A: Bounded Condition Number.}
We have
$
\limsup_{n\to \infty} \ell_{n,1}/\ell_{n,p}<\infty$.
\item {\bf B: Global Relative Eigengap.} 
Condition \ref{spectralcondition} \textup{\texttt{A}} holds, and further
\begin{equation*}\label{itheigengap}
\liminf_{n\to \infty} \min_{i\in[p-1]}\frac{\ell_{n,i}}{\ell_{n,i+1}}>1.
\end{equation*}
\end{compactitem}
\end{condition}

Condition \ref{spectralcondition} \textup{\texttt{A}} ensures that the condition number of the data matrix is bounded, which is required 
to use tightness arguments in the proofs.
Conditions  \textup{\texttt{B}} ensures that the eigengaps 
are sufficiently large, 
and in particular that the singular values are distinct,
which
is helpful for performing inference on the individual eigenvalues\footnote{One could 
relax this assumption and aim to conduct inference for singular values associated with arbitrary singular spaces; we leave this to future work.}.
In our work, we will impose either one of \textup{\texttt{A}} or \textup{\texttt{B}}, as specified below.

\subsection{Asymptotic Distributions under Unifying Conditions}

\subsubsection{Sketched Least Squares}
Recall that 
we are interested in the deterministic least squares parameter
 $\beta_n= (X_n^\top X_n)^{-1} X_n^\top y_n$.
Under Conditions \ref{generalquadratic forms} and \ref{spectralcondition}, 
we can derive the 
asymptotic distributions of
the sketched LS solutions, as well as methods for statistical inference.
See \Cref{pflemgeneols} for the proof of this result.

\begin{theorem}
[Limiting Distribution and Inference for Sketched LS Estimators]\label{lemgeneols}
Let $(X_n)_{n\ge 1}$ be a sequence of data matrices satisfying Condition \ref{spectralcondition} \textup{\texttt{A}}, and $(S_{m,n})_{m,n\ge 1}$ be a sequence of sketching matrices.
Let $\bar\ep_n=\ep_n/\|\ep_n\|$, 
$\widebar{X_n\beta_n} = X_n\beta_n/\|X_n\beta_n\|$, assuming they are well-defined when used.
Suppose that 
for any $(b_n)_{n\ge 1}$
satisfying $b_n\in\psp$ for all $n$,
$({U_nb_n, U_nb_n})_{n\ge 1}$
and $({U_nb_n, \bar{\ep}_n})_{n\ge 1}$ satisfy Condition $\mathfrak{P}$. 
Then under 
Conditions \ref{generalquadratic forms} \textup{\texttt{Bounded}} and \textup{\texttt{1-dim}}, 
the sketch-and-solve LS estimator $\hbs$ from \eqref{ss} has the limiting distribution
\beq\label{ssa}
m^{1/2}\tau_{m,n}^{-1/2}\left( V_n L_n^{-1} M_n(\bar \ep_n)L_n^{-1} V_n^\top \right)^{-1/2}\|\ep_n\|^{-1}(\hbs-\beta_n) \Rightarrow \N(0,I_p),
\eeq
 with $M_n$ from \eqref{defMn}.
Moreover, under 
Condition \ref{generalquadratic forms} \textup{\texttt{Bounded}} and \textup{\texttt{1-dim}} with $\Mp$ from \eqref{defMn}, 
the partial sketching LS estimator $\hbp$ from \eqref{ss} has the limiting distribution
\beq\label{ssb} 
m^{1/2}\tau_{m,n}^{-1/2}\left(V_n L_n^{-1} \Mp(\widebar{X_n\beta_n}) L_n^{-1} V_n^\top \right)^{-1/2}\|X_n\beta_n\|^{-1}(\hbp-\beta_n) \Rightarrow \N(0,I_p).\eeq 
In the special case that 
Condition \ref{generalquadratic forms} \textup{\texttt{1-dim'}} holds 
 for $\alpha\in\{0,1\}$,
we further have
the following inferential results for the unobserved least squares parameter $\beta_n$ based
on the observed $\tX_{m,n},\ty_{m,n}$, and $\tep_{m,n} = \ty_{m,n}-\tX_{m,n}\hbs$:
\beq\label{1dim'lsinf1}
m^{1/2}\tau_{m,n}^{-1/2}\|\tep_{m,n}\|^{-1}(\tX_{m,n}^\top \tX_{m,n})^{1/2}(\hbs-\beta_n) \Rightarrow \N(0,I_p)
\eeq
and
\beq\label{1dim'lsinf2} m^{1/2}\tau_{m,n}^{-1/2}\left(
\|\tX_{m,n} 
\hbp\|^2 \cdot (\tX_{m,n}^\top \tX_{m,n})^{-1}+ (\alpha+1)\hbp \hbpt\right)^{-1/2}(\hbp-\beta_n) \Rightarrow \N(0,I_p).\eeq

\end{theorem}
The results \eqref{1dim'lsinf1}  and 
\eqref{1dim'lsinf2} can be used in the standard way to form confidence regions for the unknown least squares parameter $\beta_n$ based on the observed $\tX_{m,n},\ty_{m,n}$.
To form some intuition for \eqref{ssa},
if Condition {\bf \textup{\texttt{1-dim'}}} holds, then we can take $M_n  = I_p$, and so  \eqref{ssa}
reduces to 
\beq\label{simols}
m^{1/2}\tau_{m,n}^{-1/2}\left(X_n^\top X_n \right)^{-1/2}\|y_n-X_n \beta_n\|^{-1}(\hbs-\beta_n) \Rightarrow \N(0,I_p),
\eeq
showing that the effective covariance matrix
is a multiple of $(X_n^\top X_n^\top )^{-1}$.
This form is familiar from classical multivariate statistics based on i.i.d.~data, and shows that our results can be viewed as a generalization.
However, our theorem applies to arbitrary fixed data $X_n, y_n$, and imposes no distributional assumptions such as i.i.d.-ness. 

In the above result, 
statistical inference requires a 
stronger condition (Condition \ref{generalquadratic forms} {\texttt{1-dim'}) than the existence of the asymptotic distribution.
It turns out that guarantees for
statistical inference more generally can be better handled case-by-case for various sketching distributions.
In the remainder of the paper, we will use this result to derive limiting distributions and inferential methods for various distributions of sketching matrices. 
Table \ref{sumasyols} 
summarizes some of these results. 

While our results are data-conditional and do not assume that the data was generated by a statistical model, 
they can also be used for inference on population regression parameters when the data follows a population linear regression model. The following corollary illustrates this. 
See \Cref{pfrndls} for the proof.
\begin{corollary}[Inference for regression parameters in a linear model]\label{rndls}
Assume the data follows a standard linear model with true parameter $\beta^*$, that is, $y_n=X_n\beta^*+\ep_n$, where the error vector $\ep_n$
has i.i.d.~components with mean zero and variance $\sigma_o^{2}$, and $X_n$ can be either random or deterministic. 
Under the conditions of \Cref{lemgeneols}, 
let Condition \ref{generalquadratic forms} \textup{\texttt{1-dim'}} and the Lindeberg-Feller condition as stated in Example 2.28 of \cite{van1998asymptotic} hold. 
If in addition $m/(n\tau_{m,n})\to 0$, we have 
\beq\label{rndls1}
m^{1/2}\tau_{m,n}^{-1/2}\|\tep_{m,n}\|^{-1}(\tX_{m,n}^\top \tX_{m,n})^{1/2}(\hbs-\beta^*) \Rightarrow \N(0,I_p).
\eeq
If instead  $m/(\tau_{m,n}\|X_n\beta^*\|^2)\to_P 0$, then 
\beq\label{rndls2} m^{1/2}\tau_{m,n}^{-1/2}\left(
\|\tX_{m,n} 
\hbp\|^2 \cdot (\tX_{m,n}^\top \tX_{m,n})^{-1}+ (\alpha+1)\hbp \hbpt\right)^{-1/2}(\hbp-\beta^*) \Rightarrow \N(0,I_p).\eeq

\end{corollary}

For various types of sketching methods discussed in this paper, $m/(n\tau_{m,n})\to 0$
reduces to $m/n\to 0$. 
Further, $m/(\tau_{m,n}\|X_n\beta^*\|^2)\to_P 0$
has a similar interpretation if $\|X_n\beta^*\|^2=\Theta(n)$, which is a genericity 
condition that holds if e.g., the matrix $X_n$ has 
i.i.d.~entries with zero mean and unit variance, and if $\|\beta^*\|^2 =\Theta(1)$.
Therefore, given data sampled i.i.d.~from a population, if the sketched sample size 
grows slower than the original sample size, 
we can directly use the sketched solution for inference about the population regression coefficient.
However, we emphasize that our general theory is data-conditional and is thus stronger and more broadly applicable, as it does not assume any specific statistical model generating the data.

{\begin{table}[ht]
        \renewcommand{\arraystretch}{1.2}
        \centering
        \caption{
       Summary of some of our results for sketching in LS.
      For various distributions of sketching matrices, we summarize the choices of 
      $\tau_{m,n}$,
      $\sigma_n^2(a_n, \ta_n)$, $G_n$,
      and Condition $\mathfrak{P}$ used
      in Condition \ref{generalquadratic forms} for which we show 
       Theorem \ref{lemgeneols}. 
       Here $\gamma_n = m/n$,
       and as defined in Definition \ref{spmx}, 
       $Q_p=\vec(I_p)\vec(I_p)^\top$,
while $P_p$ is the  $p^2\times p^2$ 
matrix representing matrix transposition in the vectorized matrix space.
Also,
$u_{n,i}$, $i\in[n]$ are the rows of the $n\times p$ left singular matrix $U_n$ of $X_n$.
       See \Cref{spec} for the details.}
        \label{sumqf}
        \centering
        \begin{tabular}{|c|c|c|c|c|}
        \hline
       Sketching Distribution& $\tau_{m,n}$ & $\sigma_n^2(a_n, \ta_n)$ & $G_n$ & Condition $\mathfrak{P}$\\
        \hline 
        \hline
         Hadamard &  $1-\gamma_n$ & $1+2(a_n^\top \ta_n)^2$ & $I_{p^2}+P_p+Q_p$& $\ell_\infty$-delocalization \\
         \hline 
         CountSketch/SSE & $1$ & $1+(a_n^\top \ta_n)^2$ & $I_{p^2}+P_p $ & $\ell_4$-delocalization \\
             \hline
       i.i.d.,~$\kappa_{n,4}-3=o(1)$ & $1$ & $1+(a_n^\top \ta_n)^2$ & $I_{p^2}+P_p $ & \\
       \hline
         Haar &  $1-\gamma_n$ & $1+(a_n^\top \ta_n)^2$  & $I_{p^2}+P_p $ & \\
         \hline
         Uniform Subsampling &  $1-\gamma_n$ & $n\sum_{i=1}^n a_{n,i}^2\ta_{n,i}^2$  & $n\sum_{i=1}^n \left((u_{n,i}u_{n,i}^\top)\otimes(u_{n,i}u_{n,i}^\top)\right)$ & See \Cref{unif} \\
         \hline
        \end{tabular}\label{sumasyols} 
    \end{table}}

\subsubsection{Sketched PCA}
\label{spca}

We now turn to sketch-and-solve principal component analysis.
For a $p\times p$ symmetric matrix $\Xi$  and each $i\in[p]$, 
we denote by $\lambda_i(\Xi)$  the $i$-th largest eigenvalue of $\Xi$,
and by $v_i(\Xi)$ an associated eigenvector
such that the first non-zero coordinate of each vector is positive.
We can define $\lambda_i$ and $v_i$ 
as continuous functions on $p\times p $ symmetric matrices,
and if $\Xi$ has distinct eigenvalues,
we can also define them differentiably
in a small neighborhood of $\Xi$,
see e.g., \cite{magnus2019matrix}. 
In this context, we also define $\Lambda_{n,i}=\lambda_i(X_n^\top X_n)$ and $\hat{\Lambda}_{m,n,i}=\lambda_i(X_n^\top S_{m,n}^\top S_{m,n} X_n)$ for $i\in[p]$. 
Additionally, $v_{n,i}$ is an eigenvector of $X_n^\top X_n$ associated with the eigenvalue $\Lambda_{n,i}$, while $\hat{v}_{m,n,i}$ 
is an eigenvector of $X_n^\top S_{m,n}^\top S_{m,n} X_n$ associated with the eigenvalue $\hat{\Lambda}_{m,n,i}$.

For each $i\in[p]$, 
we
define the function
$\Delta_i$ below,
used to characterize the asymptotic covariance structure of the $i$-th eigenvector,
defined for $\Xi\in\Sp$ with distinct eigenvalues and an absolutely symmetric $G$ as per Definition \ref{abs} as
\begin{equation}\label{defHi}
\Delta_i(\Xi, G)
:=
\sum_{k\neq i, l\neq i} \frac{\lambda_i(\Xi)\sqrt{\lambda_k(\Xi)\lambda_l(\Xi)}}{\left(\lambda_i(\Xi)-\lambda_k(\Xi)\right)\left(\lambda_i(\Xi)-\lambda_l(\Xi)\right)}G_{(ik),(il)}v_k(\Xi)v_l(\Xi)^\top.
\end{equation}
For $\Xi$ with repeated eigenvalues, we 
formally define $\Delta_i(\Xi, G) = \infty \cdot I_p$,
so that for all nonzero $c\in \R^p$,
$(c^\top \Delta_i(\Xi, G)c)^{-1/2}=0$. 
Moreover, 
we define 
$\Delta_{n,i}:=\Delta_i(X_n^\top X_n, G_n)$ 
and the plug-in estimator 
$\hat{\Delta}_{m,n,i}:=\Delta_i(X_n^\top S_{m,n}^\top S_{m,n} X_n, G_n)$. 
Our main result for the sketched PCA is the following. See \Cref{pfinfPCAgeneral} for the proof.
\begin{theorem}[Inference in Sketched PCA]\label{infPCAgeneral}
Let $(X_n)_{n\ge 1}$ satisfy Condition \ref{spectralcondition} \textup{\texttt{B}},
$(U_n)_{n\ge 1}$ satisfy Condition $\mathfrak{P}'$, and $(S_{m,n})_{m,n\ge 1}$ 
satisfy Condition \ref{generalquadratic forms} \textup{\texttt{Sym}} for some  $(G_n)_{n\ge 1}$.  
Then
for any $i\in[p]$,
we have
the following inferential result for the unobserved data eigenvalues $\Lambda_{n,i}=\lambda_i(X_n^\top X_n)$
based on the observed sketched eigenvalues $\hat{\Lambda}_{m,n,i}=\lambda_i(X_n^\top S_{m,n}^\top S_{m,n} X_n)$:
\begin{equation}\label{infsingval}
\sqrt{\frac{m}{\tau_{m,n}(G_n)_{(ii),(ii)}}}\hat{\Lambda}_{m,n,i}^{-1}\left(\hat{\Lambda}_{m,n,i}-\Lambda_{n,i}\right)\Rightarrow \N(0,1).
\end{equation}
Moreover,
for
$\Delta_i$ from \eqref{defHi}
and $\Delta_{n,i}:=\Delta_i(X_n^\top X_n, G_n)$,
for any vector $c\in\psp$ 
with 
$\liminf_{n\to\infty}c^\top \Delta_{n,i}c>0$, we have
the following  inferential result for the unobserved 
linear combinations $c^\top v_{n,i}$
of the eigenvector $v_{n,i}$ of $X_n^\top X_n$ associated with the eigenvalue $\Lambda_{n,i}$, 
based on the observed  eigenvector $\hat{v}_{m,n,i}$ of $X_n^\top S_{m,n}^\top S_{m,n} X_n$ associated with the eigenvalue $\hat{\Lambda}_{m,n,i}$ and the plug-in estimator 
$\hat{\Delta}_{m,n,i}:=\Delta_i(X_n^\top S_{m,n}^\top S_{m,n} X_n, G_n)$: 
\begin{equation}\label{infsingvec}
\sqrt{m}\tau_{m,n}^{-1/2}(c^\top\hat{\Delta}_{m,n,i}c)^{-1/2}c^\top\left(\hat{v}_{m,n,i}-v_{n,i}\right)\Rightarrow \N(0,1).
\end{equation}

In particular, under Condition \ref{generalquadratic forms} \textup{\texttt{Sym-1'}}, we have
\begin{equation}\label{infsingval'}
\sqrt{\frac{m}{(2+\alpha)\tau_{m,n}}}\hat{\Lambda}_{m,n,i}^{-1}\left(\hat{\Lambda}_{m,n,i}-\Lambda_{n,i}\right)\Rightarrow \N(0,1),
\end{equation}
and for any vector $c\in\psp$ satisfying $\limsup_{n\to\infty}(c^\top v_{n,i})^2<1$,
\begin{equation}\label{infsingvec'}
\sqrt{\frac{m}{\tau_{m,n}}}\left(\sum_{k\neq i}\frac{\hat{\Lambda}_{m,n,i}\hat{\Lambda}_{m,n,k}}{(\hat{\Lambda}_{m,n,i}-\hat{\Lambda}_{m,n,k})^2} (c^\top\hat{v}_{m,n,k})^2\right)^{-1/2}c^\top\left(\hat{v}_{m,n,i}-v_{n,i}\right)\Rightarrow \N(0,1).
\end{equation}

\end{theorem}
For inference on $k$-dimensional functionals
$C^\top v_{n,i}$, where $C\in\R^{p\times k}$ satisfies $\liminf_n\lambda_{\min}(C^\top \Delta_{n,i}C)>0$,
we have an almost identical argument that
\begin{equation*}
\sqrt{m}\tau_{m,n}^{-1/2}(C^\top\hat{\Delta}_{m,n,i}C)^{-1/2}C^\top\left(\hat{v}_{m,n,i}-v_{n,i}\right)\Rightarrow \N(0,I_k).
\end{equation*}

By leveraging a subsequence argument,
we avoid assuming convergence of the data, 
expanding the applicability of our approach compared to e.g., \cite{lopes2020error}.
Further, Table \ref{sumasyPCA} 
summarizes some of the results we obtain in \Cref{spec}
for the limiting distributions for sketching in PCA by applying \Cref{infPCAgeneral}.

{\begin{table}[ht]
        \renewcommand{\arraystretch}{1.2}
        \centering
        \caption{
     Summary of some of our results for sketching in PCA.
      For various distributions of sketching matrices, 
      and for any $i\in[p]$, the table shows the asymptotic variances $\sigma^2(T_{\lambda})$ of $T_{\lambda}:=\sqrt{m}\hat{\Lambda}_{m,n,i}^{-1}(\hat{\Lambda}_{m,n,i}-\Lambda_{n,i})$ and 
      $\sigma^2(T_{c^\top v})$ of
      $T_{c^\top v}:=\sqrt{m}\left(\sum_{k\neq i}\hat{\Lambda}_{m,n,i}\hat{\Lambda}_{m,n,k}/(\hat{\Lambda}_{m,n,i}-\hat{\Lambda}_{m,n,k})^2 (c^\top\hat{v}_{m,n,k})^2\right)^{-1/2}c^\top(\hat{v}_{m,n,i}-v_{n,i})$.
 We also show
      Condition $\mathfrak{P}$ used
      in Condition \ref{generalquadratic forms} for which we prove
       Theorem \ref{infPCAgeneral}. 
       We recall that $\gamma_n = m/n$.
       See \Cref{spec} for the details.
        }
        \label{sumasyPCA}
        \centering
        \begin{tabular}{|c|c|c|c|}
        \hline
      Sketching Distribution & $\sigma^2(T_{\lambda})$ & $\sigma^2(T_{c^\top v})$ &  Condition $\mathfrak{P}$\\
       \hline
        \hline 
         Hadamard &   $3(1-\gamma_n) $ & $1-\gamma_n $ & $\ell_\infty$-delocalization\\
         \hline
         CountSketch/SSE & $2$ & $1$ & $\ell_4$-delocalization\\
             \hline
       i.i.d.,~$\kappa_{n,4}-3=o(1)$&  $2$ & $1$ & \\
       \hline
         Haar &  $2(1-\gamma_n)$ & $1-\gamma_n $ &  \\
         \hline
         Uniform subsampling & $(n-m)\sum_{k=1}^n(U_n)_{ki}^4$  & $(1-\gamma_n) c^\top\hat{\Delta}_{m,n,i}c$ & See \Cref{unif} \\
         \hline
        \end{tabular}
    \end{table}}

For completeness, we also provide a result for inference when the data is assumed to be an i.i.d.~sample from a statistical model; see \Cref{pfrndpca} for the proof.

\begin{corollary}[Inference for eigenvalues of the covariance matrix given an i.i.d.~sample]\label{rndpca}
Let the rows of $X_n$ be sampled i.i.d.~from a $p$-dimensional distribution having finite
fourth cumulant with a covariance matrix
$\Sigma\in\Sp$. Denote $\lambda_i^*:=\lambda_i(\Sigma)$ for $i\in[p]$. 
Under the conditions of \Cref{infPCAgeneral}, 
 if in addition $m\cdot\left(n\tau_{m,n}(G_n)_{(ii),(ii)}\right)^{-1} $ $\to 0$, we have for any $i\in[p]$ that
\begin{equation}\label{ri}
\sqrt{\frac{m}{\tau_{m,n}(G_n)_{(ii),(ii)}}}\hat{\Lambda}_{m,n,i}^{-1}\left(\hat{\Lambda}_{m,n,i}-\lambda_{i}^*\right)\Rightarrow \N(0,1).
\end{equation}
\end{corollary}

Similar to the inference in a linear model as in Corollary \ref{rndls}, 
$m\cdot\left(n\tau_{m,n}(G_n)_{(ii),(ii)}\right)^{-1}\to 0$ usually reduces to $m/n\to 0$. 
Thus, 
we can use sketching for inference about the population eigenvalues if $m\ll n$.
Results for eigenvectors are similar and thus omitted.

\section{Inference for Specific Sketching Distributions}
\label{spec}

We now derive inference results 
in sketched LS and PCA
for a variety of sketching distributions.
We start with the 
Subsampled Randomized Hadamard Transform or SRHT \citep{ailon2006approximate},
which is computationally efficient and applies under broad conditions on the data, while being nontrivial to analyze
(\Cref{srht}).
We also provide results for the CountSketch \citep{charikar2002finding}
and Sparse Sign Embeddings 
\citep{achlioptas2001database}, which have similar advantages (\Cref{cs-sse}).
We further analyze 
sketching matrices with i.i.d.~entries (\Cref{iid}),
uniform orthogonal (Haar) sketching (\Cref{haar}),
and uniform sampling (\Cref{unif}).
In these sections, we present the results for PCA first, as they are simpler to state.

We let $\gamma_n = m/n$
and for $i\in[p]$, we denote the $i$-th row of $U_n$ by $U_{i:}$.

\subsection{Subsampled Randomized Hadamard Transform}
\label{srht}
Hadamard matrices $H_t$ 
are defined for non-negative integers $t$
inductively, starting with $H_0=1$ and letting 
for all $t\ge 0$, 
$$H_{t+1}  = \frac{1}{\sqrt{2}}\begin{pmatrix}
    H_t & H_t\\
  H_t & -H_t
\end{pmatrix}.$$
When $n = 2^l$  is a power of two, for some non-negative integer $l$, an $m\times n$, $m\le n$, subsampled 
randomized Hadamard transform
(SRHT) sketching matrix 
\citep{ailon2006approximate}
is defined 
as $S_{m,n} = \sqrt{n/m}B H_l D$, where $B$ is an $n\times n$ diagonal matrix with
i.i.d.~Bernoulli random variables with a success probability of $m/n$. 
Furthermore, $D$ is an $n\times n$ diagonal matrix with i.i.d.~$\pm 1$ Rademacher random variables on the diagonal.
For other $n$, one can use the SRHT for the smallest power of two greater than $n$, and by padding the data with zeroes, see \Cref{tr}. 
Thus, 
we can assume without loss of generality that $n$ is a power of two because 
this does not affect the validity of the conditions specified below.
Matrix-vector multiplication via 
the SRHT can be efficiently performed in $O(n\log n)$ floating-point operations, leveraging the Fast Fourier Transform \citep{ailon2006approximate}. 

Our results rely on the following
conditions, where $a_n, \ta_n\in\nsp$ and $U_n$ is the $n\times p$ left singular matrix of $X_n$
 for all $n\ge 1$: 
\begin{align}
\label{assmaxo1} \textbf{Condition } \mathfrak{P} &: \|a_n\|_\infty \to 0, \text{ or }\|\ta_n\|_\infty\to 0; \\
\label{l_inftydelocalize} \textbf{Condition } \mathfrak{P'} &:\lim_{n\rightarrow \infty} \max_{i=1,\ldots, n} \|U_{i:}\| = 0.
\end{align}
Condition $\mathfrak{P}'$ requires that all \emph{leverage scores} tend to zero; which can be interpreted as there being no $X_n$-outliers. 
Condition $\mathfrak{P}$ is used
implicitly in our inference results, and explicitly in the supporting results (Lemmas \ref{lemboundqfsrht} and \ref{lemgeneralqfsrht}).
Moreover, Condition $\mathfrak{P}'$ is used
explicitly in our inference results.

{\bf PCA.}
We have the following inferential result for the unobserved data eigenvalues
as well as eigenvectors
using the observed sketched eigenvalues $\hat{\Lambda}_{m,n,i}=\lambda_i(X_n^\top S_{m,n}^\top S_{m,n} X_n)$ and eigenvectors $\hat{v}_{m,n,i}$ $=$ $v_i(X_n^\top S_{m,n}^\top S_{m,n} X_n)$, $i\in [p]$.

\begin{theorem}[Inference for Eigenvalues and Eigenvectors with SRHT Sketching]\label{PCAsrht}
For all $(X_n)_{n\ge  1}$
with SVD $U_nL_n V_n^\top$
satisfying Condition \ref{spectralcondition} \textup{\texttt{B}}, 
suppose $(U_n)_{n\ge 1}$ satisfies Condition $\mathfrak{P}'$ from \eqref{l_inftydelocalize}.
Then as $m,n\to \infty$ with $p$ fixed, 
such that $m/\log^2 n\to\infty$ and 
$\limsup \gamma_n<1$, 
for any $i\in[p]$,
we have for the unobserved data eigenvalues $\Lambda_{n,i}=\lambda_i(X_n^\top X_n)$
and the observed sketched eigenvalues $\hat{\Lambda}_{m,n,i}=\lambda_i(X_n^\top S_{m,n}^\top S_{m,n} X_n)$ that \eqref{infPCAvalsrht} holds.
Further, 
we have for the unobserved 
linear combinations $c^\top v_{n,i}$
of the eigenvector $v_{n,i}$ of $X_n^\top X_n$ associated with the eigenvalue $\Lambda_{n,i}$, 
and the observed  eigenvector $\hat{v}_{m,n,i}$ of $X_n^\top S_{m,n}^\top S_{m,n} X_n$ associated with the eigenvalue $\hat{\Lambda}_{m,n,i}$
that 
for any vector $c\in\psp$ with $\limsup_{n\to\infty}$ $(c^\top v_{n,i})^2<1$:

\begin{equation}\label{infPCAvecsrht}
\sqrt{\frac{m}{1-\gamma_n}}\left(\sum_{k\neq i}\frac{\hat{\Lambda}_{m,n,i}\hat{\Lambda}_{m,n,k}}{(\hat{\Lambda}_{m,n,i}-\hat{\Lambda}_{m,n,k})^2} (c^\top\hat{v}_{m,n,k})^2\right)^{-1/2}c^\top(\hat{v}_{m,n,i}-v_{n,i})\Rightarrow \N(0,1).
\end{equation}
\end{theorem}

{\bf LS.}\label{sechada}
 The following result shows 
 how to make inferences on least squares parameters based on sketched LS estimators when using SRHT. 
\begin{theorem}[Inference in Least Squares with SRHT Sketching]\label{thhadamard}
Let  $(S_{m,n})_{m,n\ge 1}$ be a sequence of subsampled randomized Hadamard sketching matrices, and let $X_n$ satisfy Condition \ref{spectralcondition} \textup{\texttt{A}}. 
Denoting for $i\in[p]$ the $i$-th row of $U_n$ by $U_{i:}$,
if either
\begin{equation}\label{assudelocalized}
     (U_n)_{n\ge 1} \textnormal{ satisfies Condition } \mathfrak{P'} \textnormal{ from \eqref{l_inftydelocalize} }
     \quad \text{or}
    \quad \lim_{n\rightarrow \infty} \max_{i=1,\ldots, n} \frac{|\ep_{n,i}|}{\|\ep_n\|} = 0,
\end{equation}
then, as $m\to \infty$, $m/\log^2 n\to\infty$, and for 
$\gamma_n = m/n$, $\limsup \gamma_n<1$, we have 
for the sketch-and-solve LS estimator $\hbs$ from \eqref{ss} that
\begin{equation*}
\left(m/(1-\gamma_n)\right)^{1/2}\|\tep_{m,n}\|^{-1}(\tX_{m,n}^\top \tX_{m,n})^{1/2}(\hbs-\beta_n) \Rightarrow \N(0,I_p).\end{equation*} 
If either
\begin{equation}\label{assudelocalized1}
    (U_n)_{n\ge 1} \textnormal{ satisfies Condition } \mathfrak{P'} \textnormal{ from \eqref{l_inftydelocalize} }
         \quad \text{or} 
    \quad \lim_{n\to\infty}\max_{i=1,\ldots, n} \frac{|y_{n,i} - \ep_{n,i}|}{\|y_n - \ep_n\|} = 0,
\end{equation}
then, in the same asymptotic regime, we have  
for the partial sketching estimator $\hbp$ from \eqref{ss} that
$$\left(m/(1-\gamma_n)\right)^{1/2}\left(\|\tX_{m,n}\hbp\|^2 \cdot (\tX_{m,n}^\top \tX_{m,n})^{-1}+ 2\hbp \hbpt\right)^{-1/2}(\hbp-\beta_n) \Rightarrow \N(0,I_p).$$

\end{theorem}

The second condition in \eqref{assudelocalized} ensures that the normalized residuals are small, which can be viewed as there being no outliers in the outcome vector $y_n$.

Our result for the sketch-and-solve estimator 
$\hbs$
is related to Theorem 3 (i) of \cite{ahfock2021statistical}. 
They obtain the
asymptotic distribution of the sketched data $S_{m,n}X_n$ for a finite $m$. 
In contrast, we focus on the case where the sketch size $m$ tends to infinity, 
which leads to the additional $1-\gamma_n$ factor in the variance. 
In contrast to the work by \cite{ahfock2021statistical}, 
 our approach allows for $X_n$-outliers, as long as there are no $y_n$-outliers. 
 Moreover, we only require a bounded condition number for $X_n$, while their results 
 require certain convergence conditions on $X_n$ and $y_n$.

{\bf CLT.}
The following lemmas play a fundamental role in the proofs of the above theorems by specifying the boundedness and asymptotic normality of quadratic forms in Condition \ref{generalquadratic forms} for SRHT sketching.
We show that normality holds if Condition $\mathfrak{P}$ is satisfied, but does not need to hold otherwise.

\begin{lemma}[Bounded Quadratic Forms for SRHT Sketching]\label{lemboundqfsrht}
For any sequences of vectors $(a_n, \ta_n)_{n\ge 1}$, 
such that $a_n,\ta_n \in\nsp$ for all $n\ge 1$,
we have
$$
\sqrt{\frac{m}{1-\gamma_n}}\left(a_n^\top S_{m,n}^\top S_{m,n} \ta_n - a_n^\top \ta_n\right)=O_P(1).
$$
\end{lemma}

The next lemmas establish the normality
of quadratic forms arising in SRHT. 
In the proof, 
finding the asymptotic variances requires intricate and novel calculations, 
using a dyadic representation for the Hadamard matrix elements.

\begin{lemma}[Limiting Distributions of One-dimensional Quadratic Forms for SRHT Sketching]\label{lemgeneralqfsrht}
For any sequences of vectors $(a_{n})_{n\ge 1}$, $(\ta_{n})_{n\ge 1}$  satisfying Condition $\mathfrak{P}$ defined in \eqref{assmaxo1} such that $a_n,\ta_n \in\nsp$ for all $n\ge 1$,
if $m/\log^2 n\to\infty$ and $\limsup m/n <1$, as $m,n\to \infty$, we have 
\begin{equation}\label{hadafixnormal}\sqrt{\frac{m}{1-\gamma_n}} \sqrt{\frac{1}{1+2 (a_n^\top \ta_n)^2}}\left(a_n^\top S_{m,n}^\top S_{m,n}\ta_n - a_n^\top \ta_n\right) \Rightarrow \N(0, 1).\end{equation}
Moreover, 
if Condition $\mathfrak{P}$ fails, 
$\sqrt{m} a_n^\top S_{m,n}^\top S_{m,n} \ta_n$ does not need to be asymptotically normal even if $m/n$ converges.
\end{lemma}
\begin{lemma}[Asymptotic Normality of Symmetric Quadratic Form for SRHT Sketching]\label{lemsymqfsrht}
If $(U_n)_{n\ge 1}$ satisfies Condition $\mathfrak{P}'$,  
and $m/\log^2 n\to\infty$ with $\limsup m/n <1$ as $m,n\to \infty$, then
for the $p^2\times p^2$ rank-one-positive definite matrix
 $G=I_{p^2}+P_p+Q_p$,
$$\sqrt{\frac{m}{1-\gamma_n}} \vec(U_n^\top S_{m,n}^\top S_{m,n}U_n-I_p)\Rightarrow \N\left(0,G\right).$$
\end{lemma}

\subsection{Sparse Sign Embeddings and CountSketch}
\label{cs-sse}

Sparse Sign Embeddings are a type of sparse dimension reduction maps 
\citep[e.g.,][etc]{achlioptas2001database,meng2013low}. 
For a given sparsity $\zeta_m>0$,
a sparse sign embedding is a random matrix of the form $S_{m,n}=(R_{m,1},\ldots,R_{m,n})$, where for all $i\in [n]$, 
$R_{m,i}$ are i.i.d.~uniformly distributed on the set of $n$-vectors with $\zeta_m$ nonzero elements equal to $\pm1/\sqrt{\zeta_m}$, i.e.,
$\{f/\sqrt{\zeta_m}\in\R^m:f_1,\ldots,f_m\in\{\pm1,0\}\text{ and }\|e\|^2=\zeta_m\}$.
Due to their tunable sparsity, sparse sign embeddings allow the user to control the computational cost of sketching, see \Cref{comp}.
Recently, 
they have  been suggested as a ``sensible default" sketching method \citep{epperly2023which}.
For $\zeta_m=1$, 
sparse sign embeddings are also known as 
CountSketch \citep{charikar2002finding}, or
Clarkson-Woodruff sketches \citep{clarkson2017low}.

Our results rely on the following
conditions, where $a_n, \ta_n\in\nsp$ and $U_n$ is the $n\times p$ left singular matrix of $X_n$
 for all $n\ge 1$: 
\begin{align}
\label{csstate1} \textbf{Condition } \mathfrak{P} &: (\sqrt{m}/\zeta_m)\cdot\sum_{i=1}^n a_{n,i}^4\rightarrow 0, \text{ and }(\sqrt{m}/\zeta_m)\cdot\sum_{i=1}^n \ta_{n,i}^4\rightarrow 0; \\
\label{csstate2} \textbf{Condition } \mathfrak{P'} &: (\sqrt{m}/\zeta_m)\cdot\|\vec(U_n)\|^4_4 \rightarrow 0.
\end{align}
As before, these will be used explicitly or implicitly in our results.
Further, in this section, we use $\sum_{i\neq j}$ to abbreviate the summation over $(i,j)\in[n]\times[n]$ subject to $i\neq j$. 

{\bf PCA.}
We have the following result on inference for eigenvalues and eigenvectors with sparse sign embeddings.
This result allows the per-column sparsity $\zeta_m$  to be $o(\sqrt{m})$; which in particular includes the 
CountSketch.

\begin{theorem}[Inference for Eigenvalues and Eigenvectors with Sparse Sign Embeddings]\label{PCAcs}
For $(X_n)_{n\ge  1}$ satisfying Condition \ref{spectralcondition} \textup{\texttt{B}}, suppose $(U_n)_{n\ge 1}$ satisfies Condition $\mathfrak{P}'$, let $m, n\to\infty$, and suppose the sparsity satisfies  $\zeta_m^2/m\to 0$. 
Then we have for the unobserved data eigenvalues $\Lambda_{n,i}=\lambda_i(X_n^\top X_n)$
and the observed sketched eigenvalues $\hat{\Lambda}_{m,n,i}=\lambda_i(X_n^\top S_{m,n}^\top S_{m,n} X_n)$ that
\begin{equation*}
\sqrt{\frac{m}{2}}\hat{\Lambda}_{m,n,i}^{-1}\left(\hat{\Lambda}_{m,n,i}-\Lambda_{n,i}\right)\Rightarrow \N(0,1).
\end{equation*}
Further,  
we have for the unobserved 
linear combinations $c^\top v_{n,i}$
of the eigenvector $v_{n,i}$ of $X_n^\top X_n$ associated with the eigenvalue $\Lambda_{n,i}$, 
and the observed  eigenvector $\hat{v}_{m,n,i}$ of $X_n^\top S_{m,n}^\top S_{m,n} X_n$ associated with the eigenvalue $\hat{\Lambda}_{m,n,i}$
that 
for any vector $c\in\psp$ satisfying $\limsup_{n\to\infty}(c^\top v_{n,i})^2<1$,
\begin{equation*}
\sqrt{m}\left(\sum_{k\neq i}\frac{\hat{\Lambda}_{m,n,i}\hat{\Lambda}_{m,n,k}}{(\hat{\Lambda}_{m,n,i}-\hat{\Lambda}_{m,n,k})^2} (c^\top\hat{v}_{m,n,k})^2\right)^{-1/2}c^\top(\hat{v}_{m,n,i}-v_{n,i})\Rightarrow \N(0,1).
\end{equation*}
\end{theorem}

{\bf LS.}\label{seccs}
We also have the following 
result for inference about the least squares parameters based on sketch-and-solve with sparse sign embeddings.

\begin{theorem}[Inference in Least Squares with Sparse Sign Embeddings]\label{cs}
Suppose $X_n$ satisfies Condition \ref{spectralcondition} \textup{\texttt{A}}, $m, n\to\infty$, and $\zeta_m^2/m\to 0$. If 
\beq\label{olscsdelocal1}
(U_n)_{n\ge 1} \textnormal{ satisfies Condition } \mathfrak{P'} \textnormal{ from \eqref{csstate2}}
\quad\text{ and}\quad
(\sqrt{m}/\zeta_m)\cdot \frac{\|\epsilon_n\|^4_4}{\|\epsilon_n\|^4}\to 0,
\eeq
then 
for the sketch-and-solve LS estimator $\hbs$ from \eqref{ss}, we have with $ \hat{\Sigma}_{m,n} = \|\tep_{m,n}\|^2\cdot  (\tX_{m,n}^\top \tX_{m,n})^{-1}$ that
$\sqrt{m}\hat{\Sigma}_{m,n}^{-1/2}(\hbs-\beta_n) \Rightarrow \N(0,I_p)$.
Similarly, 
if 
\beq\label{olscsdelocal2}
(U_n)_{n\ge 1} \textnormal{ satisfies Condition } \mathfrak{P'} \textnormal{ from \eqref{csstate2}}\quad\text{ and}\quad
(\sqrt{m}/\zeta_m)\cdot \frac{\|y_n-\epsilon_n\|^4_4}{\|y_n-\epsilon_n\|^4}\to 0,
\eeq
then
for the sketch-and-solve LS estimator $\hbs$ from \eqref{ss} 
we have 
with $\hSigp=\hbeta_{m,n}^{\pa\top}\tX_{m,n}^\top \tX_{m,n}  \hbp (\tX_{m,n}^\top \tX_{m,n})^{-1}+ \hbp \hbpt$
that
$\sqrt{m}(\hSigp)^{-1/2}(\hbp-\beta_n) \Rightarrow \N(0,I_p)$.
\end{theorem}

{\bf CLT.}
The inference results mentioned above are based on the following central limit theorems for
quadratic forms
from 
Condition \ref{generalquadratic forms} for Sparse Sign Embeddings.
First, we have a result on one-dimensional quadratic forms.

\begin{lemma}[Limiting Distributions of One-dimensional Quadratic Forms for Sparse Sign Embeddings]\label{lemgeneralqfcs}
For
a fixed $n$, and 
any sequences $(a_{n})_{n\ge 1}$, $(\ta_{n})_{n\ge 1}$ such that $a_n,\ta_n \in\nsp$ for all $n\ge 1$, 
if
Condition $\mathfrak{P}$ from \eqref{csstate1} holds
as $m\to \infty$ and $\zeta_m^2/m\to 0$, 
then we have
\begin{equation}\label{csclt} \sqrt{\frac{m}{1+ (a_n^\top \ta_n)^2}}\left(a_n^\top S_{m,n}^\top S_{m,n}\ta_n - a_n^\top \ta_n\right) \Rightarrow \N(0, 1).\end{equation}
\end{lemma}

This lemma represents a novel 
result about CountSketch and Sparse Sign Embeddings. 
Our proof treats the quadratic forms as degenerate U-statistics and uses the martingale central limit theorem.
We next have a result on vectorized symmetric quadratic forms.

\begin{lemma}[Asymptotic Normality of Symmetric Quadratic Form for Sparse Signed Embeddings]\label{lemsymqfcs}
If as $m,n\to\infty$,
$(U_n)_{n\ge 1}$ satisfies Condition $\mathfrak{P}'$ 
from \eqref{csstate2}
while $\zeta_m^2/m\to 0$, then  
\beq\label{cssqf}
\sqrt{m}\vec(U_n^\top S_{m,n}^\top S_{m,n}U_n-I_p)\Rightarrow \N\left(0,I_{p^2}+P_p\right).
\eeq
\end{lemma}

\subsection{Uniform Random Sampling}
\label{unif}
We assume $S_{m,n}=\sqrt{\frac{n}{m}}\bar B_n$, where $\bar B_n=\diag(B_1,\ldots,B_n)$ with i.i.d.~$B_i\sim\Ber(\gamma_n)$,
$\gamma_n=m/n$ for all $i\in[n]$. 
Thus, $\bar B_n$ represents a uniform random subsampling procedure.
Our results rely on the following
conditions, where $a_n, \ta_n\in\nsp$ for all $n\ge 1$, 
$u_{n,i}$, $i\in[n]$ are the rows of the $n\times p$ left singular matrix $U_n$ of $X_n$, 
and
$G_n:=
n\sum_{i=1}^n (u_{n,i}u_{n,i}^\top)\otimes(u_{n,i}u_{n,i}^\top)$: 
\begin{align}
\textbf{Condition } \mathfrak{p}&: \text{for some } 0<c<C,\, 
c\le n\sum_{i=1}^na_{n,i}^2\ta_{n,i}^2
\le C; \ 
\textbf{Condition } \mathfrak{P}: 
\max_{i\in[n]} \frac{(a_{n,i} \ta_{n,i})^2}{\sum_{i=1}^na_{n,i}^2\ta_{n,i}^2}\to 0; \label{assdelocalsubsamp2} \\
\textbf{Condition } \mathfrak{P'}&: 
 \text{for some } 0<c<C,\, 
cI_{p(p+1)/2}\preceq 
D_p^\top G_nD_p\preceq CI_{p(p+1)/2}; \label{assdelocalsubsamp3} 
\\& \text{ for all upper triangular } \Phi\in \R^{p\times p},\,\,  \frac{\max_{i\in[n]}(u_{n,i}^\top \Phi u_{n,i})^2}{\sum_{i=1}^n (u_{n,i}^\top \Phi u_{n,i})^2}\to 0.\nonumber
\end{align}
Again, these results appear explicitly in our asymptotic normality results, and sometimes only implicitly in our inference results.

{\bf PCA.}
If the data $X_n$ has very heterogeneous rows, then 
subsampling 
can introduce a great deal of variability in the resulting estimates.
Therefore, 
obtaining inference and asymptotic normality requires stronger delocalization  conditions  on $U_n$, and 
leads to an asymptotic variance of a delicate form.
\begin{theorem}[Asymptotic Normality and
Inference for Eigenvalues and Eigenvectors with Subsampling]\label{theigss}
Suppose $(X_n)_{n\ge 1}$ satisfies Condition \ref{spectralcondition} \textup{\texttt{B}},
 $(U_n)_{n\ge 1}$ satisfies Condition $\mathfrak{P}'$, and $(S_{m,n})_{m,n\ge 1}$ 
satisfies Condition \ref{generalquadratic forms} \textup{\texttt{Sym}}. 
Then we have for any $i\in [p]$ that
\begin{equation}\label{sse}
\sqrt{\frac{m}{(1-\gamma_n)\cdot n\sum_{k=1}^n(U_n)_{ki}^4}}\hat{\Lambda}_{m,n,i}^{-1}\left(\hat{\Lambda}_{m,n,i}-\Lambda_{n,i}\right)\Rightarrow \N(0,1).
\end{equation}
Let $\tX_{m,n} =  \tU_{m,n}\hat L_{m,n}\hat V_{m,n}^\top $ be the SVD of $\tX_{m,n}$.
If $\liminf \gamma_n >0$, we have
\begin{equation}\label{sse2}
\frac{1}{\sqrt{(1-\gamma_n)\cdot m\sum_{j=1}^m \tilde{U}_{ji}^4}}\hat{\Lambda}_{m,n,i}^{-1}\left(\hat{\Lambda}_{m,n,i}-\Lambda_{n,i}\right)\Rightarrow \N(0,1).
\end{equation}

For $\Delta_{i}$ defined in \eqref{defHi}, and any vector $c\in\psp$ satisfying $\liminf_{n\to\infty}c^\top \Delta_{n,i}c>0$,
\begin{equation*}
\sqrt{m/(1-\gamma_n)}(c^\top\hat{\Delta}_{m,n,i}c)^{-1/2}c^\top\left(\hat{v}_{m,n,i}-v_{n,i}\right)\Rightarrow \N(0,1).
\end{equation*}
\end{theorem}

{\bf LS.}
Additionally, we have the following theorem for inference about the least squares parameters based on uniform subsampling.

\begin{theorem}[Inference in Least Squares with Uniform Random Sampling]\label{thsubsamp}
Let $p$ be fixed and $m,n\to\infty$ with
$0<\liminf m/n \le\limsup m/n <1$.
Suppose that $X_n$ satisfies Condition \ref{spectralcondition} \textup{\texttt{A}}, 
and 
$\|\vec(U_n)\|_\infty=O(1/\sqrt{n})$. 
Also, suppose there exist some constants $0<c<C$ such that 
\begin{equation}\label{strdelocsubsamp1}
cI_p\preceq n U_n^\top \diag\left(\frac{|\ep_{n}|}{\|\ep_n\|}\odot \frac{|\ep_{n}|}{\|\ep_n\|}\right)U_n\preceq CI_p.
\end{equation}
For
\begin{equation*}
\hat{\Sigma}_{m,n} =m(\tX_{m,n}^\top \tX_{m,n})^{-1} \tX_{m,n}^\top \diag(\tep_{m,n} \odot \tep_{m,n}) \tX_{m,n} (\tX_{m,n}^\top \tX_{m,n})^{-1}.\end{equation*}
we have 
$m^{1/2}(1-\gamma_n)^{-1/2}\hat{\Sigma}_{m,n}^{-1/2}(\hbs-\beta_n) \Rightarrow \N(0,I_p)$.
Furthermore, 
assume there exist $0<c<C$ such that
\begin{equation}\label{strdelocsubsamp2}
cI_p\preceq n U_n^\top \diag\left(\frac{|y_n-\ep_{n}|}{\|y_n-\ep_n\|}\odot \frac{|y_n-\ep_{n}|}{\|y_n-\ep_n\|}\right)U_n\preceq CI_p.
\end{equation}
Then with 
\begin{equation*}\hSigp=m(\tX_{m,n}^\top \tX_{m,n})^{-1} \tX_{m,n}^\top \diag\left((\tX_{m,n} \hbp) \odot (\tX_{m,n} \hbp)\right) \tX_{m,n} (\tX_{m,n}^\top \tX_{m,n})^{-1},
\end{equation*}
we have 
$m^{1/2}(1-\gamma_n)^{-1/2}(\hSigp)^{-1/2}(\hbp-\beta_n) \Rightarrow \N(0,I_p)$.
\end{theorem}
In fact, 
\eqref{strdelocsubsamp1} holds if there are constants $0<c<C$ such that
$
c/\sqrt{n}\le \min_{i\in[n]}|\ep_{n,i}|/\|\ep_n\|
$
and 
$
\max_{i\in[n]}|\ep_{n,i}|/\|\ep_n\|$
$\le C/\sqrt{n},
$
and similarly \eqref{strdelocsubsamp2} holds if $
c/\sqrt{n}\le \min_{i\in[n]}|y_{n,i}-\ep_{n,i}|/\|y_n-\ep_n\|$
and 
$\max_{i\in[n]}|y_{n,i}-\ep_{n,i}|/\|y_n-\ep_n\|\le C/\sqrt{n}.
$

{\bf CLT.}
The following lemmas play a significant role in the proofs of the above theorems by establishing the limiting behavior from Condition \ref{generalquadratic forms} for uniform subsampling.
\begin{lemma}[Boundedness and Asymptotic Normality of Quadratic Forms for Uniform Random Sampling]\label{lemboundqfsubsamp}
For any sequences of vectors $(a_n, \ta_n)_{n\ge 1}$
such that $a_n,\ta_n \in\nsp$ for all $n\ge 1$
satisfying Condition $\mathfrak{p}$ from \eqref{assdelocalsubsamp2}, 
we have
\begin{equation}\label{ssbo}
    \sqrt{\frac{m}{1-\gamma_n}}\left(a_n^\top S_{m,n}^\top S_{m,n} \ta_n - a_n^\top \ta_n\right)=O_P(1).
\end{equation}
Further, 
if Condition $\mathfrak{P}$ from \eqref{assdelocalsubsamp2} holds,
then
\begin{equation}\label{ssn} \sqrt{\frac{m}{1-\gamma_n}}\sqrt{\frac{1}{n\sum_{i=1}^n a_{n,i}^2\ta_{n,i}^2}}\left(a_n^\top S_{m,n}^\top S_{m,n}\ta_n - a_n^\top \ta_n\right) \Rightarrow \N(0, 1).\end{equation}
Moreover, there exist sequences $(a_{n})_{n\ge 1}$, $(\ta_{n})_{n\ge 1}$ of unit norm vectors with $a_n,\ta_n \in\R^n$ that do not satisfy 
Condition $\mathfrak{P}$, 
such that 
$\sqrt{m} a_n^\top S_{m,n}^\top S_{m,n} \ta_n$ is not asymptotically normal.
\end{lemma}

\begin{lemma}[Asymptotic Normality of Symmetric Quadratic Form for Uniform Random Sampling]\label{lemsymqfunifrandsamp}
If $U_n$ 
satisfies Condition $\mathfrak{P}'$ from \eqref{assdelocalsubsamp3}, 
then we have
\begin{equation*}
\sqrt{\frac{m}{1-\gamma_n}} (D_p^\top G_n D_p)^{-1/2} D_p^\top \vec\left(U_{n}^\top S_{m,n}^\top S_{m,n}U_{n}-I_p\right)\Rightarrow \mathcal{N}(0,I_{p(p+1)/2}).
\end{equation*}
\end{lemma}

\subsection{Sketching Matrices with i.i.d.~Entries}
\label{iid}

We now 
let  
$S_{m,n}$ be an $m \times n$ random sketching matrix with i.i.d.~entries for all $n\ge 1$. 
Moreover, we denote the entries of the scaled sketching matrix as $\sqrt{m}S_{m,n}=(T_{ij})_{i\in[m],j\in[n]}$, where $\E T_{ij}=0$ and 
$\E T_{ij}^2=1$, so that  $\EE{S_{m,n}}=0$
and
$\EE{S_{m,n}^\top S_{m,n}}=I_n$.
Additionally, for some positive $\delta^\prime$, we assume $\E T_{ij}^4=\kappa_{n,4}>1+\delta^\prime$,
which excludes distributions whose kurtosis is very close to the theoretical minimum of unity, such as the uniform distribution over $\{-1, 1\}$. 
See the discussion after \Cref{thiid} for more details.
Moreover, 
for $\delta\ge 0$, we denote the moments $\kappa_{n,4+\delta}:=\E T_{ij}^{4+\delta}<+\infty$.

Our results rely on the following
conditions, where $a_n, \ta_n\in\nsp$, and
$u_{n,i}$, $i\in[n]$ are the rows of the $n\times p$ left singular matrix $U_n$ of $X_n$: 
\begin{align}
&\textbf{Condition } \mathfrak{P}: \sup_{n\ge 1}\kappa_{n,4+\delta}<+\infty; \text{ Or }  \kappa_{n,4}\to\infty, \kappa_{n,4+\delta}^{4/(4+\delta)}/\kappa_{n,4}=o(m^{\delta/(4+\delta)}), \nonumber\\
&\qquad\qquad\qquad \qquad\qquad\text{ and for some } 0<c<C,  c\le \sum_{i=1}^n (a_{n,i}\ta_{n,i})^2\le C;\label{piid}\\
&\textbf{Condition } \mathfrak{P'}: \sup_{n\ge 1}\kappa_{n,4+\delta}<+\infty; \text{ Or }  \kappa_{n,4}\to\infty, \kappa_{n,4+\delta}^{4/(4+\delta)}/\kappa_{n,4}=o(m^{\delta/(4+\delta)}), \nonumber\\ 
&\qquad\qquad\qquad \text{ and there exists } c>0 \text{ with } cI_{p(p+1)/2}\preceq D_p^\top \sum_{i=1}^n\left((u_{n,i}u_{n,i}^\top)\otimes(u_{n,i}u_{n,i}^\top)\right)D_p. \label{p'iid}
\end{align}
Both conditions can be satisfied \emph{either}
if the $4+\delta$-th moments $\kappa_{n,4+\delta}$ are uniformly bounded, 
or if they grow sufficiently slowly, 
but the data satisfies certain additional regularity conditions.
As before, these conditions are used
implicitly or explicitly in our results.

{\bf PCA.}
\label{pcaiid}
We have the following result for inference in PCA for sketching matrices with i.i.d.~entries.

\begin{theorem}[Inference for Eigenvalues and Eigenvectors with i.i.d.~Sketching]
\label{iid1}
Consider the SVD of $X_n$ given by $U_n L_nV_n^\top$. 
Let the entries of the scaled i.i.d.~sketching matrix $\sqrt{m}S_{m,n}=(T_{ij})_{i,j}$
satisfy $\E T_{ij}=0, \E T_{ij}^2=1$, $\E T_{ij}^4=\kappa_{n,4}>1+\delta^\prime$ for some $\delta^\prime>0$, and $\E T_{ij}^{4+\delta}<+\infty$ for some $\delta>0$.
Suppose that $(X_n)_{n\ge  1}$ satisfies Condition \ref{spectralcondition} \textup{\texttt{B}} and $(U_n)_{n\ge 1}$ satisfies Condition $\mathfrak{P}'$. 
For all $n\ge 1$, define the $p^2\times p^2$ matrix $\Gamma_n$, 
where $(\Gamma_n)_{(k_1k_2),(k_3k_4)}=(\kappa_{n,4}-3)\sum_{h=1}^n (U_n)_{hk_1}(U_n)_{hk_2}(U_n)_{hk_3}(U_n)_{hk_4}$ for every $k_1,k_2,k_3,k_4\in[p]$.
Then 
\begin{equation}\label{iidgeneralval}
\sqrt{\frac{m}{2+(\Gamma_n)_{(ii),(ii)}}}\hat{\Lambda}_{m,n,i}^{-1}\left(\hat{\Lambda}_{m,n,i}-\Lambda_{n,i}\right)\Rightarrow \N(0,1);
\end{equation}
and letting
$$\hat{\Delta}_{m,n,i}=\sum_{k\neq i, l\neq i}\frac{\hat{\Lambda}_{m,n,i}\sqrt{\hat{\Lambda}_{m,n,k}\hat{\Lambda}_{m,n,l}}}{(\hat{\Lambda}_{m,n,i}-\hat{\Lambda}_{m,n,k})(\hat{\Lambda}_{m,n,i}-\hat{\Lambda}_{m,n,l})}\left(\delta_{kl}+(\Gamma_n)_{(ij),(kl)}\right)\hat{v}_{m,n,k}\hat{v}_{m,n,l}^\top,$$
for any vector $c\in\psp$ satisfying $\liminf_{n\to\infty}c^\top \Delta_{n,i}c>0$, with $\Delta_{n,i}=\Delta_i(X_n^\top X_n, G_n)$  from \Cref{spca},
\begin{equation*}\label{iidgeneralvec}
\sqrt{m}(c^\top \hat{\Delta}_{m,n,i}c)^{-1/2}c^\top(\hat{v}_{m,n,i}-v_{n,i})\Rightarrow \N(0,1).
\end{equation*}
\end{theorem}

In \eqref{iidgeneralval}, the excess kurtosis term $\Gamma_n$ vanishes if $\kappa_{n,4}=3$, which holds for Gaussian distributions. 
In that case \eqref{iidgeneralval} depends only on the empirical data.
For $\kappa_{n,4}\neq 3$, \eqref{iidgeneralval} only leads to inference if $(\Gamma_n)_{(ii),(ii)}$ vanishes, which can be viewed as a delocalization condition on the data.
\Cref{iid1}  
is related 
to the results of \cite{anderson1963asymptotic} for Gaussian sketching matrices
and to those of \cite{waternaux1976asymptotic}.
Those 
results obtain the  joint limiting distribution of  all eigenvalues, 
while our results focus on individual eigenvalues but
(1) establish inference methods, 
and (2) broaden the conditions, not requiring the convergence of the data;  
see Section \ref{seciid} for more discussion.

{\bf LS.}
We consider the problem of inference for $\beta_n$ based on 
sketching matrices with i.i.d.~entries. 
The result here is consistent with 
the special case of 
Theorem 3.2 in \cite{zhang2023framework} for fixed $p$,
but our conclusion holds for a broader class of sketching distributions, only assuming that a $4+\delta$-th moment (for any small $\delta>0$) grows sufficiently slowly, whereas they assume all moments are bounded. 
In particular, our work allows sparse i.i.d.~sketches, which requires the $4+\delta$-th moment to grow, whereas \cite{zhang2023framework} does not allow this.

\begin{theorem}[Inference in Least Squares with i.i.d.~Sketching]\label{thiid}
Let $X_n$ satisfy 
Condition \ref{spectralcondition} \textup{\texttt{A}}, let $p$ be fixed and $m,n\to\infty$.
Also, let the entries of $S_{m,n}$ be i.i.d.~copies of $S_{1,1}/m^{1/2}$, where $S_{1,1}$ has zero mean, unit variance and 
kurtosis $\kappa_{n,4} = \E S_{1,1}^4>1+\delta^\prime$ for some $\delta'>0$.
Suppose either that
\begin{compactitem}
    \item [(i)] for some fixed $\delta>0$, $\kappa_{n,4+\delta}=\E |S_{1,1}|^{4+\delta}$ satisfies $\sup_{n\ge 1}\kappa_{n,4+\delta}<+\infty$; or
    \item [(ii)] we have $\kappa_{n,4}\to\infty$, $\kappa_{n,4+\delta}^{4/(4+\delta)}/\kappa_{n,4}=o(m^{\delta/(4+\delta)})$ and also there exist some constants $0<c<C$ such that 
\begin{equation}\label{strdelociid1}
cI_p\preceq U_n^\top \diag\left(\frac{|\ep_{n}|}{\|\ep_n\|}\odot \frac{|\ep_{n}|}{\|\ep_n\|}\right)U_n\preceq CI_p.
\end{equation}
\end{compactitem} 
With
\begin{equation}\label{hsmn}
\hat{\Sigma}_{m,n} =m(\tX_{m,n}^\top \tX_{m,n})^{-1} \tX_{m,n}^\top \diag(\tep_{m,n} \odot \tep_{m,n}) \tX_{m,n} (\tX_{m,n}^\top \tX_{m,n})^{-1},\end{equation}
we have 
for the sketch-and-solve LS estimator $\hbs$ from \eqref{ss} that
$m^{1/2}\hat{\Sigma}_{m,n}^{-1/2}(\hbs-\beta_n) \Rightarrow \N(0,I_p)$.
Furthermore, with 
\begin{equation*}\hSigp=m(\tX_{m,n}^\top \tX_{m,n})^{-1} \tX_{m,n}^\top \diag\left((\tX_{m,n} \hbp) \odot (\tX_{m,n} \hbp)\right) \tX_{m,n} (\tX_{m,n}^\top \tX_{m,n})^{-1},
\end{equation*}
if assuming either (i); or assuming (ii) with \begin{equation*}\label{strdelociid2}
cI_p\preceq U_n^\top \diag\left(\frac{|y_n-\ep_{n}|}{\|y_n-\ep_n\|}\odot \frac{|y_n-\ep_{n}|}{\|y_n-\ep_n\|}\right)U_n\preceq CI_p
\end{equation*} 
instead of \eqref{strdelociid1},
then
we have  
for the partial sketching LS estimator $\hbp$ from \eqref{ss} that
$m^{1/2}(\hSigp)^{-1/2}(\hbp-\beta_n) \Rightarrow \N(0,I_p)$.
\end{theorem}

We know that $\kappa_{n,4} \ge  1$, and the minimum value is achieved for a Radamacher variable.
This minimum is excluded from our analysis, because if $\kappa_{n,4}=1$, 
some projections of $\hbs-\beta_n$
in certain directions may be deterministically zero, allowing us to exactly recover some linear functionals of $\beta_n$.
Thus, one will need a more complex result to describe the behavior of $\hbs-\beta_n$.
This can be inferred from equation \eqref{iidvar} in the proofs by choosing specific $a_n$ and $\ta_n$. 
For example, when $a_n = \ta_n = (1,0,\ldots,0)^\top$ or $a_n = (1/\sqrt{2},1/\sqrt{2},0,\ldots, 0)^\top$ and $\ta_n=(1/\sqrt{2},-1/\sqrt{2},0,\ldots,0)^\top$, 
the right-hand side of \eqref{iidvar} is zero. 
Since sketches whose entries are i.i.d.~Rademacher variables are somewhat rarely used, we leave investigating this phenomenon to future work.

For Gaussian sketching, where $S_{m,n} = m^{-1/2}Z_{m,n}$ and $Z_{m,n}$ contains independent standard Gaussian entries, Theorem 1 of \cite{ahfock2021statistical} 
shows that the distribution of $\hbs$ given $\tX_{m,n}$ is normal, with a known mean and covariance matrix.
Hence, exact finite-sample statistical inference---conditional on $\tX_{m,n}$---can be performed.
Our result corresponds to taking the limit as $m\to\infty$ in Theorem 1, part (ii) of \cite{ahfock2021statistical}, 
for the marginal distribution of $\hbs$, where our $m$ is their $k$ and 
our 
$\sum_{i=1}^n \ep_{n,i}^2 = \|y_n-X_n\beta_n\|^2$ is their $RSS_F$.
An equivalent result for Gaussian sketching 
also appears 
in equation (5) of
\cite{bartan2023distributed}. 


{\bf CLT.}
The previous results are based on the following central limit theorems
for quadratic forms from Condition \ref{generalquadratic forms} for i.i.d.~sketching.

\begin{lemma}[Limiting Distributions of One-dimensional Quadratic Forms for i.i.d.~Sketching]\label{lemgeneralqfiid}
Suppose that the entries of $S_{m,n}$ are i.i.d.~copies of $S_{1,1}/m^{1/2}$, where $S_{1,1}$ has zero mean, unit variance,
kurtosis $\kappa_{n,4} = \E S_{1,1}^4>1+\delta^\prime$ for some $\delta'>0$,
and that for some fixed $\delta>0$, Condition $\mathfrak{P}$ from \eqref{piid} holds. 
Then 
we have 
$$\sqrt{\frac{m}{1+(a_n^\top \ta_n)^2+(\kappa_{n,4}-3)\sum_{i=1}^n (a_{n,i}\ta_{n,i})^2}}\left(a_n^\top S_{m,n}^\top S_{m,n}\ta_n - a_n^\top \ta_n\right) \Rightarrow \N(0, 1).$$
\end{lemma}

\begin{lemma}[Asymptotic Normality of Symmetric Quadratic Form for i.i.d.~Sketching]\label{symiid}
Suppose that the entries of $S_{m,n}$ are i.i.d.~copies of $S_{1,1}/m^{1/2}$, where $S_{1,1}$ has zero mean, unit variance,
kurtosis $\kappa_{n,4} = \E S_{1,1}^4>1+\delta^\prime$ for some $\delta'>0$.
Let $(U_n)_{n\ge 1}$ satisfy Condition $\mathfrak{P}'$.
Then for
$G_n = I_{p^2}+P_p+\Gamma_n$ 
with $\Gamma_n$ from \Cref{iid1}, we have 
$$\sqrt{m} (D_p^\top G_n D_p)^{-1/2} D_p^\top \vec\left(U_{n}^\top S_{m,n}^\top S_{m,n}U_{n}-I_p\right)\Rightarrow \mathcal{N}(0,I_{p(p+1)/2}).$$
\end{lemma}

\subsection{Uniform Orthogonal Sketching}
\label{haar}
Now, we 
let $\gamma_{n} = m/n$
and introduce a sketching matrix 
$S_{m,n}=\gamma_{n}^{-1/2} S_0$,
where 
$S_0$ is an $m\times n$ matrix drawn from the uniform measure over $m\times n$ partial orthogonal matrices; 
so that $\E S_{m,n}=0$ and $\E S_{m,n}^\top S_{m,n}=I_n$.
Here, Condition $\mathfrak{P}$ and $\mathfrak{P'}$ do not impose any restriction.

{\bf PCA.}
We have the following result for inference in PCA when using sketching matrices uniformly distributed on the space of all $m\times n$ partial orthogonal matrices.
\begin{theorem}[Inference for Eigenvalues and Eigenvectors with Uniform Orthogonal Sketching]\label{PCAHaar}
If  $(X_n)_{n\ge  1}$ satisfies Condition \ref{spectralcondition} \textup{\texttt{B}}
and $m/n\to \gamma\in(0,1)$, then
\begin{equation*}
\sqrt{\frac{m}{2(1-\gamma_n)}}\hat{\Lambda}_{m,n,i}^{-1}\left(\hat{\Lambda}_{m,n,i}-\Lambda_{n,i}\right)\Rightarrow \N(0,1),
\end{equation*}
and for any vector $c\in\psp$ satisfying $\limsup_{n\to\infty}(c^\top v_{n,i})^2<1$,
\begin{equation*}
\sqrt{\frac{m}{(1-\gamma_n)}}\left(\sum_{k\neq i}\frac{\hat{\Lambda}_{m,n,i}\hat{\Lambda}_{m,n,k}}{(\hat{\Lambda}_{m,n,i}-\hat{\Lambda}_{m,n,k})^2} (c^\top\hat{v}_{m,n,k})^2\right)^{-1/2}c^\top(\hat{v}_{m,n,i}-v_{n,i})\Rightarrow \N(0,1).
\end{equation*}
\end{theorem}

{\bf LS.}
\Cref{thhaar} provides a method to construct confidence regions for $\beta_n$ for uniform orthogonal sketching. 
This result is a special case of Theorem 3.8 in \cite{zhang2023framework}
for fixed $p$,
and is presented only for completeness.
Moreover, we also present 
a shorter proof following our unified framework.

\begin{theorem}[Inference in Least Squares with Haar Sketching, \cite{zhang2023framework}]\label{thhaar}
Let $p$ be fixed and $m,n-m\to \infty$. Suppose that $\gamma_n^{1/2}S_{m,n}$ is uniformly distributed over the space of all partial $m\times n$ orthogonal matrices. Additionally, we assume that $X_n$ satisfies Condition \ref{spectralcondition} \textup{\texttt{A}}. For the estimator $\hat{\Sigma}_{m,n}$ defined as
$\hat{\Sigma}_{m,n} = \|\tep_{m,n}\|^2\cdot  (\tX_{m,n}^\top \tX_{m,n})^{-1}$,
we have
$\left(m/(1-\gamma_n)\right)^{1/2}\hat{\Sigma}_{m,n}^{-1/2}(\hbs-\beta_n) \Rightarrow \N(0,I_p).$
Furthermore, for the estimator $\hSigp$ defined as
\begin{equation*}\label{hshaarp}
   \hSigp= \hbeta_{m,n}^{\pa\top}\tX_{m,n}^\top \tX_{m,n} 
\hbp \cdot (\tX_{m,n}^\top \tX_{m,n})^{-1}+ \hbp \hbpt,
\end{equation*}
we have
$\left(m/(1-\gamma_n)\right)^{1/2}(\hSigp)^{-1/2}(\hbp-\beta_n) \Rightarrow \N(0,I_p).$
\end{theorem}

{\bf CLT.}
The results of the inference mentioned earlier are based on 
the following central limit theorems 
for quadratic forms arising from Haar sketching matrices.

\begin{lemma}[Limiting Distributions of One-dimensional Quadratic Forms for Haar Sketching]\label{lemgeneralqfhaar}
For any sequences $(a_{n})_{n\ge 1}$, $(\ta_{n})_{n\ge 1}$ such that $a_n,\ta_n \in\nsp$ for all $n\ge 1$, as $m,n-m\to \infty$ 
we have 
$$\sqrt{\frac{m}{1-\gamma_n}} \sqrt{\frac{1}{1+(a_n^\top \ta_n)^2}}\left(a_n^\top S_{m,n}^\top S_{m,n}\ta_n - a_n^\top \ta_n\right) \Rightarrow \N(0,1).$$
\end{lemma}

\begin{lemma}[Asymptotic Normality of Symmetric Quadratic Form for Haar Sketching]\label{lemsymqfhaar}
With the same notation as above
and $G=I_{p^2}+P_p$, $\gamma_n\to \gamma\in(0,1)$,
we have
\begin{equation*}\label{anh}
\sqrt{m}(1-\gamma_n)^{-1/2} \vec(U_{n}^\top S_{m,n}^\top S_{m,n}U_{n}-I_p)\Rightarrow \N(0,G).
\end{equation*}
\end{lemma}

\section{Computational Considerations}
\label{comp}

We now examine the computational complexity of performing statistical inference using various sketching methods for both least squares and PCA. 
We are considering computation over a single machine, 
in which case communication cost is not a bottleneck.
We perform
floating-point operation (flop) 
count calculations based on standard matrix computation algorithms. 
For least squares, the original dataset $(X_n, y_n)$ is of size $n \times (p+1)$, while the sketched dataset $(\tX_{m,n}, \ty_{m,n})$ has 
size $m \times (p+1)$; 
for PCA, we only have 
the $n \times p$ matrix
$X_n$ and 
the $m \times p$ matrix
$\tX_{m,n}$.

The computational cost 
consists of two primary components. 
The first part involves obtaining the sketched data $(\tX_{m,n}, \ty_{m,n})$ for 
least squares, or only $\tX_{m,n}$ for PCA. 
The second part involves 
solving the normal equations for the least squares problem, or the computation of eigenvalues and eigenvectors for $\tX_{m,n}^\top \tX_{m,n}$. 
The cost of the second part is $O(mp^2)$, using standard matrix computation methods \citep{golub1989matriz}
irrespective of the chosen sketching method. Therefore, we focus on the cost incurred in the initial step of obtaining the sketched data.

Computing $\tX_{m,n}$ using i.i.d.~sketching in general consists of multiplying 
two dense matrices of size $m\times n$ and $n\times p$,  
and in the worst case takes $O(mnp)$ flops. 
For uniform orthogonal sketching, a significant part of the computational cost arises from generating the sketching matrix. 
This is achieved by applying SVD to an $m \times n$ matrix with Gaussian entries and extracting the matrix of its right singular vectors. 
Since $m < n$, the cost is approximately $O(m^2 n)$, which is an upper bound on the cost of obtaining $\tX_{m,n}$, resulting in an overall complexity of $O(m^2 n)$. 
Hadamard sketching, implemented using the Fast Fourier Transform, incurs a cost of approximately $O(np \log_2 m)$. On the other hand, the cost of the Sparse Sign Embeddings for obtaining $\tX_{m,n}$ is $O(\zeta_m\mathrm{nnz}(X_n))$, where $\mathrm{nnz}(X_n)$ denotes the number of nonzero elements in $X_n$, see e.g., \cite{clarkson2017low}.
These theoretical bounds on the computational cost are summarized in Table \ref{tabcost}.

\begin{table}[ht]
\centering 
\caption{Computational cost of performing sketched LS and PCA using various types of sketching matrices. Here, 
the sample size is $n$, 
the dimension of the data is $p\le n$,
and the sketched sample size is $p\le m\le n$; 
while $\mathrm{nnz}(X_n)$ denotes the number of nonzeros of the matrix $X_n$.}
\begin{tabular}{  |c|c|c|c|c| } 
\hline
 & {i.i.d.~}& {Haar} & {Hadamard} & {Sparse Sign Embeddings}\\
\hline
{LS} & $O(mnp)$ &  $O(m^2 n)$ & $O( n p \log_2 m + m p^2)$ & $O(\zeta_m\mathrm{nnz}(X_n)+mp^2)$ \\
\hline
{PCA} & $O(mnp)$ & $O(m^2 n)$
& $O( n p \log_2 m + m p^2)$ & $O(\zeta_m\mathrm{nnz}(X_n)+mp^2)$ \\ 

\hline
\end{tabular}\label{tabcost}
\end{table}

The conclusions of this analysis in general align with the existing common beliefs that Hadamard sketching and Sparse Sign Embeddings are two of the best sketching methods among the ones considered.
Their computational complexity is lower than that of the other approaches considered such as i.i.d.~and Haar sketching. 
Moreover, we show in experiments (see \Cref{compare_optsubsamp_wallclock_Case2})
that Hadamard sketching can have a better time accuracy trade-off than subsampling methods.
However we emphasize that to achieve their full potential, each of these methods requires optimized numerical implementations.

\section{Numerical Experiments}\label{secsimulation}

We conduct simulations to evaluate and compare various methods on both simulated and empirical data, in terms of both statistical performance and running time (\Cref{rt}).
We compare various forms of sketching in \Cref{csk}.
We consider the following three cases for the simulated data:
\begin{itemize}
    \item Case 1: For both LS and PCA, 
    we generate $X_n=U_n L_nV_n^\top$, where $U_n$ is sampled uniformly from the set of $n\times p$ partial orthogonal matrices, 
    $V_n^\top$ is sampled uniformly from the set of $p\times p$ orthogonal matrices, 
    and 
    $L_n=\diag(1,1/2,\ldots,1/p)$. 
    Additionally, for LS we generate $y_n\in \R^n$ with i.i.d.~Uniform $(0,1)$ entries.
    
    \item Case 2:  
    As \cite{lopes2018error}
    in a slightly different setting, 
    we let $A_n$ be an $n\times p$ matrix with i.i.d.~rows distributed according to the multivariate t-distribution $t_2(0,C)$ where $C =(c_{ij}) = (2\cdot 0.5^{|i-j|})$. 
    Let $U_n$ be the left singular matrix of $A_n$. We construct $L_n$ with entries equally spaced in the interval $[0.1,1]$ and generate $V_n^\top$ as the right singular matrix of a matrix with i.i.d.~$\N(0,1)$ entries.
    Finally, we set $X_n = U_nL_n V_n^\top$ for both LS and PCA. 
    To generate $y_n$ in LS, we set $b= (1_{0.2p},t 1_{0.6p},1_{0.2p})^\top$ with $t=0.1$, ensuring that $0.2p$ is an integer.
    Moreover, we generate $\mathcal{E}_n\in \R^n$ with i.i.d.~normal entries having a standard deviation of 0.01, and then let $y_n = X_n b + \mathcal{E}_n$.

    \item Case 3: 
    For both LS and PCA, we let $X_n$ 
    be an $n\times p$ Gaussian random matrix
    such that for all $j\in[p]$,
    $(X_n)_{i,j}\sim \mathcal{N}(0,1)$ when $i\le  n/2$, and $(X_n)_{i,j}\sim \mathcal{N}(5,1)$ otherwise. After that, we standardize each column to have zero mean and unit variance.  
    For linear regression, we also generate $y_n\in \R^n$ with i.i.d.~Uniform$(0,1)$ entries.
\end{itemize}

In Case 1, 
the left singular matrix 
meets all delocalization conditions required by our results, 
due to the isotropy of the normal distribution.  
The design in Case 2 may not satisfy the delocalization conditions
required for Hadamard sketching and CountSketch. 
This case illustrates that those conditions are in fact needed for asymptotic normality. 
The delocalization is governed 
by the distribution from which the rows are drawn. 
If we sample each row from a heavy-tailed distribution, 
we
expect that 
more outliers---observations with a large norm---may be present in the matrix. 
Consequently, the singular vectors are more likely to be localized. 
Simulations with 
several degrees of freedom for the $t$-distributions and with several covariance matrices 
show that our formulas are more accurate
for larger degrees of freedom.
For a given degree of freedom, 
our formulas have roughly the same level of accuracy for various covariance matrices.

In Case 3,
before normalization,
$X_n$ has one spiked singular value of order $\sqrt{n}$ and $p-1$
singular values of order $\sqrt{n}$. 
The relative eigengap condition holds for the first singular value. 

\subsection{Inference from Sketch-and-Solve Estimators}
We perform experiments to compare several sketching methods discussed in the paper. 

\subsubsection{Simulated Data}
In this section, we compare the performance of various sketching methods for statistical inference in least squares and PCA. 
We generate simulated data 
with $n=2,048$ and $p=15$ for each of the cases.
To evaluate the methods, we generate 500 independent sketches for each value of $m$ ranging from 200 to 1600 with gaps of 200. 
We use three sketching methods: i.i.d.~sketching, Hadamard sketching, and CountSketch.
We note that SSE is very similar to CountSketch, both in theory and in our experiments, therefore we focus on CountSketch in the experiments.

For Case 1, \Cref{n2048p15_Case1_coverage_merge} presents the coverage of 95\% intervals for the first coordinate of $\beta_n$, 
the largest eigenvalue and the 
first coordinate of the
corresponding eigenvector of $X_n^\top X_n$, while \Cref{coverage_merge} shows the results for the second largest eigenvalue and the corresponding eigenvector. 
These figures also show a 95\% Clopper-Pearson interval for the coverage. 
All methods under consideration exhibit coverage probabilities close to the nominal level of $0.95$.
For other cases, we show the corresponding simulation results in \Cref{coverage_merge}; we sometimes observe some undercoverage for small $m$.

\begin{figure} 
    \centering
    \includegraphics[width=1.0\textwidth]{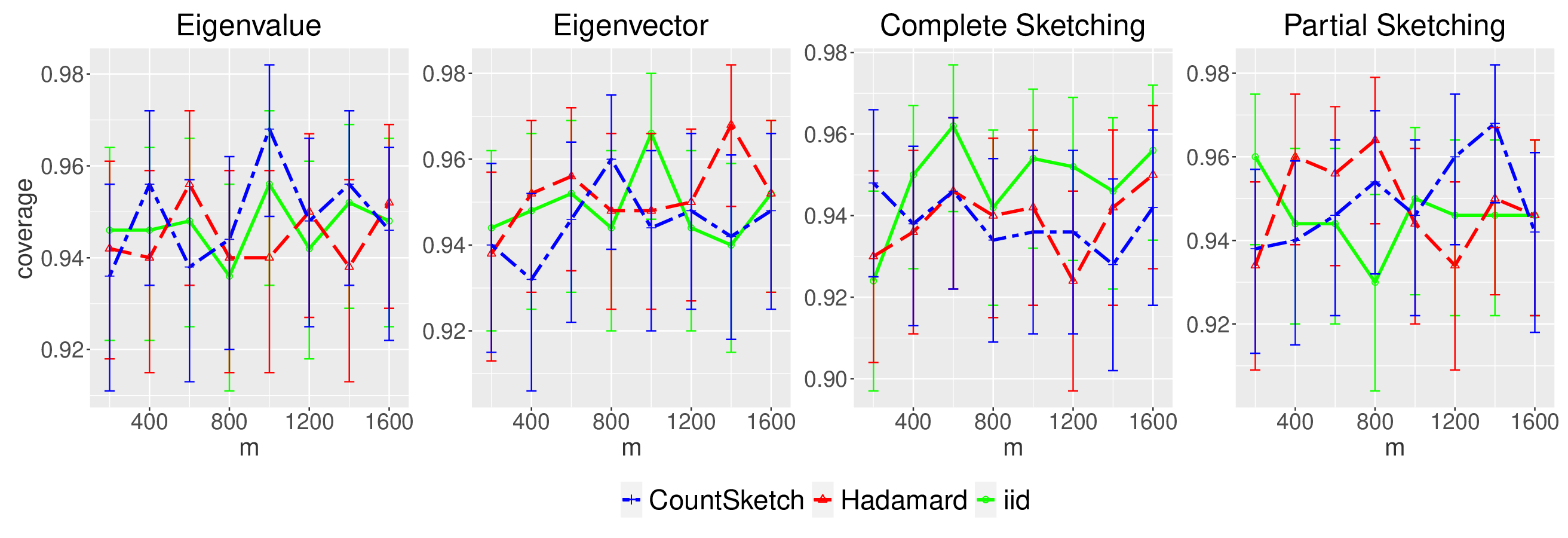}
    \caption{Coverage of nominal 95\% intervals, 
    along with 95\% Clopper-Pearson intervals for the coverage, for $\lambda_k$, $c^\top v_k$, and the first coordinates of two least squares solutions in Case 1, with $p=15, n=2048$, and 500 Monte Carlo trials for each setting.}
    \label{n2048p15_Case1_coverage_merge}
\end{figure}

\begin{figure} 
    \centering
    \includegraphics[width=1.0\textwidth]{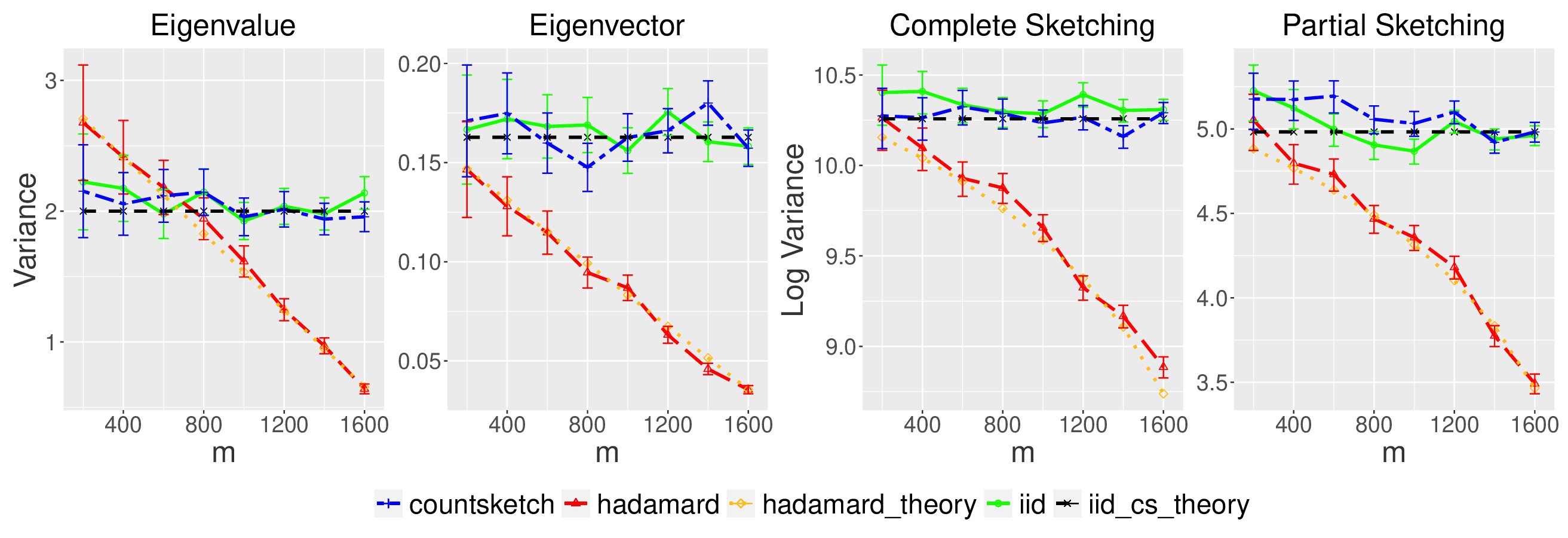}
    \caption{Variances of $\sqrt{m}\hat{\Lambda}_{m,n,k}$ and $\sqrt{m}c^\top\hat{v}_{m,n,k}$, and logarithm of variances of $\sqrt{m}c^\top \hbs$ and $\sqrt{m}c^\top \hbp$ for Case 1, using $p=15, n=2048, k=1$, and $c=(1,0,\ldots,0)^\top$. Here ``iid\_theory" is obtained from \Cref{PCAiid} and \Cref{thiid},  ``hadamard\_theory" from \Cref{PCAsrht} and \Cref{sechada}, ``CountSketch\_theory" from \Cref{PCAcs} and \Cref{cs}, corresponding to least squares and PCA, respectively.}
    \label{n2048p15_Case1_var_coverage_merge}
\end{figure}

\subsubsection{Empirical Data Example}

An empirical illustration of 
sketched least squares, 
we consider the New York flights (nycflights13) dataset \citep{nycflights13}. This has 60,448 observations and 21 features, with the response being the arrival delay. 
We show coverage results for the least squares solution in \Cref{realdata_conf_merge_new}.
We observe that all methods have coverage close to the nominal level. 

To illustrate sketch-and-solve PCA, we use the Human Genome Diversity Project (HGDP) dataset, previously analyzed in studies such as \cite{cann2002human, li2008worldwide}. 
We use the Centre d’Etude du Polymorphisme Humain panel, which contains Single Nucleotide Polymorphism (SNP) data for 1,043 observations, representing 51 populations from Africa, Europe, Asia, Oceania, and the Americas. 
Focusing on the SNPs located on chromosome 22,
we create the data matrix $X_n$ by selecting the first 20 features. In this $n\times p$ matrix, $X_{ij}$ represents the count of the minor allele of SNP $j$ in the genome of individual $i$ and takes values in $\{0, 1, 2\}$. 
To preprocess the data, we perform SNP-wise standardization by centering each SNP using its mean and scaling by its standard error. 
Subsequently, we impute the missing values as zeros, as in \cite{dobriban2019deterministic}.
\Cref{coverage_merge,var_merge} shows the coverage results as well as the asymptotic variances for different sketching methods based on the HGDP dataset.
We observe that the coverage probability of the considered methods is close to the nominal level.

We also perform sketch-and-solve PCA using the Million Song Dataset (MSD, \cite{Bertin-Mahieux2011}), 
as an example of a larger dataset, with $n = 515,344$ and $p = 90$. The analysis takes about two days on a personal computer. 
We present coverage results for the first singular value and its corresponding right singular vector in \Cref{realdata_conf_merge_new}, 
showing that the Clopper-Pearson intervals for the coverage of all approaches are centered near the nominal level.

\subsubsection{Running Time}
\label{rt}

To empirically analyze the running time of our methods,
we conduct experiments 
with various types of sketching  methods for PCA in Cases 1 and 3. 
We choose $m$ ranging from $200$ to $1600$ in increments of $200$. 
For a fixed value of $m$, we repeat the experiment process $500$ times.

From  the results in 
\Cref{testtime_inset}, Table \ref{tabcostexper1} and \ref{tabcostexper2}, 
CountSketch is the fastest method, while sparse sign embeddings with $\zeta_m > 1$ are slower but still perform very well. 
In terms of computational cost, SRHT is slower than CountSketch and sparse sign embeddings.
Sketching matrices with i.i.d.~entries are slower, and Haar sketching matrices are the slowest. 
The results from \Cref{testtime_inset} qualitatively 
align with our theoretical analysis of computational complexity from \Cref{comp}.
We emphasize that specific wall clock times depend on the numerical implementation used, and our experiments
should be viewed as giving only a
rough idea of the relative  performance of the various methods.

We also analyze a larger
data set
in Case 1, with $n=2^{21}, p=100$, and $k=1$. 
This leads to a dataset that is large enough to be somewhat challenging to be analyzed on a PC, but still realistic to occur in applications.
We use SRHT as the primary sketching method for PCA. 
We vary the value of the sketch size $m$ from 5,000 to 75,000 with increments of 10,000. 
The coverage and computation time results are presented in \Cref{massivedata}.
Our solution achieves the target coverage even when the sketch dimension is small. 
Additionally, as the sketch size increases, 
the computational time required for SRHT sketched PCA 
grows slowly, but remains significantly lower than that of PCA on the full data.
At the same time, 
as the sketching dimension increases, 
the length of the confidence intervals for the SRHT sketched solution decreases, as expected. 
Here the values of the target parameters of interest are $\lambda_1 = 1$ and 
$c^\top v_1 = 0.087$;  hence the confidence intervals are short and informative.
These experiments show the trade-off 
between statistical accuracy and computational costs.

\begin{figure}
    \centering
    \includegraphics[width=\textwidth]{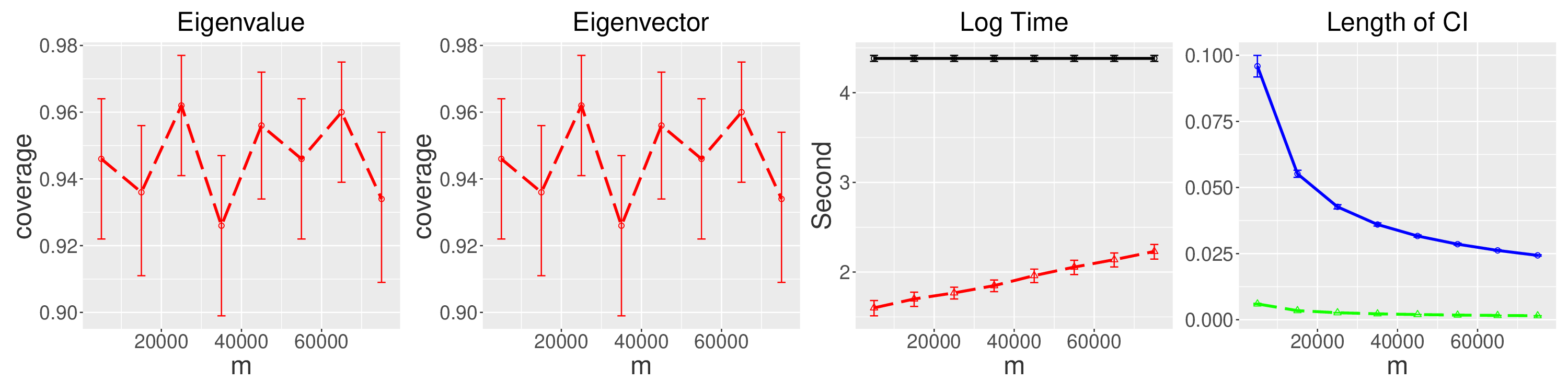}
    \caption{The first two subfigures show the coverage of nominal 95\% intervals, along with 95\% Clopper-Pearson intervals for the coverage, for $\lambda_1$ and $c^\top v_1$ as a function of the sketch size $m$; see \Cref{rt}. 
    The results cover Case 1, with $p=100, n=2^{21}$, and 500 Monte Carlo trials. 
    The third subfigure displays the computational time for the SRHT sketched PCA (red dashed line) and the original PCA solution (black solid line). The fourth subfigure shows the length of the confidence intervals of SRHT sketched  eigenvalues (blue solid line) and eigenvectors (green dashed line).}
    \label{massivedata}
\end{figure}

\begin{figure}[ht]
    \centering
    \includegraphics[width=0.5\textwidth]{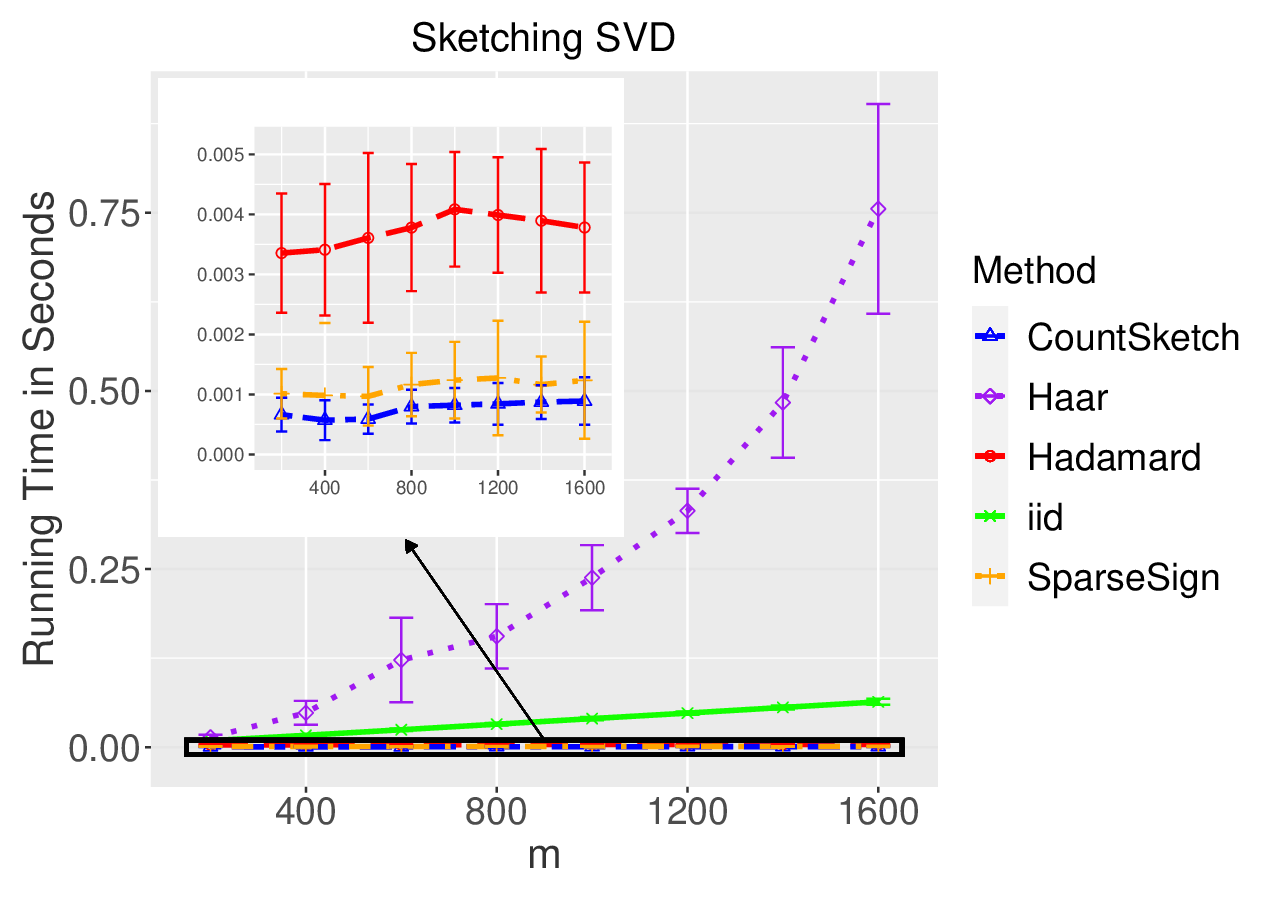}
    \caption{Computational cost experiment in Case 1 with $p=15$, $n=2048$, where $m$ ranges from $200$ to $1600$ in increments of $200$. For each sketching method, 500 simulations are conducted, and the mean time of a single step is calculated and plotted. 
    Standard errors over 500 simulations are shown. Here we use Sparse Sign Embeddings with $\zeta_m=8$.}
    \label{testtime_inset}
\end{figure}

\begin{table}[ht]\label{rt1}
\centering 
\caption{Computational cost (in seconds) of the experiments for Case 1 with $500$ simulations, $p=15$, $n=2048$. 
}
\begin{tabular}{  |c|c|c|c|c|c|c|c|c| } 
\hline
{$m$} & 200 & 400 & 600 & 800 & 1000 & 1200 & 1400 & 1600\\
\hline
{i.i.d.~Gaussian Sketching} & 4.54 & 8.13 & 11.96 & 15.73 & 19.52 & 23.28 & 27.27 & 31.31 \\
\hline
{SRHT Sketching} & 2.18 & 2.08 & 2.06 & 2.26 & 2.52 & 2.25 & 2.14 & 2.19 \\
\hline
{CountSketch} & 0.32 & 0.30 & 0.33 & 0.38 & 0.39 & 0.46 & 0.48 & 0.51 \\
\hline
{SSE, $\zeta_m=8$} & 0.67 & 0.71 & 0.72 & 0.83 & 0.85 & 0.84 & 0.90 & 0.96 \\
\hline
\end{tabular}\label{tabcostexper1}
\end{table}

\begin{table}[ht]\label{rt2}
\centering 
\caption{Computational costs  (in seconds) of the experiments for Case 3, using the same
protocol as in \Cref{rt1}.}
\begin{tabular}{  |c|c|c|c|c|c|c|c|c| } 
\hline
{$m$} & 200 & 400 & 600 & 800 & 1000 & 1200 & 1400 & 1600\\
\hline
{i.i.d.~Gaussian Sketching} & 4.48 & 8.22 & 12.13 & 16.31 & 19.93 & 24.78 & 29.98 & 33.26 \\
\hline
{SRHT Sketching} & 2.23 & 2.08 & 2.23 & 2.22 & 3.07 & 2.41 & 2.74 & 2.43 \\
\hline
{CountSketch} & 0.39 & 0.37 & 0.37 & 0.49 & 0.51 & 0.55 & 0.59 & 0.67 \\
\hline
{SSE, $\zeta_m=8$} & 0.85 & 0.76 & 1.04 & 1.14 & 1.23 & 1.19 & 1.24 & 1.25 \\
\hline
\end{tabular}\label{tabcostexper2}
\end{table}

\subsection{Comparison of Various Forms of Sketching}
\label{csk}
We proceed to compare the asymptotic variances across
various sketching methods.
\Cref{n2048p15_Case1_var_coverage_merge,var_merge}
compare various sketching methods in principal component analysis with respect to the ratio $m/n$ of the sketch size to the sample size. 
These figures display the variance of $\sqrt{m}c^\top \hat{\Lambda}_{m,n,k}$ for simulated data in Cases 1, 2, and 3, respectively, where $n=2048$, $p=10$, and $500$ replications are considered for each $m$. 

The asymptotic variances of i.i.d.~sketching estimators for least squares are provided
in equations \eqref{assiidfix} and \eqref{assiidfix2} for the sketch-and-solve and partial sketching cases, respectively. Similarly, for PCA, the theoretical results can be derived from Table \ref{sumasyPCA}.
Notably, the asymptotic properties of CountSketch closely resemble those of i.i.d.~Gaussian sketching, denoted as ``iid\_cs\_theory". Furthermore, we use ``hadamard\_theory" to denote the asymptotic results for SRHT sketching.

For sketch-and-solve least squares in Case 1, 
when the signal-to-noise ratio $R_n^2$ approaches zero,
partial sketching estimators have 
smaller variances than sketch-and-solve estimators, see Figure \ref{n2048p15_Case1_var_coverage_merge}. 
This finding aligns with the comparison of the relative efficiency of complete and partial sketching from \cite{ahfock2021statistical}. 
The empirical variances closely approximate the theoretical asymptotic variances. 

In Case 1 for PCA, 
the eigengap is large for all $k\in[p]$,
ensuring that the variances are close to the predicted values, 
see Figure \ref{n2048p15_Case1_var_coverage_merge}.
We also observe good performance 
in other cases, see \Cref{var_merge}.
Furthermore, when employing Hadamard sketching, 
the variances of both $\sqrt{m}\hat{\Lambda}_{m,n,k}$ and $\sqrt{m}c^\top\hat{v}_{m,n,k}$ 
show a linear decrease in $m$, 
and tend to zero when $m$ approaches $n$. 
In contrast, there is no significant decreasing trend in the asymptotic variance of i.i.d.~sketching and CountSketch when $p$ remains fixed. Instead, these variances remain relatively stable as $m$ increases. 
These observations are consistent with the theoretical results summarized in Table \ref{sumasyPCA}.

\subsubsection{Comparison with Optimal Subsampling}

We also compare our sketched least squares methods 
with the  optimal subsampling algorithm from \cite{yu2022optimal}. 
Their Algorithm 2 involves three steps. First, a pilot subsample of $r_0$ datapoints is chosen  uniformly,  
and is used to form a rough estimate of the parameter. 
Then, this estimate 
is used to further estimate the optimal subsample weights, 
and the new datapoints are added by weighted subsampling. 
The final subset is used to solve the maximum likelihood equation and obtain the final estimate. 

We compare the asymptotic MSE of the least squares solutions in 
Figure \ref{compare_optsubsamp_wallclock_Case2}.
We focus on Hadamard sketching due to
the availability of an efficient implementation in the same language---R---for which the implementation of the methods from 
\cite{yu2022optimal} are available. 

For Case 2, all subsampling methods are less accurate than Hadamard sketching. 
This shows 
the benefits of using sketching-type methods beyond subsampling. 
In Case 2, where the data matrices differ significantly from i.i.d.~models, subsampling methods tend to be less stable.  
On the other hand, sketching methods like Hadamard sketching mix the data before subsampling, which helps mitigate the impact of ``outlier" data points. 
Figure \ref{compare_optsubsamp_wallclock_Case2} shows that Hadamard sketching achieves the ``empirical Pareto frontier”, 
which means that among the methods compared, empirically Hadamard sketching has the best  tradeoff in balancing computational cost and statistical accuracy. 

\begin{figure}[ht]
    \centering
    \includegraphics[width=0.55\textwidth]{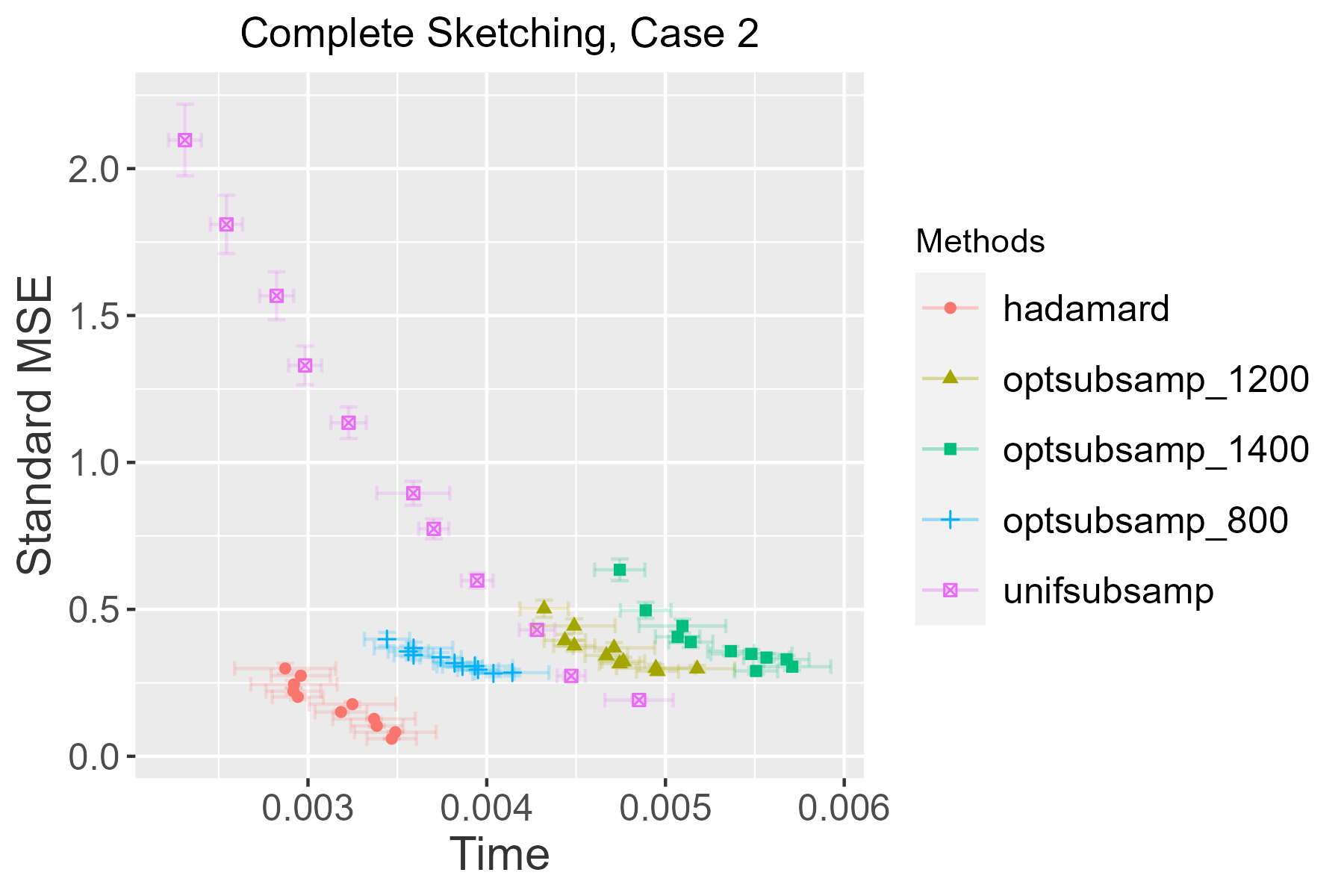}
    \caption{Comparison of subsampling methods for Case 2, using $p=25, n=4096$ with 2000 Monte Carlo trials for each setting. Here ``optsubsamp\_800” is the optimal subsampling method from Algorithm 2 in \cite{yu2022optimal} with pilot subsample size $r_0=800$, while ``optsubsamp\_1200” and ``optsubsamp\_1400” 
     refer to the same algorithm using $r_0=1200$ and $r_0=1400$.}
    \label{compare_optsubsamp_wallclock_Case2}
\end{figure}

\section{Summary}
In this work, we 
propose a unified methodology for statistical inference in sketched PCA and LS 
built on the asymptotic normality of quadratic forms. 
This relies on finding
the limiting distribution of quadratic forms of the output of several sketching algorithms. 
We find novel asymptotic distributions of quadratic forms for Sparse Sign Embeddings and SRHT sketching. 
For i.i.d.~sketching matrices, we
significantly expand the conditions on 
the data under which the results apply
compared to prior works such as \cite{zhang2023framework}.
We also perform simulations to evaluate and compare our methods on both simulated and empirical data. 
Our simulations show that in certain scenarios, SRHT sketching achieves a better computation-accuracy tradeoff compared to subsampling-type methods. 

\section*{Acknowledgements}

This work was supported in part by the NSF, the ONR, and the Sloan Foundation.

\section{Appendix}
\subsection{Technical Reminders}
\label{tr}

{\bf Additional notation.}
We use $\sum_{i<j}$ and $\sum_{i\neq j}$ to abbreviate the summation over $(i,j)\in[n]\times[n]$ subject to  $i<j$ and $i\neq j$, respectively. Also, we define $1_m=(1,\ldots,1)^\top\in\R^m$, 
and 
$\{e_i\}_{i=1}^m$
be the canonical orthonormal basis 
 of $\R^m$.
We will sometimes denote the Euclidean inner product by 
$\langle \cdot, \cdot \rangle$.
Denote by $A=_dB$ that two random vectors $A, B$ have the same distribution.

For any positive integer $p\ge 1$,
we denote the probability distribution
of a random vector $X \in \R^p$ by $L(X)$.
Since the space of probability measures on $\R^p$ is metrizable, 
there is a metric $d_p$ on this space, such that 
for a sequence of random vectors $(X_n)_{n\ge 1}$ and a random vector $X$,
the weak convergence $X_n \Rightarrow X$ is equivalent to $d_p(L(X_n),L(X))\to 0$.

{\bf Tensor Operations.} 
Next, recall that 
for any matrices $A$, $B$ and $C$ of conformable sizes, 
we have that $\vec(ABC)=(C^\top\otimes A)\vec(B)$; see 
for example
 Theorem 2.4.2 in \cite{magnus2019matrix}. Additionally, for two vectors $a,b\in \R^p$, we have that \begin{equation}\label{r1tensor}
     b\otimes a=(b\otimes I_p)a,
 \end{equation} as well as $b\otimes a=\vec(ab^\top).$
Further, 
for the matrix $P_p$ from Definition \ref{spmx},
representing matrix transposition in the vectorized matrix space,
\begin{equation}\label{Pp0}
(a\otimes b)^\top P_p(c\otimes d)=(a\otimes b)^\top (d\otimes c)=(a^\top d)(b^\top c).
\end{equation}
In particular, for any $a\in\R^p$,
$\left((a\otimes I_p)^\top P_p(a\otimes I_p)\right)_{ij}=(a\otimes e_i)^\top P_p(a\otimes e_j)=(a^\top e_i)(a^\top e_j)=e_i^\top aa^\top e_j=(aa^\top)_{ij}$,
for all $i,j\in[p]$, so \begin{equation}\label{Pp1}(a\otimes I_p)^\top P_p(a\otimes I_p)=aa^\top.\end{equation}

Moreover, for the matrix 
$Q_p=\vec(I_p)\vec(I_p)^\top$
 from Definition \ref{spmx},
we have  $(Q_p)_{ij,kl}=\delta_{ij} \delta_{kl}$. 
Thus, 
for all $p\times p$ matrices $\Phi,\Phi$,
we have 
\begin{equation}\label{Qp0}
\vec(\Phi)^\top Q_p \vec(\Psi)=\sum_{(ij),(kl)}\Phi_{ij}\Psi_{kl}\delta_{ij}\delta_{kl}
=\sum_{i,k}\Phi_{ii}\Psi_{kk}
=\biggl(\sum_{i=1}^p\Phi_{ii}\biggr)\biggl(\sum_{k=1}^p\Psi_{kk}\biggr)=(\tr\Phi)(\tr\Psi).\end{equation}
Thus, 
 for any $a\in\R^p$,
$\left((a\otimes I_p)^\top Q_p(a_n\otimes I_p)\right)_{ij}=(a\otimes e_i)^\top Q_p(a\otimes e_j)=(a^\top e_i)(a^\top e_j)=e_i^\top aa^\top e_j=(aa^\top)_{ij}$, for all $i,j\in[p]$, so \begin{equation}\label{Qp1}(a\otimes I_p)^\top Q_p(a\otimes I_p)=aa^\top.\end{equation}

{\bf Subsampled 
randomized Hadamard transform (SRHT).} 
For a positive integer $n$,
we 
let $l={\lceil \log_2 n\rceil}$
and
$n':= 2^l \ge n$.
An SRHT sketching matrix 
\citep{ailon2006approximate}
can be expressed as
$S_{m,n} = \sqrt{n'/m}B H_l D \Pi_n$, where $B$ is an $n'\times n'$ diagonal matrix with
i.i.d.~Bernoulli random variables with a success probability of $m/n'$. 

Furthermore, $D$ is an $n'\times n'$ diagonal matrix with i.i.d.~$\pm 1$ Rademacher random variables on the diagonal.
Finally, $\Pi_n$ is an $n'\times n$ 
with $\Pi_n^\top = [I_n, 0_{n'-n,n}]^\top$, where 
$0_{n'-n,n}$ is a $(n'-n) \times n$ zero matrix.

\subsection{Proof of \texorpdfstring{\Cref{lemgeneols}}{Theorem \ref{lemgeneols}}}
\label{pflemgeneols}

\subsubsection*{Sketch-and-solve}

We first present the proof for $\hbs$. 
Condition \ref{generalquadratic forms}
\textup{\texttt{Bounded}}
implies that for any sequence of 
nonzero vectors $({a_n, \ta_n})_{n\ge 1}$ of unit norm
satisfying Condition $\mathfrak{P}$, $|a_n^\top S_{m,n}^\top S_{m,n} \ta_n - a_n^\top \ta_n|=O_P (\tau_{m,n}^{1/2}/\sqrt{m})$.
By assumption,
for any $(b_n)_{n\ge 1}$
satisfying $b_n\in\psp$ for all $n$,
$({U_nb_n, U_nb_n})_{n\ge 1}$
and $({U_nb_n, \bar{\ep}_n})_{n\ge 1}$ satisfy Condition $\mathfrak{P}$. 
Therefore, taking 
$a_n=U_n b_n$ and $\ta_n=\bar{\ep}_n$, 
and noticing that 
$\|U_n b_n\|^2=b_n^\top U_n^\top U_n b_n=b_n^\top b_n=1$
and $\|\bar{\ep}_n\|=1$,
we find that
$$|b_n^\top (U_n^\top S_{m,n}^\top S_{m,n} U_n - I_p) b_n|=O_P \left(
\frac{\tau_{m,n}^{1/2}}{\sqrt{m}}\right)
\,\,\,\,\mathrm{ and  }\,\,\,\,
|b_n^\top U_n^\top S_{m,n}^\top S_{m,n} \bar{\ep}_n|=O_P\left (\frac{\tau_{m,n}^{1/2}}{\sqrt{m}}\right).$$ 
Therefore, since $p$ is fixed,
for any small $\ep>0$,
we can take a union bound over $b_n$ in a fixed $\ep$-cover of $\Sp$ to conclude that
\begin{equation}\label{nc}
    \|U_n^\top S_{m,n}^\top S_{m,n} U_n - I_p\| = O_P \left(\frac{\tau_{m,n}^{1/2}}{\sqrt{m}}\right), \quad \|U_n^\top S_{m,n}^\top S_{m,n} \bar{\ep}_n \| = O_P \left(\frac{\tau_{m,n}^{1/2}}{\sqrt{m}}\right).
\end{equation}
Using that $\tau_{m,n}/m\to 0$, it follows that 
$  \|(U_n^\top S_{m,n}^\top S_{m,n} U_n)^{-1} - I_p\| = O_P(\tau_{m,n}^{1/2}/\sqrt{m})$.
From equation \eqref{geneformucomplete0},
we have  that
\begin{equation}\label{geneformucomplete}\begin{aligned}
   \hbs& = 
   \beta_n+ 
    \|\ep_n\|\cdot  V_n L_n^{-1} (U_n^\top S_{m,n}^\top S_{m,n} U_n)^{-1} U_n^\top S_{m,n}^\top S_{m,n} \bar{\ep}_n.
\end{aligned}\end{equation}
Therefore,
from \eqref{nc},
for any $c \in\R^n$ such that 
$X_n^{\dag \top} c\neq 0$,
using also that $ U_n^\top U_{n,\perp} = 0$
so that $U_n^\top \bar{\ep}_n=0$,
and  that
$X_n^{\dag \top} = U_n L_n^{-1} V_n^\top$,
we have that if $\ep_n\neq 0$,
\begin{align}\label{approxbyqfskechols}
    &\|\ep_n\|^{-1}\|X_n^{\dag \top} c\|^{-1}c^\top (\hbs -\beta_n)\nonumber\\
    &=\|X_n^{\dag \top} c\|^{-1}c^\top  V_n L_n^{-1} [(U_n^\top S_{m,n}^\top S_{m,n} U_n)^{-1} - I_p] U_n^\top S_{m,n}^\top S_{m,n} \bar{\ep}_n
    +\|X_n^{\dag \top} c\|^{-1}c^\top  V_n L_n^{-1} U_n^\top S_{m,n}^\top S_{m,n} \bar{\ep}_n    \nonumber\\
    &=O_P\left(\frac{\tau_{m,n}}{m}\right)
    +\|X_n^{\dag \top} c\|^{-1}c^\top V_n L_n^{-1} U_n^\top (S_{m,n}^\top S_{m,n}-I_n) \bar{\ep}_n
    + \|X_n^{\dag \top} c\|^{-1}c^\top V_n L_n^{-1} U_n^\top \bar{\ep}_n
    \nonumber\\
    &=\|X_n^{\dag \top} c\|^{-1}c^\top V_n L_n^{-1} U_n^\top (S_{m,n}^\top S_{m,n}-I_n) \bar{\ep}_n+O_P\left(\frac{\tau_{m,n}}{m}\right).
\end{align}
We will use this approximation to study the distribution of $\hbs$.
The distribution of the first term of the last line of \eqref{approxbyqfskechols} can be studied via Condition \ref{generalquadratic forms}
\textup{\texttt{1-dim}}.
In that condition, taking 
$a_n=U_n b_n$ and
$\ta_n=\bar{\ep}_n$
for every $n\ge 1$, where $b_n\in\psp$,  we find that
with $M_n$ defined in  \eqref{defMn}, 
$$
m^{1/2} \tau_{m,n}^{-1/2}\left(b_n^\top M_n b_n\right)^{-1/2}b_n^\top U_n^\top \left(S_{m,n}^\top S_{m,n} - I_n \right) \bar{\ep}_n \Rightarrow \N(0,1).
$$
Since $X_n^{\dag \top} = U_n L_n^{-1} V_n^\top$,
by taking $b_n=L_n^{-1} V_n^\top c/\|X_n^{\dag \top} c\| \in\psp$ and cancelling out $\|X_n^{\dag \top} c\|$, 
we can rewrite this as 
\begin{equation}\label{ssa2}
m^{1/2} \tau_{m,n}^{-1/2}\left(c^\top V_n L_n^{-1} M_nL_n^{-1} V_n^\top c\right)^{-1/2} c^\top V_n L_n^{-1} U_n^\top (S_{m,n}^\top S_{m,n}-I_n) \bar{\ep}_n \Rightarrow \N(0,1).    
\end{equation}

Noting that $L_n^{-1}=\diag(\ell_{n,1}^{-1},\ldots,\ell_{n,p}^{-1})$, with 
 $0<\ell_{n,1}^{-1}\le \ldots\le \ell_{n,p}^{-1} $, 
let 
\beqs
T_n:=
\frac{V_n L_n^{-1} M_nL_n^{-1} V_n^\top}{ \|L_n^{-2}\|}
\textnormal{  and  }
W_n:=m^{1/2}\tau_{m,n}^{-1/2}T_{n}^{-1/2}\|\ep_n\|^{-1}\|L_n^{-1}\|^{-1}(\hbs-\beta_n).
\eeqs
where recall that $\|\cdot\|$ is the operator norm.
Then, recalling the metric $d_p$ metrizing weak convergence from \Cref{tr},
by simple algebra
our desired claim
\eqref{ssa} amounts to 
$d_p(L(W_n),\N(0,I_p))\to 0$. 

We will prove this 
by a  subsequence argument.
First, we argue that we can find a compact set $B \subset \Sp$ such that for all $n\ge 1$, $T_n \in B$. 
To see this, 
first, we have $\|T_n\|\le  \|M_n\| $; also, 
$\lambda_{\min}(T_n)\ge \lambda_{\min}(L_n^{-1})^2\cdot\lambda_{\min}\left(M_n\right)/\|L_n^{-2}\|=\ell_{n,p}^2/\ell_{n,1}^2 \lambda_{\min}\left(M_n\right)$. 
Since
by Condition \ref{generalquadratic forms} {\bf \textup{\texttt{1-dim}}},
for all $n\ge 1$, $a^\top M_n a=\sigma_n^2(U_na, \bar{\ep}_n)$ are bounded uniformly bounded away from zero and infinity for any $a\in\psp$, the spectra of all $(M_n)_{n\ge 1}$ are uniformly bounded away from zero and infinity, and 
thus using that $(X_n)_{n\ge 1}$  satisfies Condition \ref{spectralcondition} \textup{\texttt{A}},
so are the the spectra of all $(T_n)_{n\ge 1}$. 
This shows that there is  a compact set $B \subset \Sp$ such that for all $n\ge 1$, $T_n \in B$.

Next, 
we continue with the  subsequence argument.
Consider any subsequence $(n_k)_{k\ge 1}$ of $\NN$. 
Since $T_{n_k} \in B$ for all $k\ge 1$  and $B$ is bounded, there exists a further sub-subsequence $(n_k')_{k\ge 1}$ of $(n_k)_{k\ge 1}$ such that $T_{n_k'} \to T$, for some $T \in \Sp$. 
We denote  by $(m_k')_{k\ge 1}$ the sequence of sketch sizes associated with $(n_k')_{k\ge 1}$. 
Using that $(c^\top V_n L_n^{-1} M_nL_n^{-1} V_n^\top c)^{1/2} = \|T_n^{1/2} c\|\cdot\|L_n^{-1}\|$,
it follows 
from \eqref{ssa2}
that 
$$(m_k')^{1/2}\tau_{m_k', n_k'}^{-1/2}\|L_{n_k'}^{-1}\|^{-1} 
c^\top V_{n_k'} L_{n_k'}^{-1} U_{n_k'}^\top (S_{m_k',n_k'}^\top S_{m_k',n_k'}-I_{n_k'}) \bar{\ep}_{n_k'} \Rightarrow \N(0,c^\top T c).   $$ 
Then, by \eqref{approxbyqfskechols}, 
since $\|X_{n_k'}^{\dag \top} c\|/\|L_{n_k'}^{-1}\| = O(1)$ and $\tau_{m,n}/m\to 0$,
it follows that 
$$(m_k')^{1/2}\tau_{m_k', n_k'}^{-1/2}\|\ep_{n_k'}\|^{-1}\|L_{n_k'}^{-1}\|^{-1}c^\top\left(\hat{\beta}_{m_k',n_k'}^\s-\beta_{n_k'}\right) \Rightarrow \N(0,c^\top T c).
$$

Using the Cramer-Wold device,
we conclude that 
$$(m_k')^{1/2}\tau_{m_k', n_k'}^{-1/2}\|\ep_{n_k'}\|^{-1}\|L_{n_k'}^{-1}\|^{-1} 
\left(\hat{\beta}_{m_k',n_k'}^\s-\beta_{n_k'}\right)  \Rightarrow \N(0,T).   $$ 
Consequently, 
using again that $T_{n_k'} \to T$,
$d_p\left(L(W_{n_k'}),\N(0,I_p)\right)\to 0$ holds along the subsequence $(n_k')_{k\ge 1}$. Thus, for any subsequence $(n_k)_{k\ge 1}$, there exists a further sub-subsequence $(n_k')_{k\ge 1}$ of $(n_k)_{k\ge 1}$, such that the distance $d_p\left(L(W_{n_k'}),\N(0,I_p)\right)\to 0$. 
Thus, $d_p\left(L(W_{n}),\N(0,I_p)\right)$ tends to zero, finishing the proof of \eqref{ssa}.

\subsubsection*{Partial sketching}

Continuing with partial sketching, 
if $X_n \beta_n\neq 0$,
\begin{equation}\label{geneformupartial}\begin{aligned}
     \hbp &=  V_n L_n^{-1} (U_n^\top S_{m,n}^\top S_{m,n} U_n)^{-1} U_n^\top y_n\\
     &=\beta_n + \| X_n \beta_n\|\cdot V_n L_n^{-1} \left[(U_n^\top S_{m,n}^\top S_{m,n} U_n)^{-1}-I_p \right]\widebar{X_n\beta_n}.
 \end{aligned}    
\end{equation}
Thus, using \eqref{nc}, 
we have
for any $c \in\R^n$ such that 
$X_n^{\dag \top} c\neq 0$ that
\begin{align}\label{approxqfpartialols}
&\| X_n \beta_n\|^{-1}\|X_n^{\dag \top} c\|^{-1}c^\top(\hbp-\beta_n)=
-\|X_n^{\dag \top} c\|^{-1} c^\top V_n L_n^{-1} U_n^\top (S_{m,n}^\top S_{m,n} -I_n)\widebar{X_n\beta_n}+O_P\left(\frac{\tau_{m,n}}{m}\right).
\end{align}

Due to Condition \ref{generalquadratic forms} \textup{\texttt{1-dim}}, for 
$a_n=U_n b_n$
and 
$\ta_n=\widebar{X_n\beta_n}$, 
for every $n\ge 1$, where $b_n\in\psp$, we have $$
m^{1/2}\tau_{m,n}^{-1/2} (b_n^\top \Mp b_n)^{-1/2}b_n^\top U_n^\top \left(S_{m,n}^\top S_{m,n} - I_n \right) \widebar{X_n\beta_n} \Rightarrow \N(0,1),
$$
where $\Mp$ is defined in \eqref{defMn}.
Choosing $b_n=L_n^{-1} V_n^\top c/\|X_n^{\dag \top} c\|$, it follows that
$$
m^{1/2}\tau_{m,n}^{-1/2}\left(c^\top V_nL_n^{-1}\Mp L_n^{-1} V_n^\top c\right)^{-1/2} c^\top V_n L_n^{-1} U_n^\top (S_{m,n}^\top S_{m,n}-I_n) \widebar{X_n\beta_n} \Rightarrow \N(0,1).
$$

Similar to the proof for sketch-and-solve least squares above, 
we define $T'_n:=V_n L_n^{-1} \Mp L_n^{-1} V_n^\top/\|L_n^{-2}\|$, $W'_n:=m^{1/2}\tau_{m,n}^{-1/2}\cdot(T'_{n})^{ -1/2}\| X_n \beta_n\|^{-1}\|L_n^{-1}\|^{-1}(\hbp-\beta_n)$,
and need to show that 
$d_p(L(W'_n),\N(0,I_p))\to 0$. 
This is followed by a subsequence argument essentially identical to the one used above in the proof for sketch-and-solve least squares, and we omit the proof to avoid repetition.

\subsubsection*{Inference}

In particular, 
if Condition \ref{generalquadratic forms} \textup{\texttt{1-dim'}} holds,
 then $M_n  = I_p$, and so \eqref{ssa}
reduces to \eqref{simols}.
Further, we can verify that
Condition \ref{generalquadratic forms} \textup{\texttt{1-dim'}}
with $\alpha\in\{0,1\}$ implies that 
$\Mp(\widebar{X_n\beta_n}) = I_p + (\alpha+1)U_n^\top\widebar{X_n\beta_n}\widebar{X_n\beta_n}^\top U_n$
for $\Mp$ from \eqref{defMn}.
Hence, with 
$\Sigp = \|y_n-\ep_n\|^2 (X_n^\top X_n)^{-1}+ (\alpha+1)\beta_n \beta_n^\top$,
we have
\beq\label{simpa}
m^{1/2}\tau_{m,n}^{-1/2}\left(\Sigp\right)^{-1/2}(\hbp-\beta_n) \Rightarrow \N(0,I_p).
\eeq

Without loss of generality, we can consider $y_n\in\nsp$ and $X_n$ of unit spectral radius. Indeed, for general $y_n$ and $X_n$, we let  $y_n' = y_n/\|y_n\|$ and $X_n'=X_n/\|X_n\|$, 
and we consider the observed data $S_{m,n}(X_n',y_n')$. 
By applying a similar argument
and corresponding rescaled estimators,
$\|X_n\|$ and $\|y_n\|$  cancel out, leading to the desired result for general $X_n$ and $y_n$.

Next, since $\|U_n^\top S_{m,n}^\top S_{m,n} U_n - I_p\|=O_P (\tau_{m,n}^{1/2}/\sqrt{m})$ and $\|X_n\|=1$, we have \begin{equation}\label{89ahd}
    \|\tX_{m,n}^\top \tX_{m,n}-X_n^\top X_n\|=O_P (\tau_{m,n}^{1/2}/\sqrt{m}).
\end{equation}
Moreover, we can write
\begin{equation}\label{89ahd+}
\begin{aligned}
&\| \tep_{m,n}\|^2 =\|S_{m,n} (y_n - X_n \hbs)\|^2\\&= \|S_{m,n} (y_n - X_n \beta_n)\|^2 + \|S_{m,n} X_n(\beta_n - \hbs)\|^2+2\langle S_{m,n} (y_n - X_n \beta_n), S_{m,n} X_n(\beta_n - \hbs)\rangle.
\end{aligned}\end{equation}
By using Condition \ref{generalquadratic forms} \textup{\texttt{Bounded}}, we have $\|S_{m,n} (y_n - X_n \beta_n)\|^2/ \| \ep_n\|^2-1=O_P(\tau_{m,n}^{1/2}/\sqrt{m})=o_P(1)$. 
From \eqref{89ahd}, 
\begin{align*}
&\|S_{m,n} X_n(\beta_n - \hbs)\|^2=
\| X_n (\beta_n - \hbs)\|^2+O_P(\|\beta_n - \hbs\|^2\cdot\tau_{m,n}^{1/2}/\sqrt{m})\le  \|\beta_n - \hbs\|^2(1+o_P(1)).
\end{align*}
Also, from \eqref{simols},
since $\|X_n\|=1$ and 
the condition number of $X_n$ is bounded above from Condition \ref{spectralcondition} \textup{\texttt{A}},
$\|(X_n^\top X_n)^{-1}\| = O(1)$ and
$$\|\hbs-\beta_n\|=O_P\left(\frac{\tau_{m,n}^{1/2}\|\ep_n\|\cdot\|(X_n^\top X_n)^{-1}\|^{1/2}}{\sqrt{m}}\right)=O_P\left(\frac{\tau_{m,n}^{1/2}\|\ep_n\|}{\sqrt{m}}\right).$$
Therefore, 
the second term
on the second line of \eqref{89ahd+}
is $O_P \left(m^{-1} \tau_{m,n}\|\ep_n\|^2 \right)$. 
Further, the third term is $O_P \left(m^{-1/2} \tau_{m,n}^{1/2}\|\ep_n\|^2 \right)$ by the Cauchy-Schwarz inequality. Therefore, 
we can conclude that
$\| \tep_{m,n}\|^2/\| \ep_n\|^2\to_P 1$. 
This verifies \eqref{1dim'lsinf1}.

Continuing with partial sketching,
we have from \eqref{simpa} that
\begin{equation}\label{olspaest}
\|\hbp-\beta_n\|=O_P\left(\frac{\tau_{m,n}^{1/2}\|\Sigp\|^{1/2}}{\sqrt{m}}\right)=O_P\left(\frac{\tau_{m,n}^{1/2}\|\beta_n\|}{\sqrt{m}}\right).\end{equation} 
Thus, $\|\hbp \hbeta_{m,n}^{\pa\top}-\beta_n\beta_n^\top\|=o_P(\|\beta_n\|^2)$. Further,
since we can write  $\|y_{n}-\ep_{n}\|^2 = \|X_n \beta_n\|^2$, 
using \eqref{89ahd} and \eqref{olspaest}, 
we have 
\begin{align*}
&\|\tX_{m,n} \hbp\|^2 -\|y_{n}-\ep_{n}\|^2 
=(\|\tX_{m,n} \hbp\|^2-\|\tX_{m,n}\beta_n\|^2)
+(\|\tX_{m,n}\beta_n\|^2-\|y_{n}-\ep_{n}\|^2)\\
&=\tr\left(\tX_{m,n}^\top \tX_{m,n}(\hbp \hbeta_{m,n}^{\pa\top}-\beta_n\beta_n^\top)\right)+\beta_n^\top (\tX_{m,n}^\top \tX_{m,n}-X_n^\top X_n)\beta_n\\
&\le  \|\tX_{m,n}^\top \tX_{m,n}\|\|\hbp \hbeta_{m,n}^{\pa\top}-\beta_n\beta_n^\top\|+\|\beta_n\|^2\|\tX_{m,n}^\top \tX_{m,n}-X_n^\top X_n\|
=o_P(\|\beta_n\|^2). 
\end{align*}
Therefore, $\|\Sigp-\hSigp\|_{\Fr} = o_{P}(\|\beta_n\|^2)$. 
Now,
$\min_{a\in\psp} a^\top\left(\|X_n \beta_n\|^2 (X_n^\top X_n)^{-1}+(\alpha+1)\beta_n\beta_n^\top\right)a$ 
= $\lambda_{\min}(\Sigp)$, 
and we have for any $a\in\psp$ that
$$\|X_n \beta_n\|^2 a^\top(X_n^\top X_n)^{-1}a\ge \lambda_{\min}(X_n^\top X_n)\|\beta_n\|^2\lambda_{\min}\left((X_n^\top X_n)^{-1}\right)=\lambda_{\min}(X_n^\top X_n)\|\beta_n\|^2.$$ 
Combined with Condition \ref{spectralcondition} \textup{\texttt{A}}, 
it follows that $\|(\Sigp)^{-1}\|
= \lambda_{\min}(\Sigp)^{-1}
=O_P(\|\beta_n\|^{-2})$. 
Combined with $\|\Sigp-\hSigp\|_{\Fr} = o_{P}(\|\beta_n\|^2)$,
this leads to $\|(\Sigp)^{-1}\hSigp-I_p \| = o_P(1)$, 
and  \eqref{1dim'lsinf2} follows.

\subsection{Proof of \texorpdfstring{Corollary \ref{rndls}}{Corollary \ref{rndls}}}
\label{pfrndls}

Recall from \Cref{lemgeneols} that conditional on $(X_n,y_n)_{n\ge 1}$, \eqref{1dim'lsinf1} holds. 
This implies that we also have the unconditional result
\begin{equation}\label{infls1}
m^{1/2}\tau_{m,n}^{-1/2}\|\tep_{m,n}\|^{-1}(\tX_{m,n}^\top \tX_{m,n})^{1/2}(\hbs-\beta_n) \Rightarrow \N(0,I_p).
\end{equation} 
Moreover, we have 
$\sigma_o^{-1}(X_n^\top X_n)^{1/2}(\beta_n-\beta^*)\Rightarrow\N(0,1)$, 
see Example 2.28 in \cite{van1998asymptotic}.
From \eqref{89ahd} and \eqref{89ahd+}, we have $(X_n^\top X_n)^{-1}(\tX_{m,n}^\top \tX_{m,n})\to_P I_p$ and $\|\tep_{m,n}\|^{-1}\|\ep_{n}\|\to_P 1$. Due to the law of large numbers, we have $\|\ep_{n}\|^2/n\to_P \sigma_o^2$.  
From Slutsky's theorem, since $m/(n\tau_{m,n}) \to 0$, we conclude \eqref{rndls1}.

For partial sketching we have from \eqref{simpa} in \Cref{pflemgeneols} that  unconditionally,
\beq\label{infls2} m^{1/2}\tau_{m,n}^{-1/2}\left(\Sigp\right)^{-1/2}(\hbp-\beta_n) \Rightarrow \N(0,I_p).\eeq
Now, $\sigma_o^{-1}(X_n^\top X_n)^{1/2}(\beta_n-\beta^*)\Rightarrow\N(0,1)$ is equivalent to
\begin{equation*}
\sqrt{\frac{\tau_{m,n}}{m}}\sigma_o^{-1}(X_n^\top X_n)^{1/2}\left(\Sigp\right)^{1/2}\cdot\sqrt{\frac{m}{\tau_{m,n}}}\left(\Sigp\right)^{-1/2}(\beta_n-\beta^*)\Rightarrow\N(0,1).
\end{equation*}
Observe that
\begin{align*}
\left\|\left(\Sigp\right)^{-1/2}(X_n^\top X_n)^{-1/2}\right\|&=\lambda_{\min}^{-1} \left((X_n^\top X_n)^{1/2}\left(\Sigp\right)^{1/2}\right)\\
&=\lambda_{\min}^{-1/2} \left(\|y_n-\ep_n\|^2 + (\alpha+1)X_n^\top X_n\beta_n \beta_n^\top\right)\\
&\leq\left(\|y_n-\ep_n\|^2 + (\alpha+1)\lambda_{\min}(X_n^\top X_n)\lambda_{\min}(\beta_n \beta_n^\top)\right)^{-1/2}\leq\|y_n-\ep_n\|^{-1}. 
\end{align*}
From our assumptions, $m^{1/2}/(\tau_{m,n}^{1/2}\|X_n\beta^*\|)\to_P 0$. Therefore, $m^{1/2}\tau_{m,n}^{-1/2}\left\|\left(\Sigp\right)^{-1/2}(X_n^\top X_n)^{-1/2}\right\|\to_P 0$. Combining with Slutsky's theorem, we have $m^{1/2}\tau_{m,n}^{-1/2}\left(\Sigp\right)^{-1/2}(\hbp-\beta^*) \Rightarrow \N(0,I_p)$.
Finally, as in the proof of \eqref{1dim'lsinf2}, we can replace $\Sigp$ to its consistent estimator $\hSigp$. This completes the proof of \eqref{rndls2}.

\subsection{Proof of \texorpdfstring{\Cref{infPCAgeneral}}{Theorem \ref{infPCAgeneral}}}
\label{pfinfPCAgeneral}

The proof relies on the following result about asymptotic distributions for
the sketched PCA solutions, 
obtained leveraging the delta method.

\begin{proposition}[Asymptotic Normality of Eigenvectors and Eigenvalues]
\label{delta}
Suppose that the sequence $(X_{n}^\top X_{n})_{n\ge 1}$
of matrices
converges to $\Sigma\in\Sp$, and $(G_n)_{n\ge 1}$ converges to a positive semi-definite matrix $G$. Assume $\Sigma$ has the spectral decomposition 
$V^\top L^2 V$, where $V=(\nu_1,\nu_2,\ldots,\nu_p)$ and $L=\diag(\lambda_{1,0},\ldots,\lambda_{p,0})$.
Under Condition \ref{generalquadratic forms} \textup{\texttt{Sym}}, 
we have for all $i\in [p]$ that 
\begin{equation}\label{asymptoticnormaleigenvalue}
\sqrt{m} \tau_{m,n}^{-1/2}\lambda_{i,0}^{-1}\left(\hat{\Lambda}_{m,n,i}-\Lambda_{n,i}\right)\Rightarrow 
\N\left(0,G_{(ii),(ii)}\right)
\end{equation}
and with $\Delta_i$ from \eqref{defHi},
\begin{equation}\label{asymptoticnormaleigenvector}
\sqrt{m}\tau_{m,n}^{-1/2} \left(\hat{v}_{m,n,i}-v_i\right)
\Rightarrow \N \left(0,\Delta_i(\Sigma, G) \right).
\end{equation}
\end{proposition}

\begin{proof}
From \eqref{symmatform-1}
in Condition \ref{generalquadratic forms} \textup{\texttt{Sym}}
and  $G_n\to G$, we conclude that \eqref{symmatform} holds. 
Also, 
we can consider $\lambda_i$ 
and $v_i$, $i\in[p]$, as functions on 
$\R^{p^2}$;
and since all eigenvalues of $\Sigma$ are distinct,
$\lambda_i$, $v_i$, $i\in[p]$, can be viewed as smooth functions in a neighborhood 
$N\subset \R^{p^2}$ of $\vec(\Sigma)$.  
Given the assumption that $X_{n}^\top X_{n}\to \Sigma$,
it follows that for some $n_0 \in \mathbb{N}$,
$\vec(X_n^\top X_n)\in N$
for all $n \ge n_0$. 
Further, \eqref{symmatform}
implies that
$\vec\left(X_{n}^\top S_{m,n}^\top S_{m,n}X_{n}-X_{n}^\top X_{n}\right)\rightarrow_P 0$, 
so
the events $\Omega_n = \{\vec(X_n^\top S_{m,n}^\top S_{m,n} X_n)\in N\}$ have $P(\Omega_n )\to 1$. 
Now fix $i\in[p]$.
By the mean value theorem, for some $\theta_{m,n}\in [0,1]$,
\begin{align*}
&\lambda_i(X_{n}^\top S_{m,n}^\top S_{m,n}X_{n})-\lambda_i(X_{n}^\top X_{n})=\nabla\lambda_i\left(X_{n}^\top X_{n}+\theta_{m,n}(X_{n}^\top S_{m,n}^\top S_{m,n}X_{n}-X_{n}^\top X_{n}) \right)^\top \\
&\qquad\cdot \vec(X_{n}^\top S_{m,n}^\top S_{m,n}X_{n}-X_{n}^\top X_{n})I(\Omega_n )
+\left(\lambda_i(X_{n}^\top S_{m,n}^\top S_{m,n}X_{n})-\lambda_i(X_{n}^\top X_{n})\right)I(\Omega_n ^c).
\end{align*}

Since $P(\Omega_n ^c)\to 0$,
by Slutsky's theorem, the limiting distribution 
of $\lambda_i(X_{n}^\top S_{m,n}^\top S_{m,n}X_{n})-\lambda_i(X_{n}^\top X_{n})$
is determined by the first term on the event $\Omega_n $. 
Applying the continuous mapping theorem on $\Omega_n $, we have
$$\nabla\lambda_i\left(X_{n}^\top X_{n}+\theta_{m,n}(X_{n}^\top S_{m,n}^\top S_{m,n}X_{n}-X_{n}^\top X_{n}) \right)\rightarrow_P \nabla\lambda_i(\Sigma).$$
Now, 
defining $\bar X_{n}=X_{n}(X_{n}^\top X_n)^{-1/2}\Sigma^{1/2}$, 
we can write 
$$\sqrt{m} (X_{n}^\top S_{m,n}^\top S_{m,n}X_{n}-X_{n}^\top X_{n})=\sqrt{m} (X_{n}^\top X_n)^{1/2}\Sigma^{-1/2}(\bar X_{n}^\top S_{m,n}^\top S_{m,n}\bar X_{n}-\Sigma)\Sigma^{-1/2}(X_{n}^\top X_n)^{1/2}.$$ 
By Slutsky's theorem, 
it is sufficient to characterize the limiting distribution of $\sqrt{m}(\bar X_{n}^\top S_{m,n}^\top S_{m,n}\bar X_{n}-\Sigma)$.

Using the SVD $X_n=U_n L_n V_n^\top$ and that $\Sigma = V^\top L^2 V$,
we can write $\bar X_n=U_n V_n^\top V L V^\top$. 
Denoting $\bar U_n=U_n V_n^\top V$, 
the SVD of $\bar X_n$ is therefore $\bar X_n=\bar U_n L V^\top$. 
Thus,
$\bar X_n^\top S_{m,n}^\top S_{m,n}\bar X_n-\Sigma
=VL(\bar U_n^\top S_{m,n}^\top S_{m,n} \bar U_n$  $-I_p)LV^\top$, 
and
consequently, 
recalling the discussion from
 \Cref{tr},
\beq\label{vt}
\vec(\bar X_n^\top S_{m,n}^\top S_{m,n}\bar X_n-\Sigma)
=
\big((VL) \otimes (VL)\big) \vec(\bar U_n^\top S_{m,n}^\top S_{m,n} \bar U_n-I_p).
\eeq
Additionally, since $X_n^\top X_n \to \Sigma$ and the eigenvalues of $\Sigma$ are distinct,
and as our conventions resolve the sign-ambiguity of eigenvectors,
we have that $V_n \to V$.
Since $(U_n)_{n\ge 1}$ satisfies Condition $\mathfrak{P}'$, 
by Condition \ref{generalquadratic forms} \textup{\texttt{Sym}} 
and applying Slutsky's theorem
to $U_n^\top S_{m,n}^\top S_{m,n}U_n-I_p=V_n^\top V(\bar U_n^\top S_{m,n}^\top S_{m,n}\bar U_n-I_p)V^\top V_n$, 
we have
$\sqrt{m} \tau_{m,n}^{-1/2}\vec(\bar U_n^\top S_{m,n}^\top S_{m,n}\bar U_n$ $-I_p)\Rightarrow\N(0, G)$. Consequently,
by \eqref{vt}, with 
$\rho:=\big((VL) \otimes (VL)\big)G\big((LV^\top )\otimes (LV^\top)\big)$,
we have 
$$\sqrt{m}\tau_{m,n}^{-1/2}\vec(\bar X_{n}^\top S_{m,n}^\top S_{m,n}\bar X_{n}-\Sigma)\Rightarrow\N(0, \rho).$$

Next, we recall 
the gradient
$\nabla\lambda_i(\Sigma)$ and 
the Jacobian
 $Jv_i(\Sigma)$, in a form appropriate for our use.
Based on Theorem 8.9 in \cite{magnus2019matrix}, we know that $\mathrm{d} \lambda_i=v_i^\top \mathrm{d}(X) v_i$. This implies that $\mathrm{d}\lambda_i=(v_i\otimes v_i)^\top \mathrm{d}(\vec(X))$. 
Therefore, we can conclude that $\nabla\lambda_i(\Sigma)=\nu_i\otimes \nu_i$.
Similarly, 
$\mathrm{d}v_i=(\lambda_{i,0} I_p-\Sigma)^\dag \mathrm{d}(X) v_i$, and thus $\mathrm{d}v_i=\left(v_i\otimes (\lambda_{i,0} I_p-\Sigma)^\dag\right)^\top \mathrm{d}(\vec(X))$. Consequently, we have $Jv_i(\Sigma)=\nu_i\otimes \left(\lambda_{i,0} I_p-\Sigma\right)^\dag$.

Applying Slutsky's theorem and combining terms, we find that
\begin{align*}
&\sqrt{m}\tau_{m,n}^{-1/2}\left(\lambda_i(X_{n}^\top S_{m,n}^\top S_{m,n}X_{n})-\lambda_i(X_{n}^\top X_{n})\right)
\Rightarrow
\N\bigg(0,(\nu_i\otimes \nu_i)^\top \rho (\nu_i\otimes \nu_i)\bigg),
\end{align*}
as well as
\begin{align*}
&\sqrt{m}\tau_{m,n}^{-1/2}\left(v_i(X_{n}^\top S_{m,n}^\top S_{m,n}X_{n})-v_i(X_{n}^\top X_{n})\right)
\Rightarrow \N\bigg(0,\big(\nu_i\otimes (\lambda_{i,0} I_p-\Sigma)^\dag\big)^\top \rho \big(\nu_i\otimes (\lambda_{i,0} I_p-\Sigma)^\dag\big)\bigg).
\end{align*}

Since 
$V=(\nu_1,\nu_2,\ldots,\nu_p)$, 
we have $LV^\top \nu_i=Le_i=\lambda_{i,0}e_i$
for each $i\in[p]$. 
Therefore, 
$\big((LV^\top) \otimes (LV^\top)\big)(\nu_i\otimes \nu_i)=(LV^\top \nu_i)\otimes (LV^\top \nu_i)=\lambda_{i,0}^2 e_i\otimes e_i$.
 Thus,
$$(\nu_i\otimes \nu_i)^\top \rho (\nu_i\otimes \nu_i)=\lambda_{i,0}^2(e_i\otimes e_i)^\top G(e_i\otimes e_i)=\lambda_{i,0}^2G_{(ii),(ii)}.$$
This shows \eqref{asymptoticnormaleigenvalue}.
Moreover, $(\lambda_{i,0} I_p-\Sigma)^\dag=V^\top \diag((\lambda_{i,0}-\lambda_{k,0})^{-1}(1-\delta_{ki}))_{k=1}^p V$, 
where we interpret $0/0:=0$.
Also, $LV^\top(\lambda_{i,0} I_p-\Sigma)^\dag=\diag(\sqrt{\lambda_{k,0}}(\lambda_{i,0}-\lambda_{k,0})^{-1}(1-\delta_{ki}))_{k=1}^p V^\top$. 
Denote
$\alpha_k=\sqrt{\lambda_{k,0}}/(\lambda_{i,0}-\lambda_{k,0})$ for all $k\neq i$ and $\alpha_i=0$.
Also let 
$\tilde{V}=\diag(\sqrt{\lambda_{k,0}}(\lambda_{i,0}-\lambda_{k,0})^{-1}(1-\delta_{ki}))_{k=1}^p V^\top=(\alpha_1\nu_1,\ldots,\alpha_p\nu_p)^\top$. 
Then,
\begin{align*}
&\big(\nu_i\otimes (\lambda_{i,0} I_p-\Sigma)^\dag\big)^\top \rho \big(\nu_i\otimes (\lambda_{i,0} I_p-\Sigma)^\dag\big)
=\lambda_{i,0}(e_i\otimes \tilde{V})^\top G(e_i\otimes \tilde{V}).
\end{align*}
We denote by $\tilde{G}^i$
the $p\times p$ principal submatrix of $G$, for which $\tilde{G}^i_{kl}=G_{(ik),(il)}$ for all $k,l\in[p]$.
Then $$
(e_i\otimes \tilde{V})^\top G(e_i\otimes \tilde{V})=\tilde{V}^\top\tilde{G}^i\tilde{V}=(\alpha_1\nu_1,\ldots,\alpha_p\nu_p)\tilde{G}^i(\alpha_1\nu_1,\ldots,\alpha_p\nu_p)^\top=\sum_{k,l}\alpha_k\alpha_l(\tilde{G}^i)_{kl}\nu_k\nu_l^\top.
$$
Hence, 
$$
\lambda_{i,0}(e_i\otimes \tilde{V})^\top G(e_i\otimes \tilde{V})=\sum_{k\neq i, l\neq i}\frac{\lambda_{i,0}\sqrt{\lambda_{k,0}\lambda_{l,0}}}{(\lambda_{i,0}-\lambda_{k,0})(\lambda_{i,0}-\lambda_{l,0})}G_{(ik),(il)}\nu_k\nu_l^\top.
$$
Due to the definition of $\Delta_i$ from \eqref{defHi},
this proves \eqref{asymptoticnormaleigenvector}.
    
\end{proof}

We first show \Cref{infPCAgeneral}
for a sequence of data matrices $(X_n)_{n\ge 1}$ and covariance matrices $(G_n)_{n\ge 1}$ satisfying the conditions of Proposition \ref{delta}.
In that case, by using that 
\eqref{asymptoticnormaleigenvalue} holds due to Proposition \ref{delta},
and from Slutsky's theorem, 
we conclude \eqref{infsingval}.
Next,
recall that 
from the continuous mapping theorem, 
since $X_n^\top X_n\to \Sigma$, $G_n\to G$, and $X_n^\top S_{m,n}^\top S_{m,n} X_n-X_n^\top X_n\rightarrow_P 0$,
we have
$\Lambda_k\to \lambda_{k,0}$, $\hat{\Lambda}_{m,n,k}\rightarrow_P\lambda_{k,0}$,
$v_k\to \nu_k$, and
 $\hat{v}_{m,n,k}\rightarrow_P\nu_k$ for all $k\in[p]$. 
This shows that $\hat{\Delta}_{m,n,i}$ 
is well-defined with probability tending to unity, and $\hat{\Delta}_{m,n,i}\to_P \Delta_i(\Sigma, G)$,
$\Delta_{n,i}\to \Delta_i(\Sigma, G)$. Since $\liminf_{n\to\infty}c^\top \Delta_{n,i}c>0$, there exists $\ep>0$ such that $c^\top \Delta_{n,i}c>\ep$ for all sufficiently large $n\in\NN$, and as a result, $c^\top \Delta_i(\Sigma, G)c\ge\ep>0$. Thus, the left-hand side in \eqref{infsingvec} is well-defined for sufficiently large $n\in\NN$. 
Moreover, using  \eqref{asymptoticnormaleigenvector} and
Slutsky's theorem, 
\eqref{infsingvec} holds under the conditions of Proposition \ref{delta}.

Finally, we use a subsequence argument to show \Cref{infPCAgeneral} under the conditions from its statement.
Without loss of generality we can assume that $\|X_n\| = 1$ by working with $X_n/\|X_n\|$.
Indeed, 
note that $\Delta_i$  from \eqref{defHi} is a homogeneous function with respect to $\Xi$, i.e.,
for any $c >0$,
any
$G\in\R^{p^2\times p^2}$, and all feasible choices of $\Xi\in\Sp$, 
$\Delta_i(\Xi, G)=\Delta_i(c\Xi, G)$.
We observe that the left-hand side of \eqref{infsingval} and \eqref{infsingvec}
remains invariant when we scale the data $X_n$ to $X'_n=X_n/g_n$ for some nonzero scaling sequence $(g_n)_{n\ge 1}$. 
In this case, we have $\hat{\Lambda}'_{m,n,i}=\hat{\Lambda}_{m,n,i}/g_n^2$ and $\Lambda'_{n,i}=\Lambda_{n,i}/g_n^2$, leading to $\hat{\Lambda}_{m,n,i}^{-1}(\hat{\Lambda}_{m,n,i}-\Lambda_{n,i})=(\hat{\Lambda}_{m,n,i}')^{-1}(\hat{\Lambda}'_{m,n,i}-\Lambda'_{n,i})$, and 
due to 
the homogeneity of the function $\Delta_i$, 
we have
$(c^\top\hat{\Delta}_{m,n,i}c)^{-1/2}c^\top(\hat{v}_{m,n,i}-v_{n,i})=(c^\top\hat{\Delta}'_{m,n,i}c)^{-1/2}c^\top(\hat{v}'_{m,n,i}-v'_{n,i})$. 
Hence, from now on, we assume without loss of generality that $\|X_n\| = 1$.

Let
\begin{align}\label{deftw}
& T_n:=X_n^\top X_n,\qquad
W_n:=m^{1/2}\tau_{m,n}^{-1/2}(G_n)_{(ii),(ii)}^{-1/2}\hat{\Lambda}_{m,n,i}^{-1}(\hat{\Lambda}_{m,n,i}-\Lambda_{n,i}),\\ 
&\chi_n:=m^{1/2}\tau_{m,n}^{-1/2}(c^\top \Delta_{n,i}c)^{-1/2} c^\top(\hat{v}_{m,n,i}-v_{n,i}),\qquad
\tilde\chi_n:=m^{1/2}\tau_{m,n}^{-1/2}(c^\top\hat{\Delta}_{m,n,i}c)^{-1/2}c^\top(\hat{v}_{m,n,i}-v_{n,i}).\nonumber
\end{align}

We have that $\|T_n\|=
1$, and the singular values of $T_n$ are $(\ell_{n,i}^2/\ell_{n,1}^2)_{i\in[p]}
= (\ell_{n,i}^2)_{i\in[p]}$. 
Thus, due to Condition \ref{spectralcondition} 
\textup{\texttt{B}},
 there is a compact set $K \subset \Sp$ 
such that for all $n\ge 1$, $T_n \in K$. 
Since there exists $0<c\le C$, 
such that $cI_{p(p+1)/2}\preceq D_p^\top G_n D_p\preceq CI_{p(p+1)/2}$, there is a compact set $\tilde{K}\subset \R^{p(p+1)/2\times p(p+1)/2}$ such that for all $n\ge 1$, $D_p^\top G_n D_p \in \tilde{K}$.

Moreover, \eqref{infsingval} 
is equivalent to 
$d_1(L(W_n),\N(0,1))\to 0$.
Now, consider any subsequence $(n_k)_{k\ge 1}$ of $\NN$.
Since $T_{n_k} \in K, D_p^\top G_{n_k} D_p \in \tilde{K}$
for all $k\ge 1$  and $K, \tilde{K}$ are compact,
there exists a further sub-subsequence
$(n_k')_{k\ge 1}$ 
of $(n_k)_{k\ge 1}$
such that 
$T_{n_k'} \to \Sigma$ and $D_p^\top G_{n_k'} D_p \to D_p^\top GD_p$, for some $\Sigma \in \Sp$ and absolutely symmetric $p^2\times p^2$ matrix $G$. 
Since
all eigenvalues of $T_n$ are uniformly separated
by Condition \ref{spectralcondition} \textup{\texttt{B}},
all eigenvalues of $\Sigma$ are distinct. 
Then $d_1(L(W_{n_k'}),\N(0,1))\to 0$
holds along the subsequence 
$(n_k')_{k\ge 1}$. 
Thus, for any subsequence 
$(n_k)_{k\ge 1}$, 
there exists a further sub-subsequence
$(n_k')_{k\ge 1}$
of $(n_k)_{k\ge 1}$
such that
$d_1(L(W_{n_k'}),\N(0,1))\to 0$.
Therefore, $d_1(L(W_{n}),\N(0,1))\to 0$,
proving \eqref{infsingval}.

The proof of \eqref{infsingvec} is similar,
with a few differences. 
We prove the result for $\chi_n$ first. By $c^\top \Delta_{n,i}c>\ep$ for all $n\ge 1$, for every possible sub-subsequence limit $\Sigma$, $c^\top \Delta_i(\Sigma, G) c\ge \ep>0$.
This shows that     
$\chi_n$ and $\tilde\chi_n$ are well defined for all sufficiently large $n\in\NN$. 
As in the previous proof, $d_1(L(\chi_{n}),\N(0,1))\to 0$.
Since $(c^\top \hat{\Delta}_{m,n,i} c)^{-1}c^\top \Delta_{n,i}c\to_P 1$, we have $d_1(L(\tilde\chi_n),\N(0,1))\to 0$ by Slutsky's theorem. This finishes the proof.

If Condition \ref{generalquadratic forms} \textup{\texttt{Sym-1'}} holds, then $G_{(ii),(ii)}=2+\alpha$ and 
$\Delta_{n,i}=\sum_{k\neq i}\Lambda_{n,i}\Lambda_{n,k}/(\Lambda_{n,i}-\Lambda_{n,k})^2 v_{n,k}v_{n,k}^\top$. Since $\Lambda_{n,i}\Lambda_{n,k}/(\Lambda_{n,i}-\Lambda_{n,k})^2\ge \Lambda_{n,p}^2/\Lambda_{n,1}^2> c_0^2$, 
for some $c_0>0$,
we can conclude $c^\top \Delta_{n,i}c>c_0^2 \ep$
if $\sum_{k\neq i} (c^\top v_{n,k})^2>\ep$, which is equivalent to $(c^\top v_{n,i})^2<1-\ep$. Based on the proof above, \eqref{infsingval'} and \eqref{infsingvec'} hold.

\subsection{Proof of \texorpdfstring{Corollary \ref{rndpca}}{Corollary \ref{rndpca}}}
\label{pfrndpca}

Recall from \Cref{infPCAgeneral} that conditional on $(X_n)_{n\ge 1}$, \eqref{infsingval} holds. 
This implies that we also have the unconditional result
\begin{equation}\label{infPCA1}
\sqrt{\frac{m}{\tau_{m,n}(G_n)_{(ii),(ii)}}}\hat{\Lambda}_{m,n,i}^{-1}\left(\hat{\Lambda}_{m,n,i}-\Lambda_{n,i}\right)\Rightarrow \N(0,1).
\end{equation} 
Moreover, 
from \cite{waternaux1976asymptotic} on inference for one eigenvalue, if we define $\kappa_4^i$ as the kurtosis
of the marginal
distribution
of the $i$-th dimension, we have 
\begin{equation}\label{infPCA2}
\sqrt{\frac{n}{2+\kappa_{4}^i/\lambda_i^{* 2}}}\lambda_i^{* -1}\left(\Lambda_{n,i}-\lambda_i^*\right)\Rightarrow \N(0,1).
\end{equation}
From \eqref{infPCA1} and \eqref{infPCA2}, we have $\Lambda_{n,i}\to_P \lambda_i^*$ and $\hat{\Lambda}_{m,n,i}^{-1}\to_P \lambda_i^*$.
If $m\cdot\left(n\tau_{m,n}(G_n)_{(ii),(ii)}\right)^{-1} $ $\to 0$, from Slutsky's theorem, we can conclude \eqref{ri}.

\subsection{Proof of \texorpdfstring{\Cref{PCAsrht}}{Theorem \ref{PCAsrht}}}
\label{pfPCAsrht}

By Lemma \ref{lemsymqfsrht}, if $(U_n)_{n\ge 1}$ satisfies Condition $\mathfrak{P}'$,
 Condition \ref{generalquadratic forms} \texttt{Sym-1}
 holds with
$G=I_{p^2}+P_p+Q_p$ and $\tau_{m,n}=1-\gamma_n$. Therefore, Condition \ref{generalquadratic forms} \textup{\texttt{Sym-1'}}
 holds with $\alpha=1$.
Then since Condition \ref{spectralcondition} \textup{\texttt{B}} is satisfied, 
we conclude that \eqref{infPCAvalsrht} and \eqref{infPCAvecsrht} hold
due to 
\eqref{infsingval'} and \eqref{infsingvec'} from
\Cref{infPCAgeneral}, respectively.

\subsection{Proof of \texorpdfstring{\Cref{thhadamard}}{Theorem \ref{thhadamard}}}

From Lemma \ref{lemboundqfsrht}, we conclude that for any 
$(b_n)_{n\ge 1}$ such that
$b_n\in\psp$ for all $n\ge 1$, 
$({U_nb_n, U_nb_n})_{n\ge 1}$ satisfies Condition \ref{generalquadratic forms} \textup{\texttt{Bounded}}.
If \eqref{assudelocalized} holds,
$({U_nb_n, \bar{\ep}_n})_{n\ge 1}$ satisfies Condition $\mathfrak{P}$. 
Thus, from Lemma \ref{lemgeneralqfsrht}
applied to 
$(a_{n}, \ta_{n})_{n\ge 1} = ({U_nb_n, \bar{\ep}_n})_{n\ge 1}$, 
Condition \ref{generalquadratic forms} \textup{\texttt{1-dim'}} holds for $({U_nb_n, \bar{\ep}_n})_{n\ge 1}$ with $\alpha=1$.
Then by \Cref{lemgeneols}, we conclude the first result.

Similarly, if \eqref{assudelocalized1} holds,
then from Lemma \ref{lemgeneralqfsrht}, it follows that $({U_nb_n, \widebar{X_n\beta_n}})_{n\ge 1}$ satisfies Condition $\mathfrak{P}$. 
Then the second claim is established similarly to the first one above.

\subsection{Proof of \texorpdfstring{Lemma \ref{lemboundqfsrht}}{Lemma \ref{lemboundqfsrht}}}
\label{pflemboundqfsrht}

Recall that the sketching matrix is defined as
$S_{m,n} = \sqrt{n/m}B H D$, where $H = H_l$ is the Hadamard matrix.
Denote $R_i:=\left(H D a_n\right)_i$ and $\tiR_i:=\left(H D \ta_n\right)_i$ for $i\in[n]$. 
Recall $\gamma_n = m/n$
and $B = \diag(B_1, \ldots, B_n)$, where 
$B_i$ are 
i.i.d.~Bernoulli random variables with a success probability of $\gamma_n^{-1}$. 
Let $h_{kl}$, $k,l\in[n]$ be the entries of $H$, and $D = \diag(D_1, \ldots, D_n)$.
Then, letting $a_n = (a_{n,1}, \ldots, a_{n,n})^\top$ and $\tilde{a}_n = (\tilde{a}_{n,1}, \ldots, \tilde{a}_{n,n})^\top$, we have 
$R_i=\sum_{j\in[n]}  h_{ij}D_j a_{n,j}.$
Moreover, we can write
$
a_n^{\top} S_{m, n}^{\top} S_{m, n} \tilde{a}_n=\sum_{i=1}^n R_i \tiR_i \gamma_n^{-1} B_i$
and thus find that with 
\begin{align}\label{cvh}
& s_n^2:=
\Var{a_n^{\top} S_{m, n}^{\top} S_{m, n} \tilde{a}_n \mid D}
=\frac{1-\gamma_n}{\gamma_n} \sum_{i=1}^n R_i^2 \tiR_i^2,
\end{align}
we have 
\begin{align*}
& \E\left[s_n^2\right]
 =\frac{1-\gamma_n}{\gamma_n} \E \sum_{i=1}^n \sum_{j_1, j_2, j_3, j_4=1}^n h_{i j_1} h_{i j_2} h_{i j_3} h_{i j_4} D_{j_1} D_{j_2} D_{j_3} D_{j_4} a_{n, j_1} a_{n, j_2} \tilde{a}_{n, j_3} \tilde{a}_{n, j_4} \nonumber\\
& =\frac{1-\gamma_n}{\gamma_n} n^{-1}\left(\sum_{j=1}^n a_{n, j}^2 \tilde{a}_{n, j}^2+\sum_{j_1 \neq j_2} a_{n, j_1}^2 \tilde{a}_{n, j_2}^2+2 \sum_{j_1 \neq j_2} a_{n, j_1} \tilde{a}_{n, j_1} a_{n, j_2} \tilde{a}_{n, j_2}\right) \nonumber\\
& =\frac{1-\gamma_n}{\gamma_n} n^{-1}\left(1+2\left(a_n^{\top} \tilde{a}_n\right)^2-2 \sum_{j=1}^n a_{n, j}^2 \tilde{a}_{n, j}^2\right),\nonumber
\end{align*}
where the first step uses that $B_i$ are i.i.d.~Bernoulli$\left(\gamma_n\right)$ variables, the third step uses that $D_i$ are i.i.d.~variables equal to $\pm 1$ with probability $1/2$, and the last step uses that $a_n$ and $\ta_n$ are of unit norm. We also have
$$
\Var{\E\left[a_n^{\top} S_{m, n}^{\top} S_{m, n} \tilde{a}_n \mid D\right]}=\Var{\sum_{i=1}^n R_i \tiR_i}=\Var{a_n^{\top} \tilde{a}_n}=0 .
$$
Combining the above two equations, we have $\Var{a_n^{\top} S_{m, n}^{\top} S_{m, n} \tilde{a}_n}=O\left((1-\gamma_n)/m\right)$, and Condition \ref{generalquadratic forms} \textup{\texttt{Bounded}} follows.

\subsection{Proof of \texorpdfstring{Lemma \ref{lemgeneralqfsrht}}{Lemma \ref{lemgeneralqfsrht}}}

\begin{proof}
For all $i\in [n]$, 
recall the
notations $R_i = (HDa_n)_i$ and $\tiR_i=(HD\ta_n)_i$
from the proof of Lemma \ref{lemboundqfsrht} in \Cref{pflemboundqfsrht}.
Recall that
$a_{n}^\top S_{m,n}^\top S_{m,n} \ta_{n}=\sum_{i=1}^n R_i\tiR_i \gamma_{n}^{-1}B_i$, is a weighted sum of the $n$ independent Bernoulli variables $B_i$, $i\in[n]$, given $D$. 
With $s_n^2$ from \eqref{cvh},
we will show that for a small $\delta>0$, 
\begin{equation}\label{hadalyapunov}
    \frac{\sum_{i=1}^n |R_i \tiR_i|^{2+\delta}\gamma_n^{-2-\delta}\E |B_i-\E B_i|^{2+\delta}}{s_n^{2+\delta}}=o_P(1),
\end{equation}
and then, by the Lyapunov central limit theorem and Lemma \ref{rrp} below, 
the result will follow. 
To verify \eqref{hadalyapunov}, we will show the following bound: 
\begin{equation}\label{bdri}
\mathbb{P}\left\{\max_{1\le  i\le  n}\{|R_i|\vee |\tiR_i|\}\le  \sqrt{\frac{\log n}{n}} \right\} \ge 1-\frac{8}{n}.
\end{equation}
By writing $R_i = \sum_{j=1}^n h_{ij}D_j a_{n,j}$, we can apply Hoeffding's inequality to obtain $\mathbb{P}(|R_i|>t) \le 2 \exp{(-2nt^2)}$
for any $t>0$. Then \eqref{bdri} holds by taking $t=n^{-1/2}\log^{1/2}{n}$ and further taking a union bound over $i\in [n]$.
Define the event 
$$\Xi:=\left\{\max_{1\le  i\le  n}\{|R_i|\vee |\tiR_i|\}\le  (n^{-1}\log n)^{1/2}\right\} \bigcap \left\{1/2< n \sum_{i=1}^n R_i^2 \tiR_i^2 < 4\right\}.$$
By \eqref{bdri} and Lemma \ref{rrp} below, the event $\Xi$ holds with probability tending to one. 
On $\Xi$, we have
$\sum_{i=1}^n (R_i \tiR_i)^{2+\delta} \le  4(n^{-1}\log{n})^{\delta}/n$ and  $s_n^{2+\delta}>[2n\gamma_n/(1-\gamma_n)]^{-1-\delta/2}$. 
We also find 
$\E|B_i -\E B_i|^{2+\delta} = \gamma_n(1-\gamma_n)^{2+\delta}+\gamma_n^{2+\delta}(1-\gamma_n) \le  \gamma_n(1-\gamma_n).$
Combining these bounds, the left side of \eqref{hadalyapunov} is bounded by $C(m^{-1/2}\log n)^{\delta}/(1-\gamma_n)^{\delta/2}$ for some constant $C>0$.
Therefore, since $m/\log^2 n\to\infty$ and $\limsup m/n <1$, we conclude \eqref{hadalyapunov}.

\begin{lemma}\label{rrp}
We have
    \begin{equation}\label{varlimhada}
    n \sum_{i=1}^n R_i^2 \tiR_i^2 -[ 1+2(a_n^\top \ta_n)^2] =o_P(1).
\end{equation}
\end{lemma}

\begin{proof}
Below, all sums over indices such as $j_1,j_2,\cdots$ will range over the index set $[n]$, and we recall $H=(h_{ij})_{i,j\in[n]}$. We start by considering
\begin{equation*}\begin{aligned}
   &\qquad\qquad\qquad\qquad\qquad n\sum_{i=1}^n R_i^2 \tiR_i^2  =  n \sum_{i=1}^n \left(\sum_{j=1}^n h_{ij} D_j a_{n,j}\right)^2 \left(\sum_{j=1}^n h_{ij} D_j \ta_{n,j}\right)^2 \\& =
   n \sum_{i=1}^n\! \left( \sum_{j=1}^n  h_{ij}^2 a_{n,j}^2 + \sum_{j_1 \neq j_2} h_{ij_1}h_{ij_2} D_{j_1}D_{j_2}a_{n,j_1}a_{n,j_2}\right)\!\!  \left( \sum_{j=1}^n  h_{ij}^2 \ta_{n,j}^2 + \sum_{j_1 \neq j_2} h_{ij_1}h_{ij_2} D_{j_1}D_{j_2}\ta_{n,j_1}\ta_{n,j_2}\right)\!.
\end{aligned}\end{equation*}
We can thus decompose $n\sum_{i=1}^n R_i^2 \tiR_i^2$ into the sum of the following terms:
\begin{equation*}\begin{aligned}
    I_1 & = n \sum_{i=1}^n \sum_{j_1,j_2} h_{i j_1}^2 h_{i j_2}^2 a_{n,j_1}^2 \ta_{n,j_2}^2;\qquad
    I_2 =  n \sum_{i=1}^n \left( \sum_{j_1 \neq j_2} h_{ij_1}h_{ij_2} D_{j_1}D_{j_2}a_{n,j_1}a_{n,j_2}\right)\left(\sum_{j=1}^n  \ta_{n,j}^2\right);\\
    I_3 & = n \sum_{i=1}^n \left( \sum_{j_1 \neq j_2} h_{ij_1}h_{ij_2} D_{j_1}D_{j_2}\ta_{n,j_1}\ta_{n,j_2}\right)\left(\sum_{j=1}^n  a_{n,j}^2\right);
\end{aligned}
\label{g}\end{equation*}
and 
\begin{align*}
    I_4 & = n \sum_{i=1}^n \left( \sum_{j_1 \neq j_2}\sum_{j_3\neq j_4} h_{ij_1}h_{ij_2}h_{ij_3}h_{ij_4} D_{j_1}D_{j_2}D_{j_3}D_{j_4}a_{n,j_1}a_{n,j_2}\ta_{n,j_3}\ta_{n,j_4}\right)\\ 
    & = n \sum_{i=1}^n \left(\sum_{(j_1,j_2,j_3,j_4)} \star \; + 2\sum_{j_1=j_3, j_2 = j_4,j_1 \neq j_2} \star \; + 2 \sum_{j_1 = j_3, j_2\neq j_4} \star \;\right) =: I_{41}+I_{42}+I_{43},
\end{align*}
where $\star$ stands for the term inside the sum on the first line, and $(j_1,j_2,j_3,j_4)$ shows that these four indices must take different values. 

To calculate $I_1$, we observe that $h_{ij} = \pm1/\sqrt{n}$ for all $i,j$, so that 
$I_1 = \sum_{j_1,j_2} a_{n,j_1}^2 \ta_{n,j_2}^2  =1$.
To calculate $I_2$ and $I_3$, we change the order of summation and use that $\sum_{i=1}^n h_{ij_1} h_{ij_2} = 0$ when $j_1 \neq j_2$, to find that $I_2 = I_3 = 0$.
Similarly, we find that $I_{43} = 0$.

Next, $I_{42} = 2 \sum_{j_1\neq j_2} a_{n,j_1} a_{n,j_2}\ta_{n,j_1}\ta_{n,j_2}$, since $h_{ij}^2 = 1/n$ and $D_j^2=1$ for any $i,j \in [n]$.
If $a_n \neq \pm\ta_n$,
define $a'_n = (I-a_n a_n^\top) \ta_n/\|(I-a_n a_n^\top) \ta_n\|$, 
otherwise, define $a_n'$ as an arbitrary unit length vector orthogonal to $a_n$,
so that we have
\begin{equation}\label{decomptan}
    \ta_n = a_n a_n^\top \cdot \ta_n +(I_n-a_n a_n^\top)\cdot \ta_n
    =(a_n^\top \ta_n)a_n+(1-(a_n^\top \ta_n)^2)^{1/2}a'_n.
\end{equation}
Moreover, denote $t_1 = a_n^\top \ta_n, t_2=(1-(a_n^\top \ta_n)^2)^{1/2}$. 
Then, we can express $I_{42}$ as
\begin{equation*}\begin{aligned}
    &I_{42} = 2 \sum_{j_1 \neq j_2} a_{n,j_1} a_{n,j_2} (t_1 a_{n,j_1}+t_2 a'_{n,j_1})(t_1 a_{n,j_2}+t_2 a'_{n,j_2})\\&= 2\sum_{j_1\neq j_2} t_1^2  a_{n,j_1}^2a_{n,j_2}^2+2\sum_{j_1\neq j_2}t_1 t_2 a_{n,j_1}a_{n,j_2}(a_{n,j_1}a'_{n,j_2}+a'_{n,j_1}a_{n,j_2})+2\sum_{j_1\neq j_2}t_2^2 a_{n,j_1}a_{n,j_2}a'_{n,j_1}a'_{n,j_2}\\
    &=:I_{42}^{(1)}+I_{42}^{(2)}+I_{42}^{(3)}.
\end{aligned}\end{equation*}
Due to Condition $\mathfrak{P}$ defined in \eqref{assmaxo1}, without loss of generality, 
we can assume that 
$\max_{1\le  i\le  n} $ $ |a_{n,i}| \to 0$. 
Hence, we have \begin{equation*}
    I_{42}^{(1)} =2 t_1^2 \left[\biggl(\sum_{j=1}^n a_{n,j}^2\biggr)^2-\sum_{j=1}^n a_{n,j}^4\right]=2 t_1^2\left(1+o(1)\right);
\end{equation*}
and since $a_n'$ is orthogonal to $a_n$,
\begin{equation*}
    I_{42}^{(3)}=2t_2^2  \left[\biggl(\sum_{j=1}^n a_{n,j}a'_{n,j}\biggr)^2-\sum_{j=1}^n a_{n,j}^2(a'_{n,j})^2\right] = o(t_2^2).
\end{equation*}
Moreover, we find \begin{equation*}
    \left|I_{42}^{(2)}\right| = \left|-4\sum_{j=1}^nt_1 t_2 a_{n,j}^3a'_{n,j}\right|\le  4 t_1 t_2 \max_{1\le  j\le  n}{|a_{n,j}a'_{n,j}|}\cdot \sum_{j=1}^n a_{n,j}^2 = o(t_1 t_2),
\end{equation*}
Therefore, $I_{42}=2(a_n^\top \ta_n)^2 + o(1)$.

To analyze $I_{41}$, we further introduce $t_{j_1,j_2,j_3,j_4} = \sum_{i=1}^n h_{ij_1}h_{ij_2}h_{ij_3}h_{ij_4} $ 
for any indices $j_1,j_2,j_3,j_4 \in [n]$,
and express $$I_{41} =n \sum_{(j_1,j_2,j_3,j_4)} t_{j_1,j_2,j_3,j_4} D_{j_1} D_{j_2} D_{j_3} D_{j_4}a_{n,j_1}a_{n,j_2}\ta_{n,j_3}\ta_{n,j_4}.$$ 
Then, we consider \begin{equation*}
    \E I_{41}^2 = n^2 \E \sum_{(j_1,\ldots,j_4),(j_5,\ldots,j_8)} t_{j_1,\ldots,j_4}t_{j_5,\ldots,j_8}D_{j_1}\ldots D_{j_8} a_{n,j_1} a_{n,j_2} \ta_{n,j_3}\ta_{n,j_4} a_{n,j_5} a_{n,j_6} \ta_{n,j_7}\ta_{n,j_8}.
\end{equation*}
Now,
only terms such that $\{j_1,j_2,j_3,j_4\}=\{j_5,j_6,j_7,j_8\}$
have a non-zero expectation.
Thus
after decomposing $\ta_{n}$ 
into two orthogonal vectors as in handling $I_{42}$ above, using  Lemma \ref{4i} and the assumption that $\max_{i\in[n]} |a_{n,i}| \to  0$, we conclude that the above expectation is $o(1)$. Thus $I_{41} = o_P(1)$ and we conclude \eqref{varlimhada} by summing up the above bounds.
    
\end{proof}
\begin{lemma}\label{4i}
For any $j_1,j_2,j_3 \in[n]$ there is a unique $j_4$  such that $t_{j_1,j_2,j_3,j_4} \neq 0$, 
and then $t_{j_1,j_2,j_3,j_4} = 1/n$.  
\end{lemma}
\begin{proof}
With $l=\log_2 n\in\NN$, we define the binary representation of $j_i-1$ by 
$j_{i,l-1}2^{l-1} + \ldots + j_{i,2} 2^2 + j_{i,1} 2 +j_{i,0} 1$, 
and define the binary representation $(\ell_{l-1}, \ldots, \ell_0)$ of $\ell-1$ similarly.
Then 
$ h_{j_i, \ell} = n^{-1/2} (-1)^{\sum_{k=0}^{l -1}j_{i,k} \ell_{k}}$,
and it follows that \begin{equation*}
    t_{j_1,j_2,j_3,j_4} = \frac{1}{n^2}\sum_{\ell=1}^n (-1)^{ \sum_{k=0}^{l -1} \sum_{i=1}^4 j_{i,k} \ell_{k}}.
\end{equation*}
We claim the following two facts: First, if $\sum_{i=1}^4 j_{i,k}$ is even for any $k\in[0: l]:=\{0,1,\ldots,l\}$, then  
$t_{j_1,j_2,j_3,j_4}=1/n$.  
Second, if $\sum_{i=1}^4 j_{i,k_0}$ is odd for some $k_0\in [0:l]$, then $t_{j_1,j_2,j_3,j_4}=0$. 
These two claims imply Lemma \ref{4i}.

Next, we prove the two claims.
The first claim holds by noting that $\sum_{k=0}^{l -1} \sum_{i=1}^4 j_{i,k} \ell_{k}$ is even for any $\ell \in [n]$. To conclude the second claim, we write \begin{equation*}
    t_{j_1,j_2,j_3,j_4} = \frac{1}{n^2}\sum_{\ell=1}^n (-1)^{\ell_{k_0}}\cdot (-1)^{ \sum_{k=0,k\neq k_0}^{l -1} \sum_{i=1}^4 j_{i,k} \ell_{k}}=: \frac{1}{n^2}\sum_{\ell=1}^n g_{\ell}.
\end{equation*}
For any $\ell_1\in [n]$, we can find a unique $\ell_2\in [n]$ such that $\ell_{1,k_0}+\ell_{2,k_0}=1$ and the other coefficients of their binary representations match. For such a pair of $\ell_1$ and $\ell_2$, we have 
\begin{equation*}\begin{aligned}g_{\ell_1}+g_{\ell_2}&=(-1)^{\ell_{1,k_0}}\cdot (-1)^{ \sum_{k=0,k\neq k_0}^{l -1} \sum_{i=1}^4 j_{i,k} \ell_{1,k}} + (-1)^{\ell_{2,k_0}}\cdot (-1)^{ \sum_{k=0,k\neq k_0}^{l -1} \sum_{i=1}^4 j_{i,k} \ell_{2,k}} \\&= \left[(-1)^{\ell_{1,k_0}}+(-1)^{\ell_{2,k_0}}\right] \cdot (-1)^{ \sum_{k=0,k\neq k_0}^{l -1} \sum_{i=1}^4 j_{i,k} \ell_{1,k}} =0.\end{aligned}\end{equation*} 
Moreover, we can find $n/2$ such pairs that run through $[n]$, thus $t_{j_1,j_2,j_3,j_4} = n^{-2}\sum_{\ell=1}^n g_{\ell}=0$.
\end{proof}

{Finally we provide an example to show that \eqref{hadafixnormal} is generally not true if Condition $\mathfrak{P}$ define\eqref{assmaxo1} is not satisfied.
Let $n\ge 4$ and 
consider $a_n = (1/2,1/2,-1/2,-1/2, 0, \ldots, 0)^\top$ and $\ta_n=(-1/2,1/2,-1/2,1/2, 0\ldots 0)^\top$. 
For all $i\in[n]$, 
denote the $i$-th column of $H$ by $h_i$. By the structure of the Hadamard matrix, 
for any diagonal matrix $B$, we have $h_1^\top B h_2 = h_3^\top B h_4$, $h_1^\top B h_3 = h_2^\top B h_4$, and $h_1^\top B h_4 = h_2^\top B h_3$. Additionally, we have $h_i^\top B h_i = \sum_{i=1}^n B_i/n$. 
Then, we can calculate: \begin{equation*}\begin{aligned}
&a_n^\top D H B H D \ta_n =\sum_{i=1}^4 a_{n,i} \ta_{n,i} h_i^\top B h_i \\& + [(a_{n,1} \ta_{n,2}+ a_{n,2} \ta_{n,1}) D_1 D_2 + (a_{n,3} \ta_{n,4}+ a_{n,4} \ta_{n,3}) D_3 D_4] h_1^\top B h_2 \\& +  [(a_{n,1} \ta_{n,3}+a_{n,3} \ta_{n,1}) D_1 D_3 + (a_{n,2} \ta_{n,4}+a_{n,4} \ta_{n,2}) D_2 D_4] h_1^\top B h_3 \\&+ [(a_{n,1} \ta_{n,4}+a_{n,4} \ta_{n,1}) D_1 D_4 + (a_{n,2} \ta_{n,3}+a_{n,3} \ta_{n,2}) D_2 D_3) h_1^\top B h_4 
& = \frac{1}{2}(D_1 D_4 -D_2 D_3) h_1^\top B h_4.
\end{aligned}\end{equation*}
Since $D_i$ are uniformly distributed over $\pm 1$ for all $i\in [4]$,
the above random variable places a non-vanishing point mass at zero, 
and is thus not asymptotically normal.}
\end{proof}

\subsection{Proof of \texorpdfstring{Lemma \ref{lemsymqfsrht}}{Lemma \ref{lemsymqfsrht}}}
With the Cramer-Wold device, 
since 
$U_{n}^\top S_{m,n}^\top S_{m,n}U_{n}-I_p$ is a symmetric matrix, 
it suffices to show
that for any symmetric $\Phi\in\R^{p\times p}$,
\begin{equation}\label{cw_vec_srht_PCA}
\sqrt{m} \tr\left(\Phi(U_n^\top S_{m,n}^\top S_{m,n}U_n-I_p)\right)\Rightarrow \N\left(0,\vec(\Phi)^\top G\  \vec(\Phi)\right).
\end{equation}
By the spectral decomposition 
$\Phi=\sum_{i=1}^p \mu_i w_iw_i^\top$
 of $\Phi$, where $(w_i)_{i\in [p]}$ forms an orthonormal basis of $\R^p$, 
we have
$$
\begin{aligned}
&\tr\left(\Phi(U_n^\top S_{m,n}^\top S_{m,n}U_n-I_p)\right)
=\sum_{i=1}^p \mu_i (w_i^\top U_n^\top S_{m,n}^\top S_{m,n}U_n w_i-1).
\end{aligned}
$$
Now $U_n w_i$, $i\in[p]$ 
constitutes an orthogonal basis of the $p$-dimensional linear subspace of $\R^n$ given by the span of the columns of $U_n$. 
Since $(U_n)_{n\ge 1}$ satisfies Condition $\mathfrak{P}'$,  
we have 
$\|U_n w_i\|_{\infty}\le \max_{i=1,\ldots, n} \|U_{i:}\| \|w_i\| \to0$ for every $i\in[p]$,
due to Condition $\mathfrak{P}'$ from \eqref{l_inftydelocalize}. 
We now denote
$R_i^k:=(HDU_nw_k)_i$ for $i\in[n], k\in [p]$, so that $$\sum_{i=1}^p \mu_i w_i^\top U_n^\top S_{m,n}^\top S_{m,n}U_n w_i=\sum_{j=1}^n \left(\sum_{i=1}^p \mu_i\cdot\left(R_j^i\right)^2\right)\gamma_n^{-1}B_j.$$ As in the proof of Lemma \ref{lemgeneralqfsrht},
conditional on $D$, this is a weighted sum of $n$ independent Bernoulli variables, with $D$-conditional variance
$s_n^2=\sum_{j=1}^n \frac{1-\gamma_n}{\gamma_n} \left(\sum_{i=1}^p \mu_i\cdot\left(R_j^i\right)^2\right)^2$.
We write $n\sum_{j=1}^n \left(\sum_{i=1}^p \mu_i\cdot\left(R_j^i\right)^2\right)^2 = n \sum_{i_1,i_2}\sum_{j=1}^n \mu_{i_1}\mu_{i_2}\cdot \left(R_{j}^{i_1}\right)^2
\left(R_{j}^{i_2}\right)^2$.
Since $(w_j)_{j\in [p]}$ are orthonormal, we have 
according to \eqref{varlimhada} that
for $i_1\neq i_2 \in [p]$, 
it holds that $\sum_{j=1}^n  \left(R_{j}^{i_1}\right)^2\left(R_{j}^{i_2}\right)^2 \rightarrow_P 1$, and for $i_1 = i_2\in [p]$, 
it holds that $\sum_{j=1}^n  \left(R_{j}^{i_1}\right)^2\left(R_{j}^{i_2}\right)^2 \rightarrow_P 3$. 
Therefore, 
\begin{equation*}
    n\sum_{j=1}^n \left(\sum_{i=1}^p \mu_i\cdot\left(R_j^i\right)^2\right)^2- 
    \biggl(2\|\mu\|^2 +\biggl(\sum_{i=1}^p\mu_i\biggr)^2\biggr) = o_P(1),
\end{equation*}
Now, since $\mu_i$ are the eigenvalues of $\Phi$,
$2\|\mu\|^2 =2\|\Phi\|_F^2=\vec(\Phi)^\top(I_p+P_p)\vec(\Phi)$.
Also, 
$$\vec(\Phi)^\top Q_p \vec(\Phi)=\sum_{(ij),(kl)}\Phi_{ij}\Phi_{kl}\delta_{ij}\delta_{kl}
=\sum_{i,k}\Phi_{ii}\Phi_{kk}
=(\tr\Phi)^2=\biggl(\sum_{i=1}^p\mu_i\biggr)^2.$$
Thus, $m s_n^2/(1-\gamma_n)$ tends to $\vec(\Phi)^\top (I_{p^2}+P_p+Q_p)\vec(\Phi)$.

Moreover,
we can 
verify 
similarly to the proof of \eqref{hadalyapunov}
that for a small $\delta>0$
\begin{equation*}
    \sum_{j=1}^n \left(\sum_{i=1}^p \mu_i\cdot\left(R_j^i\right)\right)^{2+\delta}\gamma_n^{-2-\delta}\E |B_i-\E B_i|^{2+\delta}s_n^{-(2+\delta)}=o_P(1),
\end{equation*}
assuming that $m/\log^2 n\to\infty$. This can be verified by Hoeffding's inequality similar to
how it is used above in the proof for \eqref{hadalyapunov}. Thus, we conclude \eqref{cw_vec_srht_PCA}.

\subsection{Proof of \texorpdfstring{Theorem \ref{PCAcs}}{Theorem \ref{PCAcs}}}
From Lemma \ref{lemsymqfcs}, if $(U_n)_{n\ge 1}$ satisfies Condition $\mathfrak{P}'$,
 Condition \ref{generalquadratic forms} \texttt{Sym-1}
 holds with $G=I_{p^2}+P_p$ and $\tau_{m,n}=1$. 
Therefore, Condition \ref{generalquadratic forms} \textup{\texttt{Sym-1'}}
 holds with $\alpha=0$.
Then, since Condition \ref{spectralcondition} \textup{\texttt{B}} holds by assumption, 
we conclude the desired results based on \Cref{infPCAgeneral}.

\subsection{Proof of \texorpdfstring{Theorem \ref{cs}}{Theorem \ref{cs}}}
Due to \eqref{olscsdelocal1}, 
we conclude that
for any $b_n\in\psp$, 
$({U_nb_n, \bar{\ep}_n})_{n\ge 1}$ and $({U_nb_n, U_nb_n})_{n\ge 1}$ satisfy Condition $\mathfrak{P}$. 
Thus, from Lemma \ref{lemgeneralqfcs}
Condition \ref{generalquadratic forms} \textup{\texttt{1-dim'}} holds for $({U_nb_n, \bar{\ep}_n})_{n\ge 1}$ with $\alpha=0$.
Then by \Cref{lemgeneols}, we conclude the first result.
Similarly, if \eqref{olscsdelocal2} holds, $({U_nb_n, \widebar{X_n\beta_n}})_{n\ge 1}$ satisfies Condition $\mathfrak{P}$. 
Then the second claim follows similarly from Lemma \ref{lemgeneralqfcs}. 

\subsection{Proof of \texorpdfstring{Lemma \ref{lemgeneralqfcs}}{Lemma \ref{lemgeneralqfcs}}}
For all $i\in [n]$,
we denote the $i$-th column of $S_{m,n}$ as $R_{m,i}$. 
Then,
we have $R_{m,i} =_d \bar{D}_i G_{m,i}$, where $\bar{D}_i=\diag(D_{1,i},\ldots,D_{m,i})$, $D_{m,i}$ are i.i.d.~Rademacher random variables, and $G_{m,i}$ are uniformly distributed on the set $\{f/\sqrt{\zeta_m}\in\R^m:f_1,\ldots,f_m\in\{0,1\}\text{ and }\|f\|^2=\zeta_m\}$, independently across $i$.
We can express $a_n^\top S_{m,n}^\top S_{m,n}\ta_n = \sum_{i=1}^n a_{n,i}\ta_{n,i} + \sum_{i\neq j} a_{n,i}\ta_{n,j} 
R_{m,i}^\top R_{m,j}.$

Defining 
$h_n(x,y)=x^\top y$ for $x,y\in \R^n$, we have
$$\sum_{i\neq j} a_{n,i}\ta_{n,j} R_{m,i}^\top R_{m,j} 
=\sum_{i<j} (a_{n,i}\ta_{n,j}+a_{n,j}\ta_{n,i}) h_n(R_{m,i}, R_{m,j}).$$ The variables $T_i:=R_{m,i}$ are i.i.d.~across $i \in [n]$, 
making the sum a weighted U-statistic of order two. 
Moreover,
for $i\neq j$,
$\E[h_n(T_i, T_j)|T_i]=(R_{m,i})^\top \E[R_{m,j}]=0$ and $\E[h_n(T_i, T_j)|T_j]=0$. 
This shows that the U-statistic is degenerate.

We use
the martingale central limit theorem,
see e.g., Corollary 3.1 from \cite{hall2014martingale},  to derive our result. In particular, we focus on analyzing the sum $\sum_{i<j} w_{ij} h_n(T_i, T_j)$, where $w_{ij}=a_{n,i}\ta_{n,j}+a_{n,j}\ta_{n,i}$ for all $i,j$.
To analyze this sum, we construct the sequence
\begin{equation}\label{weightedUstat}
\Upsilon_k=\sum_{j=1}^{k-1} w_{jk} h_n(T_j, T_k),\, k\in[n].
\end{equation} 
Due to the independence between $T_j$ and $T_k$ for all $j\neq k$, 
we have $\E[\Upsilon_k|T_1,\ldots,T_{k-1}]=0$ for all $k\in[n]$. 
Consequently, the sequence $\{S_i=\sum_{l=2}^i \Upsilon_l, 2\le  i \le  n\}$ forms a martingale with respect to the $\sigma$-field generated by $T_1,\ldots, T_i$, $i\in[n]$. Moreover, we can write $S_n=\sum_{i<j} w_{ij} h_n(T_i, T_j)$.

To show that $S_n/\sqrt{\Var{S_n}}$ tends to a standard normal distribution, 
it is enough to show the Lyapunov condition
$
s_n^{-4}\sum_{i=2}^{n}\E \Upsilon_i^4\rightarrow 0.
$
Additionally, we need to show that
$
s_{n}^{-2}q_n^{2}\rightarrow_P 1,
$
where $q_n^{2}:=\sum_{i=2}^{n}\E[\Upsilon_i^{2}|T_{1},\ldots,T_{i-1}].$

\subsubsection*{Proof of the Lyapunov condition}

First we show the
Lyapunov condition
$
s_n^{-4}\sum_{i=2}^{n}\E \Upsilon_i^4\rightarrow 0.
$
Using the properties of square-integrable martingales, we have
$
s_n^2=\sum_{i=2}^n \E \Upsilon_i^2$, 
and
\begin{align*}
\E \Upsilon_k^2
&=\E \left(\sum_{j=1}^{k-1} w_{jk} h_n(T_j, T_k)\right)^2
=\sum_{j=1}^{k-1} \sum_{i=1}^{k-1} w_{ik}w_{jk}\E[h_n(T_j,T_k)h_n(T_i,T_k)]\\
&=\sum_{j=1}^{k-1} \sum_{i=1}^{k-1} w_{ik}w_{jk}\E[R_{m,j}^\top R_{m,k} R_{m,i}^\top R_{m,k}]
=\sum_{i=1}^{k-1}w_{ik}^2\E[(R_{m,i}^\top R_{m,k})^2].
\end{align*}
Here, for $i\neq j$, we use the independence of $R_{m,i}$, $R_{m,j}$ and $R_{m,k}$. 
Also, since 
$D_{m,i}$ are i.i.d.~Rademacher random variables, and in particular  $D_{k,i}^2=1$ for all $k,i$
\begin{align*}
&\E[(R_{m,i}^\top R_{m,k})^2]=\E[(G_{m,i}^\top \bar{D}_i\bar{D}_kG_{m,k})^2]=\E[G_{m,i}^\top \bar{D}_i\bar{D}_kG_{m,k}G_{m,k}^\top \bar{D}_i\bar{D}_kG_{m,i}]\\
&=\E[G_{m,i}^\top \diag(G_{m,k}\odot G_{m,k})G_{m,i}].
\end{align*}
Since
$G_{m,i}$, $G_{m,j}$ are independently and uniformly distributed on the set $\{f/\sqrt{\zeta_m}\in\R^m:f_1,\ldots,f_m\in\{0,1\}\text{ and }\|f\|^2=\zeta_m\}$,
every component of $G_{m,i}$ equals to zero or $1/\sqrt{\zeta_m}$,
and hence $G_{m,i}\odot G_{m,i} = G_{m,i}/\sqrt{\zeta_m}$ for all $i$.
Further,
$ \E[G_{m,k}] = \sqrt{\zeta_m} 1_m/m$,
and 
we have for $i\neq k$ that
\begin{align*}
&\E[G_{m,i}^\top \diag(G_{m,k}\odot G_{m,k})G_{m,i}]
=\E[(G_{m,i}\odot G_{m,i})^\top (G_{m,k}\odot G_{m,k})]=\E[G_{m,i}^\top] \E[G_{m,k}]/\zeta_m=1/m.
\end{align*}
Therefore,
\begin{align}\label{sn2}
s_n^2
&=\sum_{k=2}^n \E \Upsilon_k^2=\sum_{k=2}^n \frac{1}{m}\sum_{i=1}^{k-1}w_{ik}^2=\frac{1}{m}\sum_{i<k}w_{ik}^2=\frac{1}{m}\sum_{i<k}\left(a_{n,i}\ta_{n,k}+a_{n,k}\ta_{n,i}\right)^2\nonumber\\
&=\frac{1}{m}\sum_{i<k}\left(a_{n,i}^2\ta_{n,k}^2+a_{n,k}^2\ta_{n,i}^2+2a_{n,i}\ta_{n,i} a_{n,k}\ta_{n,k}\right)=\frac{1}{m}\sum_{i\neq k}\left(a_{n,i}^2\ta_{n,k}^2+a_{n,i}\ta_{n,i} a_{n,k}\ta_{n,k}\right)\nonumber\\
&=\frac{1}{m}\left(\left(\sum_{i=1}^n a_{n,i}^2\right)\left(\sum_{i=1}^n \ta_{n,i}^2\right)-\sum_{i=1}^n a_{n,i}^2\ta_{n,i}^2+\left(\sum_{i=1}^n a_{n,i}\ta_{n,i}\right)^2-\sum_{i=1}^n a_{n,i}^2\ta_{n,i}^2\right)\nonumber\\
&=\frac{1}{m}\left(1+(a_n^\top \ta_n)^2-2\sum_{i=1}^n a_{n,i}^2\ta_{n,i}^2\right).
\end{align}

Additionally, $\E \Upsilon_k^4$ equals
\begin{align*}
&\E \left[\sum_{j=1}^{k-1} w_{jk} h_n(T_j, T_k)\right]^4
=\sum_{i_1, i_2,i_3,i_4=1}^{k-1} w_{i_1 k}w_{i_2 k}w_{i_3 k}w_{i_4 k}\E\left[h_n\left(T_{i_1},T_k\right)h_n\left(T_{i_2},T_k\right)h_n\left(T_{i_3},T_k\right)h_n\left(T_{i_4},T_k\right)\right]\\
&=\sum_{i_1, i_2,i_3,i_4=1}^{k-1} w_{i_1 k}w_{i_2 k}w_{i_3 k}w_{i_4 k}
\E\left[R_{m,i_1}^\top R_{m,k}R_{m,i_2}^\top R_{m,k}R_{m,i_3}^\top R_{m,k}R_{m,i_4}^\top R_{m,k}\right].
\end{align*}
Due to independence and as $\E[R_{m,i}] = 0$ for all $i\in[k-1]$,
if an index occurs only once in $(i_1,i_2,i_3,i_4)$, 
the corresponding term is zero. 
Therefore, the sum becomes
$$
\sum_{i=1}^{k-1} w_{ik}^4 \E\left[\left(R_{m,i}^\top R_{m,k}\right)^4\right]+3 \sum_{i=1}^{k-1}\sum_{j=1}^{k-1} (1-\delta_{ij}) w_{ik}^2 w_{jk}^2 \E\left[\left(R_{m,i}^\top R_{m,k}\right)^2 \left(R_{m,j}^\top R_{m,k}\right)^2\right].
$$
Define $M_{\zeta_m}\in\R^{m\times m}$, 
with $(M_{\zeta_m})_{ii}=1$ for $i\in[m]$ and $(M_{\zeta_m})_{ij}=(\zeta_m-1)/(m-1)$ for $i\neq j, i,j\in[m]$. Then one can verify 
that for all $k$,
$\E[G_{m,k}G_{m,k}^\top]=M_{\zeta_m}/m$. Therefore, 
since $D_{k,i}^2=1$ for all $k,i$,
and $G_{m,i}\odot G_{m,i} = G_{m,i}/\sqrt{\zeta_m}$ for all $i$
we have 
for all distinct $ i,j,k$ that 
\begin{align*}
&\E\left[(R_{m,i}^\top R_{m,k})^2 (R_{m,j}^\top R_{m,k})^2\right]
=\E\left[(G_{m,i}^\top \bar{D}_{i}\bar{D}_{k}G_{m,k})^2(G_{m,j}^\top \bar{D}_{j}\bar{D}_{k}G_{m,k})^2\right]\\
&=\frac{1}{\zeta_m^2}
\E\left[G_{m,i}^\top G_{m,k}G_{m,k}^\top G_{m,j}\right]=\frac{1}{\zeta_m^2}\E\left[G_{m,i}\right]^\top \E\left[G_{m,k}G_{m,k}^\top\right] \E\left[G_{m,j}\right]
=\frac{1}{\zeta_m^3} \left(\frac{\zeta_m}{m}\right)^2\frac{1_m^\top M_{\zeta_m} 1_m}{m}=\frac{1}{m^2}.
\end{align*}
In addition,
for all $i\neq k$,
again since $G_{m,i}\odot G_{m,i} = G_{m,i}/\sqrt{\zeta_m}$ for all $i$
\begin{align*}
\E\left[\left(R_{m,i}^\top R_{m,k}\right)^4\right]&=\E\left[\left(G_{m,i}^\top \bar{D}_{i}\bar{D}_{k}G_{m,k}\right)^4\right]=\E\left[\left(\sum_{l=1}^mD_{l,i}D_{l,k}\left(G_{m,i}\right)_l \left(G_{m,k}\right)_l \right)^4\right]\\
&=\E\left[\sum_{l=1}^m\left(G_{m,i}\right)_l^4 \left(G_{m,k}\right)_l^4\right]+6\E\left[\sum_{l_1<l_2}\left(G_{m,i}\right)_{l_1}^2 \left(G_{m,k}\right)_{l_1}^2\left(G_{m,i}\right)_{l_2}^2 \left(G_{m,k}\right)_{l_2}^2\right]\\
&=\frac{1}{\zeta_m^3}\E\left[\sum_{l=1}^m\left(G_{m,i}\right)_l \left(G_{m,k}\right)_l\right]+\frac{3}{\zeta_m^2} \E\left[\sum_{l_1\neq l_2}\left(G_{m,i}\right)_{l_1} \left(G_{m,k}\right)_{l_1}\left(G_{m,i}\right)_{l_2} \left(G_{m,k}\right)_{l_2}\right]\\
&=\frac{1}{\zeta_m^3}\E\left[G_{m,i}^\top G_{m,k}\right]
+\frac{3}{\zeta_m^2} \E\left[\left(\sum_{l=1}^m \left(G_{m,i}\right)_{l} \left(G_{m,k}\right)_{l}\right)^2-\sum_{l=1}^m\left(G_{m,i}\right)_{l}^2 \left(G_{m,k}\right)_{l}^2\right].
\end{align*}
This further equals
\begin{align*}
&\frac{1}{\zeta_m^3}\E\left[G_{m,i}^\top G_{m,k}\right]
+\frac{3}{\zeta_m^2} \E\left[\left(\sum_{l=1}^m \left(G_{m,i}\right)_{l} \left(G_{m,k}\right)_{l}\right)^2\right]
-
\frac{3}{\zeta_m^3} \E\left[
\sum_{l=1}^m(G_{m,i})_{l} (G_{m,k})_{l}\right]\\
&=-\frac{2}{m\zeta_m^2}+\frac{3}{\zeta_m^2} \E\left[\E\left[G_{m,i}^\top G_{m,k}G_{m,k}^\top G_{m,i}|G_{m,i}\right]\right]\\
&=-\frac{2}{m\zeta_m^2}+\frac{3}{\zeta_m^2} \E\left[\sum_{l=1}^m \left(G_{m,i}\right)_{l}^2+\frac{\zeta_m-1}{m-1}\sum_{l_1\neq l_2}\left(G_{m,i}\right)_{l_1}\left(G_{m,i}\right)_{l_2}\right]\\
&=-\frac{2}{m\zeta_m^2}+\frac{3}{m\zeta_m^2}\left(\frac{m}{\zeta_m}\cdot \frac{\zeta_m}{m}+\frac{\zeta_m-1}{m-1}\cdot \frac{m(m-1)}{\zeta_m} \cdot \frac{\zeta_m(\zeta_m-1)}{m(m-1)}\right)
=\frac{m-1+3(\zeta_m-1)^2}{m(m-1)\zeta_m^2}.
\end{align*}
Hence, we can express $\E \Upsilon_k^4$ as
$$
\E \Upsilon_k^4=\frac{m-1+3(\zeta_m-1)^2}{m(m-1)\zeta_m^2}\sum_{i=1}^{k-1} w_{ik}^4+\frac{3}{m^2}\sum_{i=1}^{k-1}\sum_{j=1}^{k-1} (1-\delta_{ij}) w_{ik}^2 w_{jk}^2.
$$
Therefore, 
$$
\frac{\sum_{k=2}^n \E \Upsilon_k^4}{s_n^4}=\frac{\frac{m(m-1+3(\zeta_m-1)^2)}{(m-1)\zeta_m^2}\sum_{i<k} w_{ik}^4+3\sum_{k=2}^n\sum_{i=1}^{k-1}\sum_{j=1}^{k-1} (1-\delta_{ij}) w_{ik}^2 w_{jk}^2}{\left(1+(a_n^\top \ta_n)^2-2\sum_{i=1}^n a_{n,i}^2\ta_{n,i}^2\right)^2}.
$$
Thus, we only need to show that
\begin{equation}\label{martprof4momentto01}
\frac{m(m-1+3(\zeta_m-1)^2)}{(m-1)\zeta_m^2}\sum_{i<k} w_{ik}^4\to 0,\quad\text{ and }\quad
\sum_{k=2}^n\sum_{i=1}^{k-1}\sum_{j=1}^{k-1} w_{ik}^2 w_{jk}^2\to 0,
\end{equation}
where recall that  $w_{ij}=a_{n,i}\ta_{n,j}+a_{n,j}\ta_{n,i}$ for all $i,j$.
Using the Cauchy-Schwarz inequality, we find $w_{ij}^2=(a_{n,i}\ta_{n,j}+a_{n,j}\ta_{n,i})^2\le (a_{n,i}^2+\ta_{n,i}^2)(a_{n,j}^2+\ta_{n,j}^2)$. Consequently,
\begin{align*}
\sum_{i<k} w_{ik}^4
&\le  \sum_{i,k} (a_{n,i}^2+\ta_{n,i}^2)^2(a_{n,k}^2+\ta_{n,k}^2)^2= \left(\sum_{i=1}^n (a_{n,i}^2+\ta_{n,i}^2)^2\right)^2
\le  \left(2\sum_{i=1}^n (a_{n,i}^4+\ta_{n,i}^4)\right)^2.
\end{align*}
Based on
Condition $\mathfrak{P}$ from \eqref{csstate1} and 
since $\zeta_m^2/m\to 0$, we find
$\sqrt{\left(m-1+3(\zeta_m-1)^2\right)/\zeta_m^2}\sum_{i=1}^n a_{n,i}^4\rightarrow 0$ and $\sqrt{\left(m-1+3(\zeta_m-1)^2\right)/\zeta_m^2}\cdot\sum_{i=1}^n \ta_{n,i}^4\rightarrow 0$, 
and thus $m\left(m-1+3(\zeta_m-1)^2\right)/\left((m-1)\zeta_m^2\right)\cdot\sum_{i<k} w_{ik}^4\rightarrow 0$.

Similarly, since $\zeta_m^2/m\to 0$, $\sum_{k=1}^n (a_{n,i}^4+\ta_{n,i}^4)=o(1)$. 
Thus,
\begin{align*}
&\sum_{k=2}^n\sum_{i=1}^{k-1}\sum_{j=1}^{k-1} w_{ik}^2 w_{jk}^2  = \sum_{k=2}^n \left(\sum_{i=1}^{k-1} w_{ik}^2\right)^2
\le  \sum_{k=1}^n \left(\sum_{i=1}^{k-1} (a_{n,i}^2+\ta_{n,i}^2)(a_{n,k}^2+\ta_{n,k}^2)\right)^2\\
&\le  \sum_{k=1}^n \left(\sum_{i=1}^n (a_{n,i}^2+\ta_{n,i}^2)(a_{n,k}^2+\ta_{n,k}^2)\right)^2
= 4\sum_{k=1}^n (a_{n,k}^2+\ta_{n,k}^2)^2
\le  8\sum_{k=1}^n (a_{n,i}^4+\ta_{n,i}^4)\rightarrow 0.
\end{align*}
In the last equation, we have used that for all $i$,
$\sum_{i=1}^n a_{n,i}^2=\sum_{i=1}^n \ta_{n,i}^2=1$.

Additionally, $2\sum_{i=1}^n a_{n,i}^2\ta_{n,i}^2\le  \sum_{i=1}^n (a_{n,i}^4+\ta_{n,i}^4)\rightarrow 0$.
In summary, we have shown that $\sum_{k=2}^n \E \Upsilon_k^4/s_n^4\rightarrow 0$, which verifies the Lyapunov condition.

\subsubsection*{Proof of ratio-consistency of variance}

Next, 
 recalling $q_n^2
=\sum_{k=2}^n \E[\Upsilon_k^{2}|T_{1},\ldots,T_{k-1}]$,
and $s_n^2$ from \eqref{sn2},
we aim to prove that $
s_{n}^{-2}q_n^{2}\rightarrow_P 1.
$
For this, we calculate, for any $k\in \{2,\ldots, n\}$ 
\begin{align*}
\E[\Upsilon_k^{2}|T_{1},\ldots,T_{k-1}]
&=\sum_{j=1}^{k-1}\sum_{i=1}^{k-1} w_{ik}w_{jk}\E[h_n(T_j,T_k)h_n(T_i,T_k)|T_{1},\ldots,T_{k-1}]\\
&=\sum_{j=1}^{k-1}\sum_{i=1}^{k-1} w_{ik}w_{jk}R_{m,j}^\top \E[R_{m,k} R_{m,k}^\top] R_{m,i}=\sum_{j=1}^{k-1}\sum_{i=1}^{k-1} w_{ik}w_{jk}R_{m,j}^\top M_{\zeta_m} R_{m,i}/m.
\end{align*}
Therefore, we have
\begin{align*}
q_n^2
&=\frac{1}{m}\sum_{k=1}^n \sum_{j=1}^{k-1}\sum_{i=1}^{k-1} w_{ik}w_{jk}R_{m,j}^\top M_{\zeta_m} R_{m,i}
=\frac{1}{m}\sum_{k=1}^n \left(\sum_{i=1}^{k-1} w_{ik}^2+\sum_{i=1}^{k-1}\sum_{j=1}^{k-1}(1-\delta_{ij}) w_{ik}w_{jk}R_{m,j}^\top M_{\zeta_m} R_{m,i}\right).
\end{align*}
and thus
$$
q_n^2-s_n^2=\frac{1}{m}\sum_{k=1}^n \sum_{i=1}^{k-1}\sum_{j=1}^{k-1}(1-\delta_{ij}) w_{ik}w_{jk}R_{m,j}^\top M_{\zeta_m} R_{m,i}.
$$
As a result,
\begin{align}\label{qs}
\E [(q_n^2-s_n^2)^2]
&= \frac{1}{m^2}\E\left[\left(\sum_{k=1}^n \sum_{i=1}^{k-1}\sum_{j=1}^{k-1}(1-\delta_{ij}) w_{ik}w_{jk} R_{m,j}^\top M_{\zeta_m} R_{m,i}\right)^2\right]\\
&= \frac{1}{m^2}\sum_{i_1\neq j_1<k_1}\sum_{i_2\neq j_2<k_2} w_{i_1 k_1}w_{j_1 k_1}w_{i_2 k_2}w_{j_2 k_2}
\E\left[R_{m,i_1}^\top M_{\zeta_m} R_{m,j_1} R_{m,i_2}^\top M_{\zeta_m} R_{m,j_2}\right],\nonumber
\end{align}
where $\sum_{i_1\neq j_1<k_1} \star \;$ is an abbreviation for $\sum_{k_1=1}^n \sum_{i_1=1}^{k_1-1}\sum_{j_1=1}^{k_1-1}(1-\delta_{i_1j_1}) \; \star \;$. 
From the computation of the fourth moment of $\Upsilon_k$, we know that the terms above
are nonzero only when either
$i_2=i_1, j_2=j_1$ or $i_2=j_1, j_2=i_1$. 
Due to the symmetry of the inner product, 
$R_{m,i_2}^\top M_{\zeta_m} R_{m,j_2}=R_{m,i_1}^\top M_{\zeta_m} R_{m,j_1}$ in both cases.
Moreover, 
based on our previous results, for $i_1\neq j_1$,
\begin{align*}
    &\E\left[\left(R_{m,i_1}^\top M_{\zeta_m} R_{m,j_1}\right)^2\right]=\E\left[R_{m,i_1}^\top M_{\zeta_m} \E[R_{m,j_1}R_{m,j_1}^\top] M_{\zeta_m} R_{m,i_1}\right]=1/m\cdot\E\left[\tr(M_{\zeta_m}^2 R_{m,i_1}R_{m,i_1}^\top)\right]\\
    &=1/m^2\cdot\tr(M_{\zeta_m}^2)=1/m\cdot\left(1+(\zeta_m-1)^2/(m-1)\right).
\end{align*}
This yields 
that \eqref{qs} equals
\begin{align}\label{intq}
\frac{1+(\zeta_m-1)^2/(m-1)}{m^3}\sum_{k_1=2}^n\sum_{k_2=2}^n\sum_{i_1=1}^{k_1\wedge k_2-1}\sum_{j_1=1}^{k_1\wedge k_2-1}(1-\delta_{i_1j_1}) w_{i_1 k_1}w_{j_1 k_1}w_{i_1 k_2}w_{j_1 k_2}.
\end{align}
Now,
 recalling that  $w_{ij}=a_{n,i}\ta_{n,j}+a_{n,j}\ta_{n,i}$ for all $i,j$,
and applying the Cauchy-Schwarz inequality, we have
$$\left|w_{i_1 k_1}w_{j_1 k_1}w_{i_1 k_2}w_{j_1 k_2}\right|\le  \left(a_{n,i_1}^2+\ta_{n,i_1}^2\right)\left(a_{n,j_1}^2+\ta_{n,j_1}^2\right)\left(a_{n,k_1}^2+\ta_{n,k_1}^2\right)\left(a_{n,k_2}^2+\ta_{n,k_2}^2\right).$$
Thus, 
 using that for all $i$,
$\sum_{i=1}^n a_{n,i}^2=\sum_{i=1}^n \ta_{n,i}^2=1$,
\eqref{intq} can be bounded by
\begin{align*}
&\frac{1+\frac{(\zeta_m-1)^2}{m-1}}{m^3}\sum_{k_1=2}^n\sum_{k_2=2}^n\sum_{i_1=1}^{k_1\wedge k_2-1}\sum_{j_1=1}^{k_1\wedge k_2-1}\left(a_{n,i_1}^2+\ta_{n,i_1}^2\right)\left(a_{n,j_1}^2+\ta_{n,j_1}^2\right)\left(a_{n,k_1}^2+\ta_{n,k_1}^2\right)\left(a_{n,k_2}^2+\ta_{n,k_2}^2\right)\\
&\le \frac{1+\frac{(\zeta_m-1)^2}{m-1}}{m^3}\sum_{i_1, j_1, k_1, k_2}\left(a_{n,i_1}^2+\ta_{n,i_1}^2\right)\left(a_{n,j_1}^2+\ta_{n,j_1}^2\right)\left(a_{n,k_1}^2+\ta_{n,k_1}^2\right)\left(a_{n,k_2}^2+\ta_{n,k_2}^2\right) 
=\frac{16\left(1+\frac{(\zeta_m-1)^2}{m-1}\right)}{m}.
\end{align*}
Therefore, since $\zeta_m^2/m\to 0$, we have established that
$$
s_n^{-4}\E [(q_n^2-s_n^2)^2]\le 
\frac{16/m^3\cdot\left(1+(\zeta_m-1)^2/(m-1)\right)}{\left(1+(a_n^\top \ta_n)^2-2\sum_{i=1}^n a_{n,i}^2\ta_{n,i}^2\right)^2/m^2}\rightarrow 0.
$$
This shows that 
$s_n^{-2} q_n^2\stackrel{\ell_2}{\rightarrow} 1$, and thus $s_n^{-2} q_n^2\rightarrow_P 1$.
Since
$s_n^{-1}S_n\Rightarrow \N(0,1)$ and
$
m s_n^2/\!\left(1+(a_n^\top \ta_n)^2\right)$ $ \to_P 1
$, \eqref{csclt} follows.

\subsection{Proof of \texorpdfstring{Lemma \ref{lemsymqfcs}}{Lemma \ref{lemsymqfcs}}}
We will use a martingale central limit theorem argument, as in the proof of Lemma \ref{lemgeneralqfcs}.
By  the Cramer-Wold device as shown in \eqref{cw_vec_srht_PCA}, we can focus on
$$
\tr\left(\Phi(U_n^\top S_{m,n}^\top S_{m,n}U_n-I_p)\right)
=\sum_{i=1}^p \mu_i (w_i^\top U_n^\top S_{m,n}^\top S_{m,n}U_n w_i-1).
$$
We denote $t_{i}^k:=(U_nw_i)_k$ for $i\in [p], i\in [n]$. Then, we can rewrite the above as
\begin{align*}
\sum_{i=1}^p \mu_i \left(\sum_{k\neq j} t_{k}^it_{j}^i  R_{m,k}^\top R_{m,j}\right)
=\sum_{k<j} 2\left(\sum_{i=1}^p \mu_i t_{k}^it_{j}^i\right)  R_{m,k}^\top R_{m,j}.
\end{align*}
As in \eqref{weightedUstat}, 
this is a weighted degenerate U-statistic,
where $w_{kj}=2\left(\sum_{i=1}^p \mu_i t_{k}^it_{j}^i\right)$ for all $k,j$. 
Using notations from Lemma \ref{lemgeneralqfcs},
as well as $t^k = (t_{i}^k)_{i\in [n]}$ for $k\in [p]$, 
we find
\begin{align*}
m s_n^2
&=\sum_{i<j} w_{ij}^2=4\sum_{i<j} \left(\sum_{k=1}^p 
\mu_k t_{i}^k t_{j}^k \right)^2=4\sum_{k_1,k_2} \mu_{k_1}\mu_{k_2} \sum_{i<j} 
t_{i}^{k_1} t_{j}^{k_1} t_{i}^{k_2} t_{j}^{k_2}\\
&=2\sum_{k_1,k_2} \mu_{k_1}\mu_{k_2} \left(\left(t^{k_1,\top} t^{k_2}\right)^2-\sum_{i=1}^n 
\left(t_{i}^{k_1}\right)^2 \left(t_{i}^{k_2}\right)^2\right).
\end{align*}
By the orthogonality of $U_n$, 
$t^{k_1,\top} t^{k_2} = w_{k_1}^\top w_{k_2} = \delta_{k_1 k_2}$, and so 
this equals
\begin{align*}
&2\sum_{k_1,k_2} \mu_{k_1}\mu_{k_2} \left(\delta_{k_1 k_2}-\sum_{i=1}^n 
\left(t_{i}^{k_1}\right)^2 \left(t_{i}^{k_2}\right)^2\right)
=2\sum_{k=1}^p \mu_{k}^2 
- 2\sum_{k_1,k_2} \mu_{k_1}\mu_{k_2} \sum_{i=1}^n 
\left(t_{i}^{k_1}\right)^2 \left(t_{i}^{k_2}\right)^2.
\end{align*}
Since $(U_n)_{n\ge 1}$ satisfies Condition $\mathfrak{P}'$ 
from \eqref{csstate2},
we have $\sum_{i=1}^n (U_n)_{ki}^4=o\left(\zeta_m/\sqrt{m}\right)$ for every $k\in [p]$,
and 
so recalling that  $t_{i}^k=(U_nw_k)_i$ for $k\in [p], i\in [n]$,
by the Cauchy-Schwarz inequality, 
we can derive that $\sum_{i=1}^n\left(t_{i}^{k_1}\right)^2\left(t_{i}^{k_2}\right)^2=o\left(\zeta_m/\sqrt{m}\right)$ for every $k_1,k_2\in [p]$. 
Thus, 
$ms_n^2-2\sum_{k=1}^p \mu_k^2\rightarrow 0$.
By applying the Cauchy-Schwarz inequality once more, we find that
$$
w_{ij}^2=4\left(\sum_{k=1}^p \mu_k t_{i}^k t_{j}^k\right)^2\le  4\left(\sum_{k=1}^p\mu_k^2 \left(t_{i}^k\right)^2 \right)\left(\sum_{k=1}^p \left(t_{j}^k\right)^2\right)\le  4\left(\sum_{k=1}^p\mu_k^2 \right)\left(\sum_{k=1}^p\left(t_{i}^k\right)^2 \right)\left(\sum_{k=1}^p \left(t_{j}^k\right)^2\right).
$$
Denoting $C:=4\left(\sum_{k=1}^p\mu_k^2 \right)$, we arrive at
\begin{align*}
\sum_{i<j} w_{ij}^4 
&\le  C^2 \sum_{i<j} \left(\sum_{k=1}^p\left(t_{i}^k\right)^2 \right)^2\left(\sum_{k=1}^p \left(t_{j}^k\right)^2\right)^2\le  C^2 \sum_{i,j}\left(\sum_{k=1}^p\left(t_{i}^k\right)^2 \right)^2\left(\sum_{k=1}^p \left(t_{j}^k\right)^2\right)^2\\
&=\left(C\sum_{i=1}^n \left(\sum_{k=1}^p\left(t_{i}^k\right)^2\right)^2 \right)^2=\left(C\sum_{k_1,k_2}\sum_{i=1}^n\left(t_{i}^{k_1}\right)^2\left(t_{i}^{k_2}\right)^2 \right)^2.
\end{align*}
Since $\sum_{i=1}^n\left(t_{i}^{k_1}\right)^2\left(t_{i}^{k_2}\right)^2=o(\zeta_m/\sqrt{m})$, we conclude that
$\frac{m\left(m-1+3(\zeta_m-1)^2\right)}{\left((m-1)\zeta_m^2\right)}\cdot\sum_{i<j} w_{ij}^4 \rightarrow 0$, which verifies 
the first part of 
\eqref{martprof4momentto01}.
Additionally,
since  for all $k\in[p]$,
$\sum_{i=1}^n (t_{i}^k)^2
=\sum_{i=1}^n (U_nw_k)_i^2=\|w_k\|^2=1$,
\begin{align*}
\sum_{j=1}^n \left(\sum_{i=1}^{j-1} w_{ij}^2\right)^2
&\le  C^2\sum_{j=1}^n 
\left[\sum_{i=1}^{j-1}\left(\sum_{k=1}^p \left(t_{i}^k\right)^2\right)
\left(\sum_{k=1}^p \left(t_{j}^k\right)^2\right)\right]^2
\le  C^2\sum_{j=1}^n \left[\sum_{i=1}^n\left(\sum_{k=1}^p \left(t_{i}^k\right)^2\right)\left(\sum_{k=1}^p \left(t_{j}^k\right)^2\right)\right]^2\\
&= C^2\sum_{j=1}^n \left(p\sum_{k=1}^p \left(t_{j}^k\right)^2\right)^2
= p^2C^2\sum_{k_1,k_2}\sum_{i=1}^n\left(t_{i}^{k_1}\right)^2\left(t_{i}^{k_2}\right)^2\rightarrow 0,
\end{align*}
which verifies the second part of \eqref{martprof4momentto01}. Finally, we have
\begin{align*}
&\frac{1+(\zeta_m-1)^2/(m-1)}{m} \sum_{k_1=2}^n\sum_{k_2=2}^n\sum_{i_1=1}^{k_1\wedge k_2-1}\sum_{j_1=1}^{k_1\wedge k_2-1} w_{i_1 k_1}w_{j_1 k_1}w_{i_1 k_2}w_{j_1 k_2}\\
&\le  \frac{C^2\left(1+(\zeta_m-1)^2/(m-1)\right)}{m}\sum_{i_1,j_1,k_1,k_2} \left(\sum_{l=1}^p \left(r_{i_1}^l\right)^2\right)
\left(\sum_{l=1}^p \left(r_{j_1}^l\right)^2\right)
\left(\sum_{l=1}^p \left(r_{k_1}^l\right)^2\right)
\left(\sum_{l=1}^p \left(r_{k_2}^l\right)^2\right)\\
&= \frac{C^2\left(1+(\zeta_m-1)^2/(m-1)\right)}{m}\left(\sum_{i=1}^n\sum_{l=1}^p \left(r_{i}^l\right)^2\right)^4= \frac{C^2p^4\left(1+(\zeta_m-1)^2/(m-1)\right)}{m}\rightarrow 0.
\end{align*}
Thus, 
the conditions needed for the martingale central limit theorem, as used previously in the proof of Lemma \ref{lemgeneralqfcs},
are satisfied. 
Since 
$2\sum_{k=1}^p \mu_k^2=\vec(\Phi)^\top (I_{p^2}+P_p)\vec(\Phi)$, 
we conclude that
$$
\sqrt{m}\tr\left(\Phi(U_n^\top S_{m,n}^\top S_{m,n}U_n-I_p)\right)\Rightarrow \N\left(0,\vec(\Phi)^\top (I_{p^2}+P_p)\vec(\Phi)\right),
$$
and \eqref{cssqf} follows.

\subsection{Proof of \texorpdfstring{Theorem \ref{theigss}}{Theorem \ref{theigss}}}

The proofs of all claims other than \eqref{sse2} can be directly obtained from \Cref{infPCAgeneral}
by letting $G_n:= n\sum_{i=1}^n\left((u_{n,i}u_{n,i}^\top)\otimes(u_{n,i}u_{n,i}^\top)\right)$,
as $(G_n)_{(ii),(ii)}=n\sum_{k=1}^n(U_n)_{ki}^4$.
Now, to conclude \eqref{sse2} from \eqref{sse},
let $u_{n,i}$ be the $i$-th column of $U_n$.
By calculating variances, we argue that
\begin{equation}\label{ra}
    \frac{m\sum_{j=1}^m (s_j^\top u_{n,i})^4} {n\sum_{i=1}^n (U_n)^4_{ki}}\to_P 1.
\end{equation}
Since
$
\sum_{j=1}^m (s_j^\top u_{n,i})^4=n^2/m^2\cdot\sum_{l=1}^n (b_l(U_{n})_{li})^4=n^2/m^2\cdot\sum_{l=1}^n (U_{n})_{li}^4 b_l,
$
by Condition $\mathfrak{P}'$, the variance of the left hand side of \eqref{ra} is $$\frac{1-\gamma_n}{\gamma_n}\frac{\sum_{l=1}^n (U_{n})_{li}^8}{(\sum_{l=1}^n (U_{n})_{li}^4)^2}\le \frac{1-\gamma_n}{\gamma_n}\frac{\max_{l\in[n]} (U_{n})_{li}^4}{\sum_{l=1}^n (U_{n})_{li}^4}\to_P 0.$$
Now, recall that $\tX_{m,n} =  \tU_{m,n}\hat L_{m,n}\hat V_{m,n}^\top $ is the SVD of $\tX_{m,n}$. 
By Condition $\mathfrak{P}'$ defined in \eqref{assdelocalsubsamp3},
$n\sum_{k=1}^n(U_n)_{ki}^4/m=(G_n)_{(ii),(ii)}/m\to 0$; combining this with \Cref{theigss}, we have
$L_n^{-1} \hat L_{m,n} \to_P I_p$ and $V_n^{-1}\hat V_{m,n} \to_P I_p$.
Noting that $\tU_{m,n} = \tX_{m,n}\hat V_{m,n}  \hat L_{m,n}^{-1}
=  S_{m,n}U_n L_n V_n^\top \hat V_{m,n} \hat  L_{m,n}^{-1}$,
and $ L_n V_n^\top \hat V_{m,n} \hat  L_{m,n}^{-1} \to_P I_p$. Expanding the summation in the numerator according to the rules of matrix element multiplication, we conclude that
\begin{equation*}
    n^{-1}m\sum_{j=1}^m \tilde{U}_{ji}^4 / \left[n^{-1}m\sum_{j=1}^m (s_j^\top u_{n,i})^4\right] \to_P 1.
\end{equation*}
Hence, \eqref{sse2} follows from \eqref{sse}.

\subsection{Proof of \texorpdfstring{Theorem \ref{thsubsamp}}{Theorem \ref{thsubsamp}}}
The proof is based on the following asymptotic distribution and consistency properties.
\begin{lemma}[Asymptotic Normality in Least Squares with Uniform Random Sampling]\label{lemsubsamp}
Under the conditions of \Cref{thsubsamp},
with
\begin{equation}\label{asssubsamp}    \Sigma_{n}:=n(X_n^\top X_n)^{-1} \left( \sum_{i=1}^n \ep_{n,i}^2x_{n,i} x_{n,i}^\top \right) (X_n^\top X_n)^{-1},
\end{equation}
we have
$m^{1/2}(1-\gamma_n)^{-1/2}\Sigma_n^{-1/2}(\hbs-\beta_n) \Rightarrow \N(0,I_p).$
Similarly, defining 
\begin{equation}\label{asssubsamp2}     \Sigp:=n(X_n^\top X_n)^{-1} \left(\sum_{i=1}^n (y_{n,i}- \ep_{n,i})^2 x_{n,i} x_{n,i}^\top  \right)(X_n^\top X_n)^{-1},\end{equation}
we have
$m^{1/2}(1-\gamma_n)^{-1/2}(\Sigp)^{-1/2}(\hbp-\beta_n) \Rightarrow \N(0,I_p)$.
\end{lemma}
The asymptotic result \eqref{asssubsamp} coincides with Theorem 2 of \cite{yu2022optimal} by taking $\Psi(x) = x$ and $p_i = r/n$ therein. Their result is assumes that $X_n$ has i.i.d.~rows and $y_n$ has i.i.d.~entries,
and then conditions on $X_n$ and $y_n$.
In contrast, we do not make this assumption  on $X_n$ and $y_n$.

\begin{proof}
For any $b_n\in\psp$, $n\ge 1$, since $\|\vec(U_n)\|_\infty=O(1/\sqrt{n})$, we have $\|U_nb_n\|_\infty=O(1/\sqrt{n})$. 
Then $({U_nb_n, U_nb_n})_{n\ge 1}$ satisfies Condition $\mathfrak{p}$, and due to Lemma \ref{lemboundqfsubsamp}, 
Condition \ref{generalquadratic forms} \textup{\texttt{Bounded}} holds for the sequences $({U_nb_n, U_nb_n})_{n\ge 1}$.

From Lemma \ref{lemboundqfsubsamp}, $\sigma_n^2(a_n, \ta_n)=n\sum_{i=1}^n (a_{n,i}\ta_{n,i})^2$. 
By \eqref{strdelocsubsamp1}, $({U_nb_n, \bar{\ep}_n})_{n\ge 1}$ satisfies Condition $\mathfrak{P}$ and the sequence of variances
$\sigma_n^2(U_nb_n, \bar{\ep}_n)=n\sum_{i=1}^n (U_nb_n)_{i}^2(\bar{\ep}_n)_{i}^2$ is uniformly bounded away from zero and infinity. Therefore, 
Condition \ref{generalquadratic forms} \textup{\texttt{1-dim}} holds for the sequences $({U_nb_n, \bar{\ep}_n})_{n\ge 1}$, and $M_n=n U_n^\top\diag\big(\bar{\ep}_n\odot \bar{\ep}_n\big) U_n$.
Similarly to the proof in \Cref{lemiid}, combining this with Lemma \ref{lemboundqfsubsamp}, we conclude 
the desired claim for $\hbs$. 

Similarly, Condition \ref{generalquadratic forms} \textup{\texttt{1-dim}} holds for $({U_nb_n, \widebar{X_n\beta_n}})_{n\ge 1}$, and using Lemma \ref{lemboundqfsubsamp}, we have
$\Mp = n\|X_n \beta_n\|^{-2}U_n^\top\diag((X_n\beta_n) \odot (X_n \beta_n))U_n \in \Sp$. 
Therefore, if we take $b_n=L_n^{-1} V_n^\top c/\|X_n^{\dag \top}c\|$,
the claim for $\hbp$ follows as in \Cref{lemiid}.
\end{proof}
\begin{proposition}[Consistency of the Uniform Random Sampling Estimator of the Covariance Matrices]\label{consestsubsamp}
Under the conditions of \Cref{thsubsamp}, we have 
\begin{equation*}
\|  \Sigma_n^{-1}\hat{\Sigma}_{m,n}-I_p \|_{\Fr} = o_{P}(1)
\,\textnormal{  and  }\, \|(\Sigp)^{-1}\hSigp-I_p \|_{\Fr} = o_{P}(1).\end{equation*}
\end{proposition}
\begin{proof}
We consider $y_n\in \nsp$ and $X_n$ of a unit spectral radius first.
From \eqref{nc}, $\|\tX_{m,n}^\top \tX_{m,n} - X_n^\top X_n\| = o_P(1)$.
Thus, it suffices to show that
\begin{equation*}
\left\| \frac{m}{n}\tX_{m,n}^\top \diag(\tep_{m,n} \odot \tep_{m,n}) \tX_{m,n} 
- \sum_{i=1}^n \ep_{n,i}^2x_{n,i} x_{n,i}^\top\right\| = o_P(1).
\end{equation*}
Since $S_{m,n}=\diag(\sqrt{n/m}B_1,\ldots,\sqrt{n/m}B_n)$, we have
\begin{align*}
\tX_{m,n}^\top \diag(\tep_{m,n} \odot \tep_{m,n}) \tX_{m,n}
&=\sqrt{n/m}X_n^\top S_{m,n}\diag(B_1^2(y_{n}-X_n \hbs)_{1}^2, \ldots, B_n^2(y_{n}-X_n \hbs)_{n}^2) S_{m,n} X_{n}\\
&=(n/m)\cdot\sum_{i=1}^n (n/m) B_i(\ep_n+X_n(\beta_n- \hbs))_i^2x_{n,i}x_{n,i}^\top.
\end{align*}
Here, we can check
that 
$\sum_{i=1}^n [(n/m)B_i-1]\ep_{n,i}^2(x_{n,i}^\top c)^2=o_P(1)$, while $\|\beta_n- \hbs\|=O_P(\sqrt{1-\gamma_n)/m})$. 
Therefore, 
$$\sum_{i=1}^n (n/m)B_i\cdot 2\ep_{n,i}x_{n,i}^\top(\beta_n- \hbs) (x_{n,i}^\top c)^2=O_P(\sqrt{(1-\gamma_n)/m}\cdot n/m)$$
and $\sum_{i=1}^n (n/m)B_i\cdot (x_{n,i}^\top(\beta_n- \hbs))^2 (x_{n,i}^\top c)^2=O_P((1-\gamma_n)/m\cdot n/m)$.

Thus, the first claim of Proposition \ref{consestsubsamp} holds when $\|y_n\|=1$ and $\|X_n\|=1$. 
For general $y_n$ and $X_n$, the proof is similar to that of Proposition \ref{consest}. 
For the second conclusion, the proof is almost identical to the above argument.
\end{proof}

\subsection{Proof of \texorpdfstring{Lemma \ref{lemboundqfsubsamp}}{Lemma \ref{lemboundqfsubsamp}}}
To show \eqref{ssbo}, notice that
$\Var{a_n^{\top} S_{m, n}^{\top} S_{m, n} \tilde{a}_n}=n\sum_{i=1}^n a_{n,i}^2\ta_{n,i}^2\cdot(1-\gamma_n)/m$.
Then 
\eqref{ssbo} follows by Condition $\mathfrak{p}$ from \eqref{assdelocalsubsamp2}.

To see \eqref{ssn},
we know that
$a_n^\top S_{m,n}^\top S_{m,n}\ta_n=\frac{n}{m}\sum_{i=1}^n a_{n,i}\ta_{n,i} B_i$ and $\E B_i=m/n$. 
Denoting $J_i=(1/\gamma_n-1)^{-1/2}(\frac{n}{m} B_i-1)$,
we have that $J_i$ are i.i.d.~with $\E J_i=0$ and $\E J_i^2=1$.
As a consequence, $a_n^\top S_{m,n}^\top S_{m,n}\ta_n-a_n^\top \ta_n=(1/\gamma_n-1)^{1/2}\sum_{i=1}^n a_{n,i}\ta_{n,i} J_i$.
We check the Lindeberg condition
\begin{equation}\label{lcss}
    \sum_{i=1}^n \E[(a_{n,i} \ta_{n,i})^2 J_i^2I\{(a_{n,i} \ta_{n,i})^2J_i^2>\epsilon s_n^2\}]/s_n^2\rightarrow 0
\end{equation}
for all $\epsilon>0$, where $s_n^2=\Var{\sum_{i=1}^n a_{n,i}\ta_{n,i} J_i}=\sum_{i=1}^n a_{n,i}^2\ta_{n,i}^2\Var{J_i}=\sum_{i=1}^n a_{n,i}^2\ta_{n,i}^2$.
We have
$$\E [J_i^2I\{(a_{n,i} \ta_{n,i})^2J_i^2>\epsilon s_n^2\}]\le  \E\left [J_i^2I\left\{\max_{i\in[n]} \frac{(a_{n,i} \ta_{n,i})^2}{s_n^2}J_i^2>\epsilon \right\}\right],$$
so the left hand side of \eqref{lcss} is upper bounded by
$$
\sum_{i=1}^n \frac{(a_{n,i} \ta_{n,i})^2}{s_n^2} \E\left [J_i^2I\left\{\max_{i\in[n]} \frac{(a_{n,i} \ta_{n,i})^2}{s_n^2}J_i^2>\epsilon \right\}\right]
=
\E\left [J_i^2I\left\{\max_{i\in[n]} \frac{(a_{n,i} \ta_{n,i})^2}{s_n^2}J_i^2>\epsilon \right\}\right]\to 0,
$$
where the convergence to zero holds due to Condition $\mathfrak{P}$.
Then \eqref{ssn} follows.

A counterexample 
that does not satisfy 
Condition $\mathfrak{P}$, 
such that 
$\sqrt{m} a_n^\top S_{m,n}^\top S_{m,n} \ta_n$ is not asymptotically normal
is $a_n=\ta_n=(1,0,\ldots,0)$. Then,
$$
\sqrt{\frac{m}{1-\gamma_n}}\sqrt{\frac{1}{n\sum_{i=1}^n a_{n,i}^2\ta_{n,i}^2}}\sum_{i=1}^n a_{n,i}\ta_{n,i} J_i=\sqrt{\frac{m}{n-m}}J_1.
$$
This random variable takes two values, 
so it is not asymptotically normal.

\subsection{Proof of \texorpdfstring{Lemma \ref{lemsymqfunifrandsamp}}{Lemma \ref{lemsymqfunifrandsamp}}}
By the subsequence argument and Cramer-Wold device as in the beginning of the proof of Lemma \ref{symiid}, it suffices to show that for any upper triangular $\Phi\in\R^{p\times p}$,
\begin{equation*}
\sqrt{\frac{m}{1-\gamma_n}} \left(\vec(\Phi)^\top G_n\vec(\Phi)\right)^{-1/2} \vec(\Phi)^\top \vec(U_n^\top S_{m,n}^\top S_{m,n}U_n-I_p)\Rightarrow \N\left(0,1\right).
\end{equation*}
Denote $U_n=(u_{n,1},\ldots,u_{n,n})^\top$, where $u_{n,i}^\top$ is the $i$-th row of $U_n$. Then
$$
\tr\left(\Phi(U_n^\top S_{m,n}^\top S_{m,n}U_n-I_p)\right)
=\tr\left(\Phi\sum_{i=1}^n u_{n,i} u_{n,i}^\top\left(\frac{n}{m} B_i-1\right)\right)
=\sum_{i=1}^n u_{n,i}^\top\Phi u_{n,i}  \left(\frac{n}{m} B_i-1\right).
$$
Denote $J_i=(1/\gamma_n-1)^{-1/2}(\frac{n}{m} B_i-1)$ for all $i\in [n]$, 
so that $\E J_i=0, \E J_i^2=1$. 
Then from Lemma \ref{lemboundqfsubsamp},
if 
$U_n$ 
satisfies 
$\frac{\max_{i\in[n]}(u_{n,i}^\top \Phi u_{n,i})^2}{\sum_{i=1}^n (u_{n,i}^\top \Phi u_{n,i})^2}$ $\to 0$
in Condition $\mathfrak{P}'$ from \eqref{assdelocalsubsamp3}, 
the above term is asymptotically normal.

\subsection{Proof of \texorpdfstring{Theorem \ref{iid1}}{Theorem \ref{iid1}}}
By Lemma \ref{symiid}, since $(U_n)_{n\ge 1}$ satisfies Condition $\mathfrak{P}'$,
 Condition \ref{generalquadratic forms} \textup{\texttt{Sym}}
 holds 
with
$\tau_{m,n}=1$ and
$G_n = I_{p^2}+P_p+\Gamma_n$. 
Moreover, for all $i\in [p]$,
$(G_n)_{(ii),(ii)}=2+(\Gamma_n)_{(ii),(ii)},$ and thus from \eqref{defHi},
$$\Delta_i=\Delta_i(\Xi, G_n)=
\sum_{k\neq i, l\neq i} \frac{\lambda_i(\Xi)\sqrt{\lambda_k(\Xi)\lambda_l(\Xi)}}{\left(\lambda_i(\Xi)-\lambda_k(\Xi)\right)\left(\lambda_i(\Xi)-\lambda_l(\Xi)\right)}\left(\delta_{kl}+(\Gamma_n)_{(ij),(kl)}\right)v_k(\Xi)v_l(\Xi)^\top.
$$
Since $(X_n)_{n\ge 1}$ satisfies Condition \ref{spectralcondition} \textup{\texttt{B}}, 
the conclusion holds due to \Cref{infPCAgeneral}.

\subsection{Proof of \texorpdfstring{Theorem \ref{thiid}}{Theorem \ref{thiid}}}
We first show the following result, which establishes the asymptotic normality of sketched least squares estimators. 
The result can be regarded as a special case of Theorem 3.2 in \cite{zhang2023framework} for a fixed dimension $p$,
but 
with the weaker condition that the moments 
of the sketching matrices have a $(4+\delta)$-th moment, as opposed to having all moments finite.
Also, it turns out 
that for i.i.d.~sketching,
using our general framework would require
assuming that $(U_n)_{n\ge 1}$ satisfies the conditions of Lemma \ref{symiid}.
However, by a more direct analysis, we can avoid this condition; and 
so the proof we present deviates from our general framework.

In the remainder, we denote $X_n=(x_{n,1}^\top,\ldots,x_{n,p}^\top)^\top\in\R^{n\times p}$, $y_n=(y_{n,1},\ldots,y_{n,p})^\top\in\R^p$.

\begin{lemma}[Asymptotic Normality in Least Squares with i.i.d.~Sketching]\label{lemiid}
Under the conditions of \Cref{thiid},
with
\begin{equation}\label{assiidfix}    \Sigma_{n}:=(X_n^\top X_n)^{-1} \left\{ \sum_{i=1}^n \ep_{n,i}^2\left[X_n^\top X_n+(\kappa_{n,4}-3)  x_{n,i} x_{n,i}^\top\right] \right\} (X_n^\top X_n)^{-1},
\end{equation}
we have
$m^{1/2}\Sigma_n^{-1/2}(\hbs-\beta_n) \Rightarrow \N(0,I_p).$
Similarly, defining 
\begin{equation}\label{assiidfix2}     \Sigp:=(X_n^\top X_n)^{-1} \|y_{n}-\ep_{n}\|^2 + (X_n^\top X_n)^{-1} \left[(\kappa_{n,4}-3) \sum_{i=1}^n (y_{n,i}- \ep_{n,i})^2 x_{n,i} x_{n,i}^\top  \right](X_n^\top X_n)^{-1}+\beta_n\beta_n^\top,\end{equation}
we have
$m^{1/2}(\Sigp)^{-1/2}(\hbp-\beta_n) \Rightarrow \N(0,I_p)$.
\end{lemma}

\begin{proof}
Consider a sequence of vectors $({a_n, \ta_n})_{n\ge 1}$ such that $a_n, \ta_n\in\nsp$ for all $n$. 
From Lemma \ref{lemgeneralqfiid}, 
we have Condition \ref{generalquadratic forms} \textup{\texttt{Bounded}}. 

For the first claim, 
\eqref{vlb}
in the proof of Lemma \ref{lemgeneralqfiid} in \Cref{pflemgeneralqfiid}
shows that the sequence of variances
$\sigma_n^2(a_n, \ta_n)=1+(a_n^\top \ta_n)^2+(\kappa_{n,4}-3)\sum_{i=1}^n (a_{n,i}\ta_{n,i})^2$ is uniformly bounded away from zero and infinity. Therefore,
Condition \ref{generalquadratic forms} \textup{\texttt{1-dim}} holds for the sequence $({U_nb_n, \bar{\ep}_n})_{n\ge 1}$, and we have $M_n=I_p+(\kappa_{n,4}-3)U_n^\top\diag\big(\bar{\ep}_n\odot \bar{\ep}_n\big) U_n$.
Thus,
\begin{align*}
V_nL_n^{-1}M_nL_n^{-1}V_n^\top
&=V_nL_n^{-2}V_n^\top+(\kappa_{n,4}-3)V_nL_n^{-1}U_n^\top\diag\big(\bar{\ep}_n\odot \bar{\ep}_n\big) U_nL_n^{-1}V_n^\top\\
&=(X_n^\top X_n)^{-1}+(\kappa_{n,4}-3)(X_n^\top X_n)^{-1}X_n^{\top}\diag\big(\ep_n\odot \ep_n\big)X_n(X_n^\top X_n)^{-1}/\|\ep_n\|^2\\
&=(X_n^\top X_n)^{-1}+(\kappa_{n,4}-3)(X_n^\top X_n)^{-1}\biggl(\sum_{i=1}^n \ep_{n,i}^2  x_{n,i} x_{n,i}^\top\biggr)(X_n^\top X_n)^{-1}/\|\ep_n\|^2.
\end{align*}
Therefore, 
the claim for $\hbs$ follows 
from \eqref{ssa} in \Cref{lemgeneols}; where we remark that all necessary conditions hold due to the stated assumptions. 

For the second claim, we define $\tau_{m,n}=\kappa_{n,4}$, 
so that by the H{\"o}lder inequality, $\tau_{m,n}=o(m)$. Also, $\sigma_n^2(a_n, \ta_n)=(1+(a_n^\top \ta_n)^2+(\kappa_{n,4}-3)\sum_{i=1}^n (a_{n,i}\ta_{n,i})^2)/\kappa_{n,4}=\sum_{i=1}^n (a_{n,i}\ta_{n,i})^2+o(1)$. From the proof of Lemma \ref{lemgeneralqfiid}, 
$\sigma_n^2(a_n, \ta_n)$,  are uniformly bounded away from zero and infinity for  $n\ge 1$. 
Similarly, Condition \ref{generalquadratic forms} \textup{\texttt{1-dim}} holds for $({U_nb_n, \widebar{X_n\beta_n}})_{n\ge 1}$, and using Lemma \ref{lemgeneralqfiid}, we have  
$\Mp = I_p+ \|X_n \beta_n\|^{-2} L_nV_n \beta_n \beta_n^\top V_n^\top L_n +\|X_n \beta_n\|^{-2} (\kappa_{n,4}-3)U_n^\top\diag((X_n\beta_n) \odot (X_n \beta_n))U_n \in \Sp$. 
Therefore, 
we can calculate
\begin{align*}
\Sigp&=\|X_n\beta_n\|^2V_nL_n^{-1}\Mp L_n^{-1}V_n^\top\\ &=\|X_n\beta_n\|^2(X_n^\top X_n)^{-1}+\beta_n\beta_n^\top+(\kappa_{n,4}-3)V_nL_n^{-1}U_n^\top\diag\left((X_n\beta_n)\odot(X_n\beta_n)\right)U_nL_n^{-1}V_n^\top\\
&=(X_n^\top X_n)^{-1} \|y_{n}-\ep_{n}\|^2 + (X_n^\top X_n)^{-1} \left[(\kappa_{n,4}-3) \sum_{i=1}^n (y_{n,i}- \ep_{n,i})^2 x_{n,i} x_{n,i}^\top  \right](X_n^\top X_n)^{-1}+\beta_n\beta_n^\top,
\end{align*}
the claim for $\hbp$ follows from \Cref{lemgeneols}.
\end{proof}

We will also use 
the following proposition, which shows that $\hat{\Sigma}_{m,n}$ and $\hSigp$ are ratio-consistent estimators of $\Sigma_n$ and $\Sigp$, respectively. 
\begin{proposition}[Consistency of the i.i.d.~Sketching Estimator of the Covariance Matrices]\label{consest}
Under the conditions of \Cref{thiid}, we have 
\begin{equation*}
\|  \Sigma_n^{-1}\hat{\Sigma}_{m,n}-I_p \|_{\Fr} = o_{P}(1)
\,\textnormal{  and  }\, \|(\Sigp)^{-1}\hSigp-I_p \|_{\Fr} = o_{P}(1).\end{equation*}
\end{proposition}

\begin{proof}
We consider $y_n\in \nsp$ and $X_n$ of a unit spectral radius first. For each $i\in[m]$, we denote the $i$-th row of $\sqrt{m}S_{m,n}$ by $s_i^\top$.
To prove the first claim, 
due to \eqref{hsmn} and \eqref{assiidfix},    
it suffices to show that
\begin{equation}\label{89ahd1}\|\tX_{m,n}^\top \tX_{m,n} - X_n^\top X_n\|_{\Fr} = o_P(1)\end{equation} 
and
\begin{equation}\label{89ahd2}
\left\| m\tX_{m,n}^\top \diag(\tep_{m,n} \odot \tep_{m,n}) \tX_{m,n} 
- \sum_{i=1}^n \ep_{n,i}^2\left[X_n^\top X_n+(\kappa_{n,4}-3)  x_{n,i} x_{n,i}^\top\right]\right\|_{\Fr} = o_P(1).
\end{equation}

We can conclude \eqref{89ahd1} from \eqref{nc}.
Also, 
since 
$\sqrt{m}\tep_{n,i} = s_i^\top (y_n-X_n \hbs)$,
to establish \eqref{89ahd2}, it is sufficient to show that for any fixed $c\in \R^p$, $c\neq 0$,
for
$\sigma_n^2 := c^\top \left\{ \sum_{i=1}^n \ep_{n,i}^2\left[X_n^\top X_n+(\kappa_{n,4}-3)  x_{n,i} x_{n,i}^\top\right] \right\} c$,
we have
\begin{equation}\label{final}
\sigma_n^2- 
\frac{1}{m}\sum_{i=1}^m \left(c^\top X_n^\top s_i s_i^\top (y_n - X_n \hbs)\right)^2  = o_P(1).   
\end{equation}

From \eqref{iidvar} with $a_n= X_n c$ and $\ta_n= \ep_n =y_n-X_n \beta_n$, and using the law of large numbers for triangular arrays (see, for instance, Theorem 2.2.6 in \cite{durrett2019probability}), we deduce that  \begin{equation}\label{89ahf}
     \frac{1}{m} \sum_{i=1}^m \sigma_n^{-2} \left(c^\top X_n^\top s_i s_i^\top (y_n - X_n \beta_{n})\right)^2 \rightarrow_P 1.
\end{equation}
We claim that  \begin{equation}\label{89ahg}
    \frac{1}{m} \sum_{i=1}^m \sigma_n^{-2} \left(c^\top X_n^\top s_i s_i^\top  X_n (\beta_{n}-\hbs)\right)^2 \rightarrow_P 0.
\end{equation}
This claim follows from the following three facts:
\begin{equation*}\|\beta_n - \hbs\| = O_P(m^{-1/2}\sigma_n),\quad  \max_{i\in [m]} \|s_i^\top X_n\|=o_P(m^{1/2}),\quad \text{and}\quad
 c^\top X_n^\top S_{m,n}^\top S_{m,n} X_n c =  O_P(1).
\end{equation*}
The first claim follows from Lemma \ref{lemiid}; 
the second is a consequence of Markov's inequality and a union bound, and the third one follows from
$\|a_n^\top S_{m,n}^\top S_{m,n} \ta_n-a_n^\top \ta_n\|=O_P(m^{-1/2})$. 

Then, from \eqref{89ahf}, \eqref{89ahg}, 
we can conclude \eqref{final}.
Thus, the first claim of Proposition \ref{consest} holds when $\|y_n\|=1$ and $\|X_n\|=1$. 

For general $y_n$ and $X_n$, we let  $y_n' = y_n/\|y_n\|$ and $X_n'=X_n/\|X_n\|$, 
and we consider the observed data $S_{m,n}(X_n',y_n')$. 
By applying the above argument with  
$\sigma_n^2 =   
c^\top \{ \sum_{i=1}^n (\ep_{n,i}/\|y_n\|)^2$ $[X_n^\top X_n+(\kappa_{n,4}-3)  x_{n,i} x_{n,i}^\top]/\|X_n\|^2 \} c$
rescaled by $\|X_n\|$ and  $\|y_n\|$, 
and corresponding rescaled estimators,
as well as all other appropriate
quantities rescaled,
the $\|X_n\|$ and $\|y_n\|$ terms cancel out, leading to the desired result for general $X_n$ and $y_n$.

Next, we consider the second conclusion from Proposition \ref{consest}.  
Similarly to the analysis of the first conclusion, it suffices to consider $y_n\in\nsp$.
By choosing $a_n = X_n c$ and $\ta_n =X_n \beta_n$ in \eqref{iidvar}, we can apply the law of large numbers for triangular arrays (see, for instance, Theorem 2.2.6 in \cite{durrett2019probability}) to obtain that
with
 \begin{equation*}\begin{aligned}
    (\sigp)^{2} &:= (c^\top X_n^\top X_n c)(\beta_n^\top X_n^\top X_n \beta) + (c^\top X_n^\top X_n \beta_n)^2 + (\kappa_{n,4}-3)c^\top X_n^\top \diag(X_n\beta_n \odot X_n \beta_n) X_n c\\
    &=c^\top\left( \|y_{n}-\ep_{n}\|^2\left[X_n^\top X_n+(\kappa_{n,4}-3)  x_{n,i} x_{n,i}^\top  \right]+X_n^\top y_n y_n^\top X_n\right) c,
    \end{aligned}
\end{equation*}
we have
$\frac{1}{m} \sum_{i=1}^m (\sigp)^{-2} \left(c^\top X_n^\top s_i s_i^\top  X_n \beta_{n}\right)^2 \rightarrow_P 1 .$
The conclusion follows similarly 
as for sketch-and-solve above, using  that $\|\beta_n - \hbp\| = O_P(m^{-1/2}\sigp)$.
\end{proof}
Combining the above two results, the conclusion follows.

\subsection{Proof of \texorpdfstring{Lemma \ref{lemgeneralqfiid}}{Lemma \ref{lemgeneralqfiid}}}
\label{pflemgeneralqfiid}

We start by observing that
$
a_{n}^\top S_{m,n}^\top S_{m,n} \ta_{n} -a_n^\top \ta_n =\frac{1}{m}\sum_{i=1}^m (s_i^\top a_n \ta_n^\top s_i - a_n^\top \ta_n).
$
By applying equation (9.8.6) in \cite{bai2010spectral}, 
each of the above i.i.d.~centered random variables
has variance
\begin{equation}\label{iidvar}
\sigma_n^2:=\E  (s_i^\top a_n \ta_n^\top s_i - a_n^\top \ta_n)^2=  (a_n^\top a_n)(\ta_n^\top \ta_n)+(a_n^\top \ta_n)^2+ (\kappa_{n,4}-3)\sum_{i=1}^n (a_{n,i} \ta_{n,i})^2.
\end{equation}

To establish the asymptotic normality, we use the Lindeberg-Feller central limit theorem for triangular arrays
for $\xi_{n,i}:=(s_i^\top a_n \ta_n^\top s_i-a_n^\top \ta_n)/\sigma_n$, so that  $\xi_{n,1},\ldots,\xi_{n,m}$ are i.i.d.~random variables with unit variance, 
whose distribution may depend on $n$. 
To check the Lindeberg condition, 
we denote $r=1+\delta/4>1$ and use the H{\"o}lder inequality to obtain
{\beq\label{h}
\E \left[\xi_{n,1}^2 I\{|\xi_{n,1}|>\epsilon \sqrt{m}\}\right]\le  
\E [\xi_{n,1}^{2r}]^{1/r} \cdot  
\PP{|\xi_{n,1}|>\epsilon \sqrt{m}}^{1-1/r}.
\eeq}
Applying Chebyshev's inequality, we have
{\beq\label{ch}
\PP {|\xi_{n,1}|>\epsilon\sqrt{m}}\le  \frac{\E[\xi_{n,1}^2]}{\epsilon^2 m}=\frac{1}{\epsilon^2 m}.
\eeq}
Furthermore, by Lemma B.26 in \cite{bai2010spectral} 
for the centered quadratic form $\xi_{n,1}$ in $s_1$, we find $
\E [\xi_{n,1}^{2r}]\le  C(\kappa_{n,4}^r+\kappa_{n,4r})/\sigma_n^{2r},
$
where $C$ is an absolute constant. 
Therefore, 
$\E \left[\xi_{n,1}^2 I\{|\xi_{n,1}|>\epsilon \sqrt{m}\}\right]\le  
[C(\kappa_{n,4}^r+\kappa_{n,4r})]^{1/r}/\sigma_n^{2} \cdot  
(\epsilon^2 m)^{1/r-1}$.

Now consider the first condition from Condition $\mathfrak{P}$ defined in \eqref{piid}.
We observe that if $\kappa_{n,4}\ge 3$, 
then $1\le \sigma_n^2\le 2+\kappa_{n,4}-3\le \kappa_{n,4+\delta}^{4/(4+\delta)}-1<+\infty$.
Otherwise, if $1+\delta^\prime\le \kappa_{n,4}<3$, we have
\begin{align}\label{vlb}\sigma_n^2 
& = \frac{\kappa_{n,4}-1}{2}+\frac{\kappa_{n,4}-1}{2}\left(\sum_{i=1}^n a_{n,i} \ta_{n,i}\right)^2+ \frac{3-\kappa_{n,4}}{2}\left[\sum_{i\neq j} a_{n,i}^2 \ta_{n,j}^2 + \sum_{i\neq j} a_{n,i}\ta_{n,i}a_{n,j}\ta_{n,j}\right]\nonumber\\
&  = \frac{\kappa_{n,4}-1}{2}+\frac{\kappa_{n,4}-1}{2}\left(\sum_{i=1}^n a_{n,i} \ta_{n,i}\right)^2+\frac{3-\kappa_{n,4}}{4} \sum_{i\neq j} (a_{n,i}\ta_{n,j}+a_{n,j}\ta_{n,i})^2 
\ge  \frac{\delta^\prime}{2}.
\end{align}
Thus $\sigma_n^2$ are bounded uniformly away from zero. 
Hence, $\E \left[\xi_{n,1}^2 I\{|\xi_{n,1}|>\epsilon \sqrt{m}\}\right]\to 0$.

Next, consider the second condition from Condition $\mathfrak{P}$ defined in \eqref{piid}, 
In this case, since $\kappa_{n,4}\to \infty$, $\sigma_n^2/\kappa_{n,4}=\sum_{i=1}^n (a_{n,i}\ta_{n,i})^2+o(1)$ from the expression in \eqref{iidvar}, so $\sigma_n^2/\kappa_{n,4}$ are uniformly bounded away from zero and infinity.
Therefore, we can choose a $\delta''>0$ in this case such that $\sigma_n^2/\kappa_{n,4}\ge \delta''$.
Consequently, 
using H{\"o}lder's inequality again, 
and as $\kappa_{n,4+\delta}^{4/(4+\delta)}/\kappa_{n,4}=o(m^{\delta/(4+\delta)})$,
$\E \left[\xi_{n,1}^2 I\{|\xi_{n,1}|>\epsilon \sqrt{m}\}\right]$ is bounded above by
$$
\frac{C(\kappa_{n,4}^r+\kappa_{n,4r})^{1/r}}{(\epsilon^2 m)^{1-1/r}\kappa_{n,4}(\delta'\wedge\delta'')^2}\le  
\frac{C'\kappa_{n,4r}^{1/r}}{(\epsilon^2 m)^{1-1/r}\kappa_{n,4}(\delta'\wedge\delta'')^2}
=o\left(  
\frac{C'(m^{1-1/r})}{(\epsilon^2 m)^{1-1/r}(\delta'\wedge\delta'')^2}\right)\to 0.
$$
By the Lindeberg-Feller central limit theorem, the conclusion follows.

\subsection{Proof of \texorpdfstring{Lemma \ref{symiid}}{Lemma \ref{symiid}}}
By the subsequence argument as in the proof of \eqref{ssa} from \Cref{lemgeneols}, it suffices to show
that for any $\Phi\in\R^{p\times p}$,
\begin{equation}
\sqrt{m} \left(\vec(\Phi)^\top D_p^\top G_n D_p  \vec(\Phi)\right)^{-1/2} \vec(\Phi)^\top D_p^\top\vec(U_n^\top S_{m,n}^\top S_{m,n}U_n-I_p)\Rightarrow \N\left(0,1\right).
\end{equation}
Since $D_p\vec(\Phi)=\vec\left(\tilde{\Phi}\right)$, where $\tilde{\Phi}$ is an upper triangular matrix, we only need to show that for any upper triangular $\Phi\in\R^{p\times p}$,
\begin{equation*}
\sqrt{m} \left(\vec(\Phi)^\top G_n\vec(\Phi)\right)^{-1/2} \vec(\Phi)^\top \vec(U_n^\top S_{m,n}^\top S_{m,n}U_n-I_p)\Rightarrow \N\left(0,1\right).
\end{equation*}
For each $i\in[m]$, we denote by $s_i^\top$
the $i$-th row of $\sqrt{m}S_{m,n}$, where $S_{m,n}^\top=(s_1,\ldots,s_m)$. 
Since
$S_{m,n}^\top S_{m,n}=\frac{1}{m}\sum_{i=1}^m s_is_i^\top$, 
we can write 
$$
\begin{aligned}
&\sqrt{m}\vec(\Phi)^\top\vec(U_n^\top S_{m,n}^\top S_{m,n}U_n-I_p)=\sqrt{m}\tr\left(\Phi \left(U_{n}^\top S_{m,n}^\top S_{m,n}U_{n}-I_p\right)\right)
= \frac{1}{\sqrt{m}}\sum_{i=1}^m\left( s_i^\top U_{n} \Phi U_{n}^\top s_i-\tr(\Phi)\right).
\end{aligned}
$$
Denoting $\xi_{n,i} = s_i^\top U_{n} \Phi U_{n}^\top s_i-\tr(\Phi)$
for all $i\in[n]$,
$\xi_{n,1},\ldots,\xi_{n,m}$ are i.i.d.~random variables.
According to (9.8.6) in \cite{bai2010spectral},
$$
\Var{\xi_{n,1}}=(\kappa_{n,4}-3)\sum_{i=1}^n (U_{n}\Phi U_{n}^\top)_{ii}^2+2\tr\left(U_{n}\Phi U_{n}^\top (U_{n}\Phi U_{n}^\top)^\top\right).
$$
Now
$$
\tr\left(U_{n}\Phi U_{n}^\top (U_{n}\Phi U_{n}^\top)^\top\right)
=\tr(\Phi U_{n}^\top U_{n}\Phi^\top U_{n}^\top U_{n})=\tr(\Phi \Phi^\top)=\vec(\Phi)^\top\vec(\Phi)=\vec(\Phi)^\top P_p \vec(\Phi),
$$
and thus 
$
2\tr(U_{n}\Phi U_{n}^\top (U_{n}\Phi U_{n}^\top)^\top)=\vec(\Phi)^\top (I_{p^2}+P_p) \vec(\Phi).
$

On the other hand,  $(\kappa_{n,4}-3)\sum_{i=1}^n (U_{n}\Phi U_{n}^\top)_{ii}^2=\vec(\Phi)^\top \Gamma_n \vec(\Phi)$, where for every $k_1,k_2,k_3,k_4\in[p]$, $$(\Gamma_n)_{(k_1k_2),(k_3k_4)}=(\kappa_{n,4}-3)\sum_{h=1}^n (U_{n})_{hk_1}(U_{n})_{hk_2}(U_{n})_{hk_3}(U_{n})_{hk_4}.$$
Therefore, 
$\Var{\xi_{n,1}}=\vec(\Phi)^\top G_n\vec(\Phi)$.

In order to apply the central limit theorem for triangular arrays $(1/\sqrt{m})\cdot\sum_{i=1}^m \left(\xi_{n,i}/\sqrt{\Var{\xi_{n,i}}}\right)$, we need to verify the Lindeberg-Feller condition that for every fixed $\epsilon>0$, 
$\E\left[\frac{\xi_{n,1}^2}{\Var{\xi_{n,1}}} I \left\{|\xi_{n,1}|>\epsilon\sqrt{m\Var{\xi_{n,1}}}\right\}\right]\rightarrow 0.$
With $r=1+\delta/4>1$, 
by an argument identical to the one in the proof of Lemma \ref{lemgeneralqfiid}, 
we can conclude \eqref{h} and the first inequality in \eqref{ch}.
Then 
$\E[\xi_{n,1}^2]$ and $\E [\xi_{n,1}^{2r}]$ can be bounded as in that argument via Lemma B.26 in \cite{bai2010spectral},
and using that $\Phi$ does not depend on $m$ and $n$.

Now consider the first condition from Condition $\mathfrak{P}$ defined in \eqref{piid}.
One can check as in the proof of Lemma \ref{lemgeneralqfiid}, calculating similarly as in \eqref{vlb} that
$(\delta'/2)\cdot\|\Phi\|_{\Fr}^2\le \vec(\Phi)^\top G_n\vec(\Phi)
$ $\le (\kappa_{n,4+\delta}^{4/(4+\delta)}-1)\|\Phi\|_{\Fr}^2$. 

Next, consider the second condition from Condition $\mathfrak{P}$ defined in \eqref{piid}, so that  $\kappa_{n,4}\to\infty$, and $\kappa_{n,4+\delta}^{4/(4+\delta)}/\kappa_{n,4}=o(m^{\delta/(4+\delta)})$. Moreover, $\vec(\Phi)^\top G_n\vec(\Phi)/\kappa_{n,4}=\sum_{i=1}^n (U_{n}\Phi U_{n}^\top)_{ii}^2+o(1)$, 
so $\vec(\Phi)^\top G_n\vec(\Phi)/\kappa_{n,4}$ are uniformly bounded away from zero and infinity,
as in the previous proof for Lemma \ref{lemgeneralqfiid}.
Therefore, 
$$\E\left[\frac{\xi_{n,1}^2}{\Var{\xi_{n,1}}} I \left\{|\xi_{n,1}|>\epsilon\sqrt{m\Var{\xi_{n,1}}}\right\}\right]\le \frac{C'\kappa_{n,4r}^{1/r}}{(\epsilon^2 m)^{1-1/r}\delta^{\prime 2}} \to 0,$$ 
verifying the Lindeberg-Feller condition.

\subsection{Proof of \texorpdfstring{Theorem \ref{PCAHaar}}{Theorem \ref{PCAHaar}}}
By Lemma \ref{lemsymqfhaar}, if $(U_n)_{n\ge 1}$ satisfies Condition $\mathfrak{P}'$,
 Condition \ref{generalquadratic forms} \texttt{Sym-1}
 holds with $G=I_{p^2}+P_p$ and $\tau_{m,n}=1-\gamma_n$.
 Therefore, Condition \ref{generalquadratic forms} \textup{\texttt{Sym-1'}}
 holds with $\alpha=0$.
Then under Condition \ref{spectralcondition} \textup{\texttt{B}}, we 
finish the proof based on \Cref{infPCAgeneral}.

\subsection{Proof of \texorpdfstring{Theorem \ref{thhaar}}{Theorem \ref{thhaar}}}
From Lemma \ref{lemgeneralqfhaar}, we know that Condition \ref{generalquadratic forms} \textup{\texttt{1-dim'}} holds for $({U_nb_n, \bar{\ep}_n})_{n\ge 1}$ with $\alpha=0$.
Then by \Cref{lemgeneols}, we conclude the first result, 
and the second one follows similarly.

\subsection{Proof of \texorpdfstring{Lemma \ref{lemgeneralqfhaar}}{Lemma \ref{lemgeneralqfhaar}}}
If $a_n \neq \pm\ta_n$,
define $a'_n = (I-a_n a_n^\top) \ta_n/\|(I-a_n a_n^\top) \ta_n\|$, 
otherwise, define $a_n'$ as an arbitrary unit length vector orthogonal to $a_n$, 
so that we have 
\eqref{decomptan}.
Thus, we can write \begin{equation}\label{eq096}\begin{aligned}
 a_n^\top S_{m,n}^\top S_{m,n}\ta_n =(a_n^\top \ta_n) a_n^\top   S_{m,n}^\top S_{m,n} a_n + (1-(a_n^\top \ta_n)^2)^{1/2}a_n^\top S_{m,n}^\top S_{m,n} a'_n.
\end{aligned}
\end{equation}
Now, let $Z = (Z_1,\ldots, Z_n)^\top, \tZ = (\tZ_1, \cdots, \tZ_n)^\top$ be two independent vectors, each containing independent standard normal entries. 
For a vector $v \in \R^n$ and some
integers $m,m'$ with
$1 \le m\le m'\le n$, we denote $v_{m:m'} = (v_m,\ldots, v_m')^\top$.
We then claim that we have the following representation:
\begin{equation}\label{rep}
   S_{m,n} a_n =_d  \gamma_{n}^{-1/2} \frac{Z_{1:m}}{\|Z\|},\quad S_{m,n} a'_n =_d  \gamma_{n}^{-1/2} \frac{\tZ_{1:m}- \frac{Z^\top \tZ}{\|Z\|^2}Z_{1:m}}{\sqrt{\|\tZ\|^2-\frac{(Z^\top \tZ)^2}{\|Z\|^2}}}.
\end{equation}
This follows because 
$\gamma_n^{1/2}S_{m,n}(a_n, a'_n) $ has the same distribution as an 
$m \times 2$ submatrix of a Haar distributed matrix. 
By generating the Haar matrix using the Gram-Schmidt process starting from $(Z,\tZ)$, we obtain the representation from \eqref{rep}. 
Specifically, the first column is given by $w_1= Z/\|Z\|$, and the second is obtained as $[\tZ-(w_1^\top \tZ) w_1]/\|\tZ-(w_1^\top \tZ) w_1\|$. Taking the first $m$ coordinates of these columns gives us the representation in \eqref{rep}.

We define the following quantities: \begin{equation*}
    \mathcal{X}_{n,1}= m^{-1}\|Z_{1:m}\|^2,\,
\mathcal{X}_{n,2}=(n-m)^{-1}\|Z_{(m+1):n}\|^2,\,
\mathcal{C}_{n,1}= m^{-1}\sum_{i=1}^m Z_i \tZ_i,\, \mathcal{C}_{n,2}= (n-m)^{-1}\sum_{i=m+1}^n Z_i \tZ_i,
\end{equation*}
and $\tilde{\mathcal{X}}_{n,1}$, $\tilde{\mathcal{X}}_{n,2}$ are defined by replacing the coordinates of $Z$ in the definitions of $\mathcal{X}_{n,1}, \mathcal{X}_{n,2}$ with the corresponding ones of $\tZ$. 
According to the 
central limit theorem, the asymptotic distribution
as $m, n-m\to\infty$ 
of the vector
\begin{align}
   &\left(\sqrt{m}(\mathcal{X}_{n,1}-1) ,
\sqrt{n-m}(\mathcal{X}_{n,2}-1), 
\sqrt{m}\mathcal{C}_{n,1},\sqrt{n-m}\mathcal{C}_{n,2}, 
\sqrt{m}(\tilde{\mathcal{X}}_{n,1}-1),
\sqrt{n-m}(\tilde{\mathcal{X}}_{n,2}-1)\right) \label{ldn}
\end{align}
follows a multivariate normal distribution with mean zero and a diagonal covariance matrix with the vector $(2,2,1,1,2,2)$ on the diagonal. 
Finally, we have 
\begin{equation*}    \|S_{m,n}a_{n}\|^2 = \frac{n}{m}\frac{\|Z_{1:m}\|^2}{\|Z_{1:m}\|^2+\|Z_{(m+1):n}\|^2} =\frac{\mathcal{X}_{n,1}}{\gamma_{n} \mathcal{X}_{n,1}+(1-\gamma_{n})\mathcal{X}_{n,2}},\end{equation*} 
and
\begin{equation*} \begin{aligned} a_n^\top &S_{m,n}^\top S_{m,n}a'_n = \frac{\mathcal{C}_{n,1}-\frac{ \mathcal{X}_{n,1}}{\gamma_n\mathcal{X}_{n,1}+(1-\gamma_n)\mathcal{X}_{n,2}}(\gamma_n\mathcal{C}_{n,1}+(1-\gamma_n)\mathcal{C}_{n,2})}{\sqrt{\gamma_n\mathcal{X}_{n,1}+(1-\gamma_n)\mathcal{X}_{n,2}}\left[\gamma_n \mathcal{X}_{n,1}+(1-\gamma_n)\mathcal{X}_{n,2}-\frac{(\gamma_n\mathcal{C}_{n,1}+(1-\gamma_n)\mathcal{C}_{n,2})^2}{\gamma_n\mathcal{X}_{n,1}+(1-\gamma_n)\mathcal{X}_{n,2}}\right]}.\end{aligned} \end{equation*} 
Combining \eqref{eq096} with the above two equations, we define the function
$g:(0,\infty)^6\to \R$
such that for all values of 
$\mathcal{X}_{n,p,i}$, $\mathcal{C}_{n,p,i}$,
$\tilde{\mathcal{X}}_{n,p,i}$, for $i=1,2$, we have 
\begin{equation*}  a_n^\top S_{m,n}^\top S_{m,n}\ta_n = g(\mathcal{X}_{n,1},\mathcal{X}_{n,2},\mathcal{C}_{n,1},\mathcal{C}_{n,2},\tilde{\mathcal{X}}_{n,1},\tilde{\mathcal{X}}_{n,2}).\end{equation*} 
By direct calculation, 
we find that \begin{align*}
&\nabla g(1,1,0,0,1,1) = \big(a_n^\top \ta_n(1-\gamma_n), -a_n^\top \ta_n(1-\gamma_n),\sqrt{1-(a_n^\top \ta_n)^2}(1-\gamma_n), -\sqrt{1-(a_n^\top \ta_n)^2}(1-\gamma_n),0,0\big).
\end{align*}
Using the delta method, we have \begin{equation*}\begin{aligned}\sqrt{\frac{m}{1-\gamma_n}}& \left(g(\mathcal{X}_{n,1},\mathcal{X}_{n,2},\mathcal{C}_{n,1},\mathcal{C}_{n,2},\tilde{\mathcal{X}}_{n,1},\tilde{\mathcal{X}}_{n,2}) -g(1,1,0,0,1,1)\right)\\&=a_n^\top \ta_n \sqrt{1-\gamma_n}\sqrt{m}(\mathcal{X}_{n,1}-1)-a_n^\top \ta_n\sqrt{\gamma_n}\sqrt{n-m}(\mathcal{X}_{n,2}-1)\\
& +\sqrt{1-(a_n^\top \ta_n)^2}\sqrt{1-\gamma_n}\sqrt{m}\mathcal{C}_{n,1}-\sqrt{1-(a_n^\top \ta_n)^2}\sqrt{\gamma_n}\sqrt{n-m}\mathcal{C}_{n,2}+o_P(1).
\end{aligned}\end{equation*}
Dividing both sides by $\sqrt{1+(a_n^\top \ta_n)^2}$ and using the limiting distribution in \eqref{ldn}, 
we can conclude the desired result.

\subsection{Proof of \texorpdfstring{Lemma \ref{lemsymqfhaar}}{Lemma \ref{lemsymqfhaar}}}
We denote $\Tilde{U}_n=(U_n,U_{n,\perp})$, 
so 
$S_{m,n}U_n=S_{m,n}\Tilde{U}_n(I_p,0)^\top$. 
Since the
distribution of $S_{m,n}$ is rotationally invariant, i.e., $S_{m,n}=_d S_{m,n}\Tilde{U}_n$,
we have
$$
U_n^\top S_{m,n}^\top S_{m,n}U_n=(I_p,0)\Tilde{U}_n^\top S_{m,n}^\top S_{m,n}\Tilde{U}_n(I_p,0)^\top=_d (I_p,0)S_{m,n}^\top S_{m,n}(I_p,0)^\top.
$$

We can represent $S_0=(ZZ^\top)^{-1/2}Z$, where $Z$ is an $m\times n$ matrix with i.i.d.~$\N(0,1)$ entries, see for example, Theorem 8.2.5 in \cite{gupta1999matrix}. Thus, we
can write 
$S_0=(Z_1 Z_1^\top+Z_2 Z_2^\top)^{-1/2}(Z_1, Z_2)$, where $Z_1\in\R^{m\times p}$ and $Z_2\in\R^{m\times(n-p)}$. 
Since $S_{m,n}=\gamma_{n}^{-1/2} S_0$,
we have
$$
(I_p,0)S_{m,n}^\top S_{m,n}(I_p,0)^\top
=\gamma_{n}^{-1} Z_1^\top(Z_1Z_1^\top+Z_2Z_2^\top)^{-1}Z_1.
$$
Applying the Sherman-Morrison-Woodbury formula, we obtain
$$
P_1:=Z_1^\top(Z_1Z_1^\top+Z_2Z_2^\top)^{-1}Z_1=Z_1^\top(Z_2Z_2^\top)^{-1}Z_1\left(I+Z_1^\top(Z_2Z_2^\top)^{-1}Z_1\right)^{-1}.
$$
Now, to show that $Z_1^\top(Z_2Z_2^\top)^{-1}Z_1$ is asymptotically normal,
we have that if the following almost sure limits exist:
\begin{align}\label{aslimitforhaar}
\lim_{m\to\infty}\frac{1}{m} \sum_{i=1}^m \left(\left(Z_2Z_2^\top/(n-p)\right)^{-1}_{ii}\right)^2&=\omega\ ; 
\lim_{m\to\infty}\frac{1}{m} \tr\left(\left(Z_2Z_2^\top/(n-p)\right)^{-2}\right)=\theta\ ,
\end{align}
then
by
Theorem 11.8 in \cite{yao2015large}  with $A_n=\left(Z_2Z_2^\top/(n-p)\right)^{-1}$ and $W=Z_1^\top$,
\beq\label{rmtor}
R_m:=\frac{1}{\sqrt{m}}\left(Z_1^\top\left(Z_2Z_2^\top/(n-p)\right)^{-1}Z_1-\tr\left(\left(Z_2Z_2^\top/(n-p)\right)^{-1}\right)I_p\right)\Rightarrow R,
\eeq
where $R$ follows a Gaussian distribution described below. 
By the Skorokhod representation theorem,
instead of the
almost sure limits in 
\eqref{aslimitforhaar}, it is enough to assume
convergence in probability, i.e., 
\begin{align}\label{problimitforhaar}
\frac{1}{m} \sum_{i=1}^m \left(\left(Z_2Z_2^\top/(n-p)\right)^{-1}_{ii}\right)^2&\to_P\omega;\qquad 
\frac{1}{m} \tr\left(\left(Z_2Z_2^\top/(n-p)\right)^{-2}\right)\to_P\theta.
\end{align}

Let $x_{i}$, $i \in [p]$, represent i.i.d.~random variables with a standard normal distribution.
Then, 
recalling Wick's formula $\E(x_{i}x_{j}x_{i'}x_{j'})=\delta_{ij}\delta_{i'j'}+\delta_{ii'}\delta_{jj'}+\delta_{ij'}\delta_{ji'}$, according to Lemma 11.7 of \cite{yao2015large},
the covariance matrix of $R$ is given by
$$
\Cova[R_{ij},R_{i'j'}]=\omega\left\{\E(x_{i}x_{j}x_{i'}x_{j'})-\delta_{ij}\delta_{i'j'}\right\}+(\theta-\omega)\left\{\delta_{ij}\delta_{i'j}+\delta_{ii'}\delta_{jj'}\right\}=\theta\left\{\delta_{ij'}\delta_{i'j}+\delta_{ii'}\delta_{jj'}\right\}.
$$

Now, $Z_2Z_2^\top\sim W_m(n-p, I_m)$ follows a Wishart distribution. Hence, 
$(Z_2Z_2^\top)^{-1}\sim IW_m(n-p+m+1, I_m)$ follows an inverted Wishart distribution, 
see e.g., Theorem 3.4.1 in \cite{gupta1999matrix}. 
Furthermore, Corollary 3.4.2.1 in \cite{gupta1999matrix} allows us to deduce that each diagonal element of $(Z_2Z_2^\top)^{-1}$ follows an inverted Gamma distribution $Q_i:=(Z_2Z_2^\top)^{-1}_{ii}\sim \textnormal{InvGamma}(\frac{n-p-m+1}{2},\frac{1}{2})$. 
Using the properties of this distribution,  
$$
\E (n-p)^2Q_i^2=\frac{(n-p)^2}{(n-p-m-1)(n-p-m-3)}\to \frac{1}{(1-c)^2},
$$
as well as
$$
\Var{Q_i^2}=\E Q_i^4-(\E Q_i^2)^2=\frac{8(n-p-m-4)}{(n-p-m-1)^2(n-p-m-3)^2(n-p-m-5)(n-p-m-7)};
$$
Applying the Cauchy-Schwarz inequality, we obtain
$$
\begin{aligned}
&\E{\left[(n-p)^2\left(\frac{1}{m}\sum_{i=1}^m(Q_i^2-\E Q_i^2)\right)\right]^2}
\le  \frac{(n-p)^4}{m}\E{\sum_{i=1}^m(Q_i^2-\E Q_i^2)^2}= (n-p)^4 \Var{Q_1^2}\\
&\qquad\qquad= \frac{8(n-p-m-4)(n-p)^4}{(n-p-m-1)^2(n-p-m-3)^2(n-p-m-5)(n-p-m-7)},
\end{aligned}
$$
which tends to zero as $n,m\to\infty$ with $m/n\to c>0$. Therefore, by using Chebyshev's inequality, we conclude that
$$
\lim_{m\to\infty}\frac{1}{m} \sum_{i=1}^m \left((Z_2Z_2^\top/(n-p))^{-1}_{ii}\right)^2-\E (n-p)^2Q_1^2\rightarrow_P 0.
$$
Consequently, we have $\omega=1/(1-c)^2$.

Furthermore, $Z_2Z_2^\top/(n-p)\sim W_m(n-p, \frac{1}{n-p}I_m)$, and  
by equation (9.3.9.g) in
Section 9.3.9 of \cite{holgersson2020collection} it follows that $\theta=1/(1-c)^3$, and the proof of \eqref{problimitforhaar} is complete.
As a result, \eqref{rmtor} follows. 

Moreover, based on Theorem 4.2 in \cite{pielaszkiewicz2020mixtures}, we conclude that $$\frac{n-p}{\sqrt{m}}\left(\tr\left((Z_2Z_2^\top)^{-1}\right)-\frac{m}{n-m-p-1}\right)\rightarrow_P 0.$$
Thus
$
\sqrt{m}\left(Z_1^\top (Z_2Z_2^\top)^{-1} Z_1-\frac{m}{n-m-p-1}I_p\right)\Rightarrow cR.
$
By the delta method, we can deduce that
$$
\sqrt{m}\left(Z_1^\top (Z_2Z_2^\top)^{-1} Z_1(I+Z_1^\top (Z_2Z_2^\top)^{-1} Z_1)^{-1}-\frac{m}{n-p-1}I_p\right)\Rightarrow (1-c)^2cR.
$$
Hence, 
$
\sqrt{m}\left(P_1-\frac{m}{n-p-1}I_p\right)\Rightarrow (1-c)^2cR,
$
and thus 
$$
\sqrt{m}\left(\frac{m}{n}U_{n}^\top S_{m,n}^\top S_{m,n}U_{n}-\frac{m}{n-p-1}I_p\right)\Rightarrow (1-c)^2cR.
$$
Equivalently,
$
\sqrt{m}\left(U_{n}^\top S_{m,n}^\top S_{m,n}U_{n}-I_p\right)\Rightarrow (1-c)^2R.
$
Since
the covariance matrix of $\vec(R)$ is $\theta(I_{p^2}+P_p)$, that
of $\vec\left((1-c)^2R\right)$ is $(1-c)^4\theta(I_{p^2}+P_p)=(1-c)(I_{p^2}+P_p)$.
Therefore, by Slutsky's theorem, 
Condition \ref{generalquadratic forms} \textup{\texttt{Sym}}
holds for 
$G=I_{p^2}+P_p$.


\subsection{PCA after Sketching with i.i.d.~Gaussian Matrices}\label{seciid}
Here we  
present how our result from Section \ref{pcaiid} 
simplifies for a Gaussian sketch, 
recovering results of \cite{anderson1963asymptotic} for individual eigenvalues.
\begin{corollary}[Inference for Eigenvalues and Eigenvectors with i.i.d.~Gaussian Sketching]\label{PCAiid}
Let $(X_n)_{n\ge  1}$ satisfy Condition \ref{spectralcondition} \textup{\texttt{B}}, and 
the sketching matrices $m^{1/2} S_{m,n}$ have i.i.d.~standard normal entries.
Then,
\begin{equation}\label{singvaliidGaussian}
\sqrt{\frac{m}{2}}\hat{\Lambda}_{m,n,i}^{-1}\left(\hat{\Lambda}_{m,n,i}-\Lambda_{n,i}\right)\Rightarrow \N(0,1),
\end{equation}
and for any vector $c\in\psp$ satisfying $\limsup_{n\to\infty}(c^\top v_{n,i})^2<1$,
\begin{equation}\label{singveciidGaussian}
\sqrt{m}\left(\sum_{k\neq i}\frac{\hat{\Lambda}_{m,n,i}\hat{\Lambda}_{m,n,k}}{(\hat{\Lambda}_{m,n,i}-\hat{\Lambda}_{m,n,k})^2} (c^\top\hat{v}_{m,n,k})^2\right)^{-1/2}c^\top(\hat{v}_{m,n,i}-v_{n,i})\Rightarrow \N(0,1).
\end{equation}
\end{corollary}
This result is implied by \Cref{iid1}.

\subsection{Additional Experimental Results}

In this section, we present some additional experimental results.
For each setting, we use 500 Monte Carlo trials. 
For the simulated datasets given in \Cref{secsimulation}, 
we use $p=15, n=2048, c=(1,0,\ldots,0)^\top$ as in the figures in the main text, 
and vary $m$ from 200 to 1600. 

In 
\Cref{coverage_merge,realdata_conf_merge_new}, we show coverage results of 95\% intervals
using 
sketched PCA and LS, for various sketching methods and datasets; 
along with 
95\% Clopper-Pearson intervals for the coverage.
We provide details in the figure caption.

\begin{figure}[ht]
    \begin{minipage}{0.26\linewidth}
 	\vspace{3pt}
 	\centerline{\includegraphics[width=\textwidth]{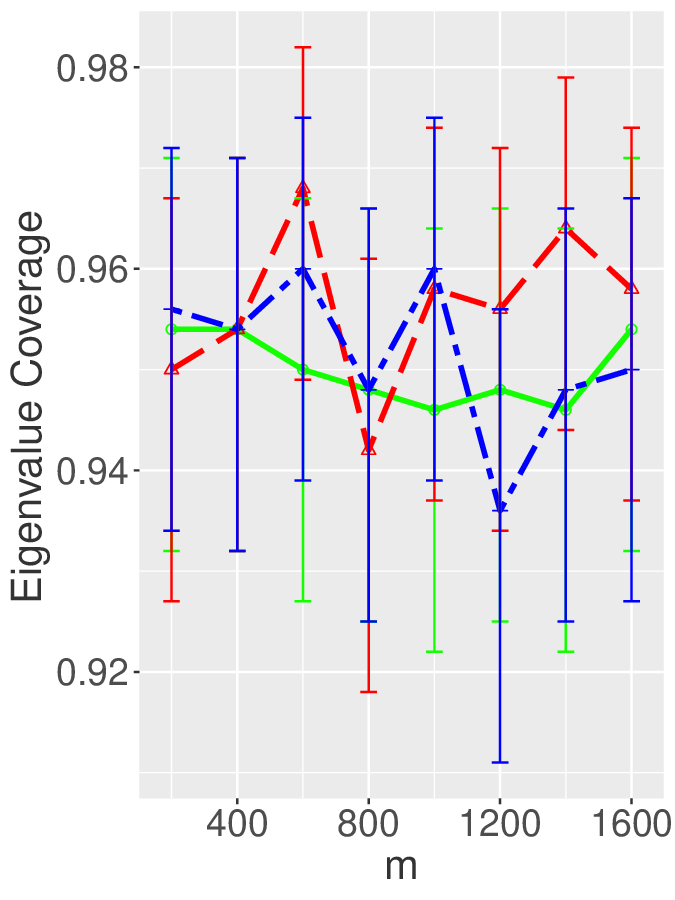}}
 	\vspace{3pt}
 	\centerline{\includegraphics[width=\textwidth]{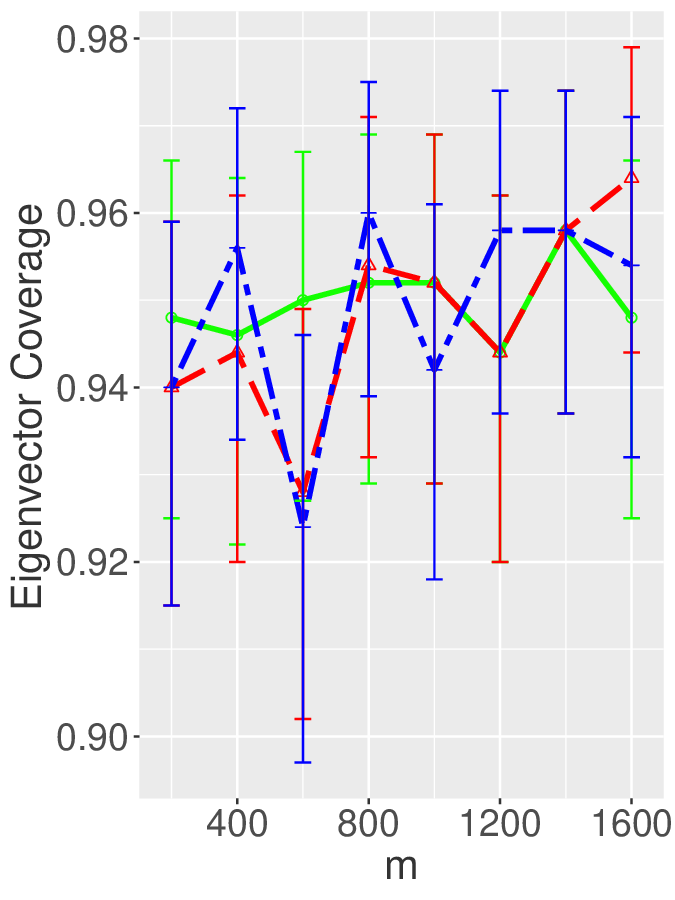}}
 	\vspace{3pt}
 	\centerline{(a) $k=2$, Case 1}
    \end{minipage}
    \begin{minipage}{0.23\linewidth}
 	\vspace{3pt}
 	\centerline{\includegraphics[width=\textwidth]{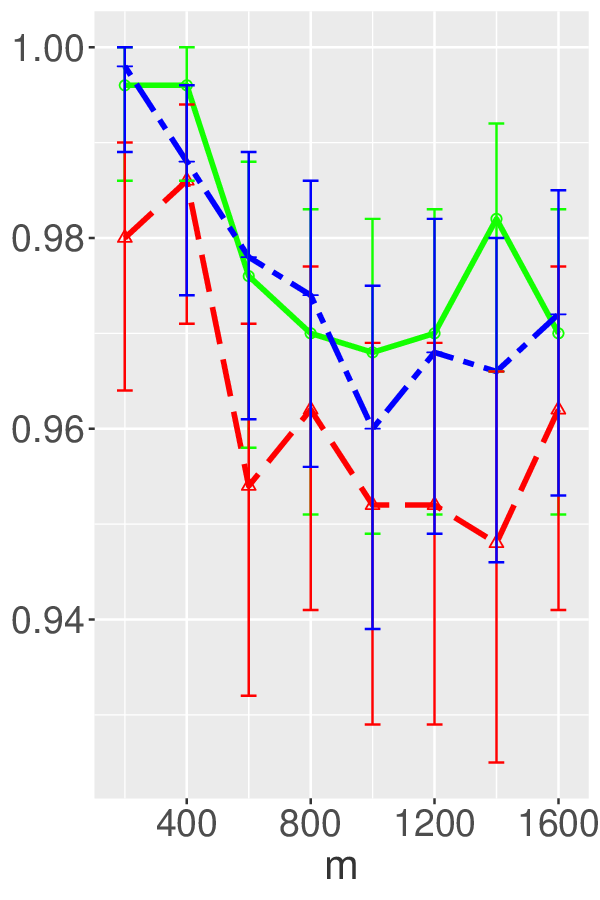}}
 	\vspace{3pt}
 	\centerline{\includegraphics[width=\textwidth]{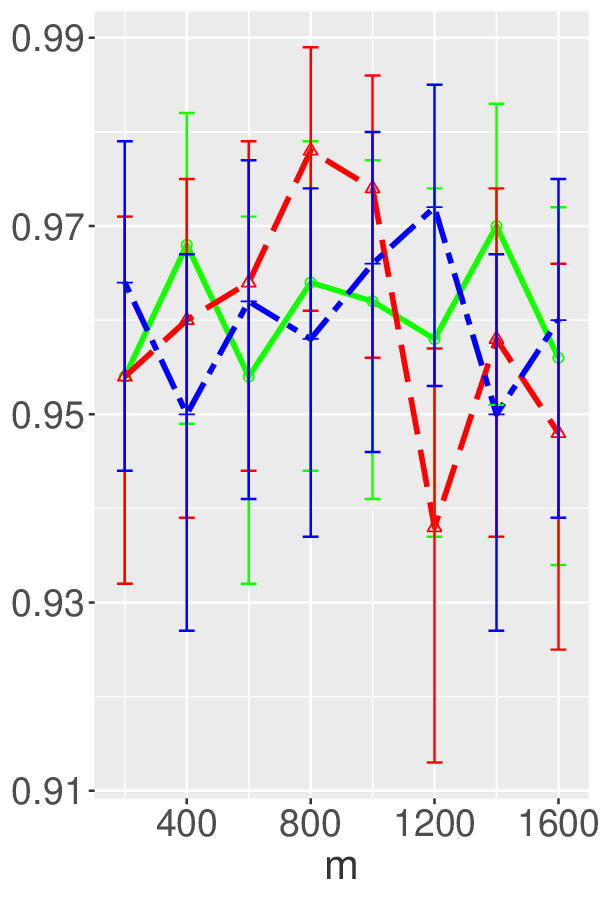}}
 	\vspace{3pt}
 	\centerline{(b) $k=1$, Case 2}
    \end{minipage}
    \begin{minipage}{0.23\linewidth}
 	\vspace{3pt}
 	\centerline{\includegraphics[width=\textwidth]{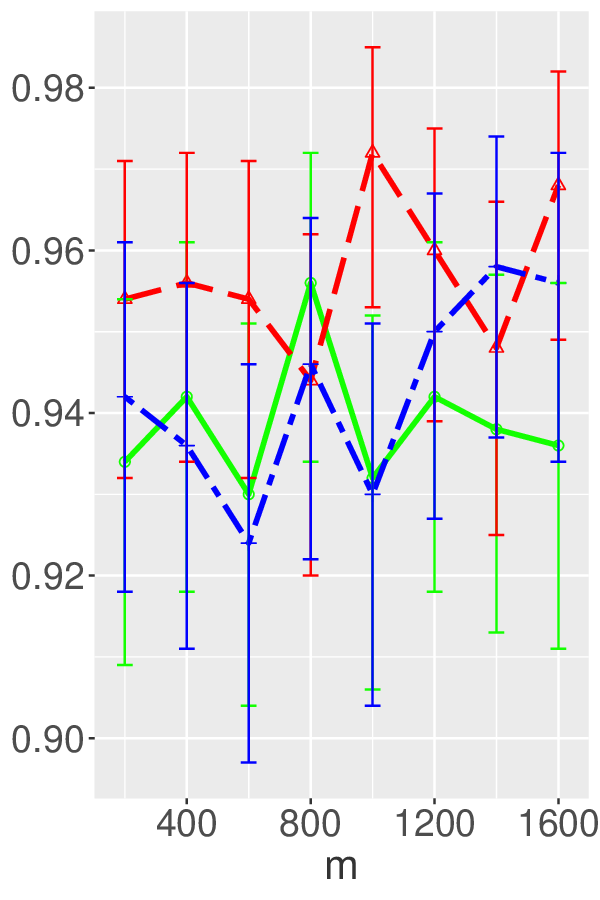}}
 	\vspace{3pt}
 	\centerline{\includegraphics[width=\textwidth]{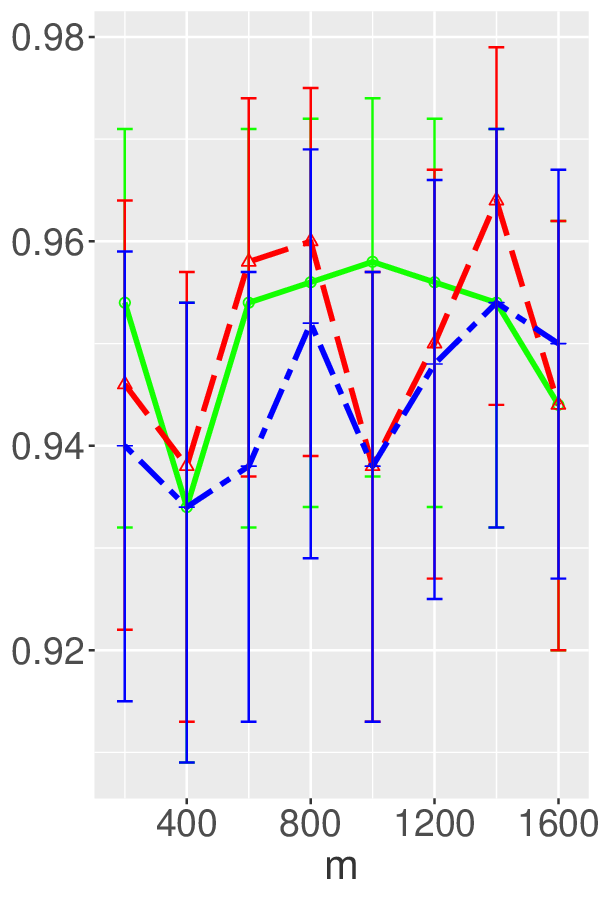}}
 	\vspace{3pt}
 	\centerline{(c) $k=1$, Case 3}
    \end{minipage}
    \begin{minipage}{0.23\linewidth}
 	\vspace{3pt}
 	\centerline{\includegraphics[width=\textwidth]{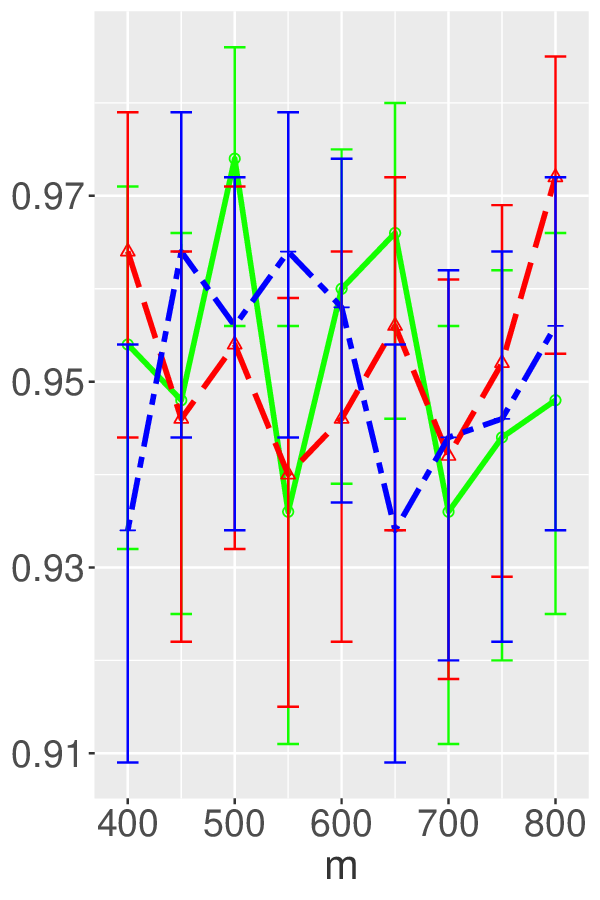}}
 	\vspace{3pt}
 	\centerline{\includegraphics[width=\textwidth]{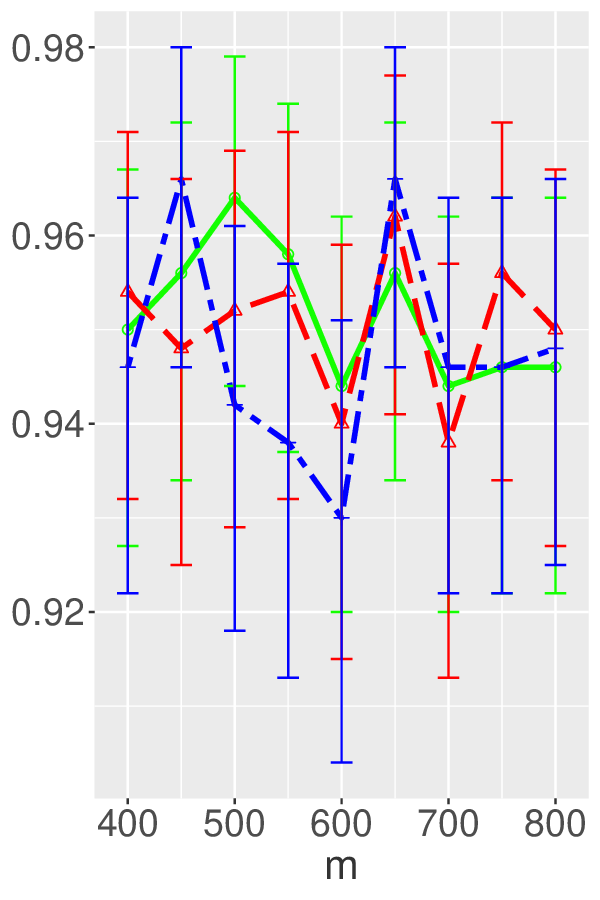}}
 	\vspace{3pt}
 	\centerline{(d) $k=1$, HGDP Dataset}
    \end{minipage}
    \caption{Coverage rates of $\lambda_k$ and $c^\top v_k$ in a variety of cases, using the same
    protocol as in \Cref{n2048p15_Case1_coverage_merge}:
    (a) $k=2$, Case 1. (b) $k=1$, Case 2. (c) $k=1$, Case 3. (d) $k=1$, HGDP dataset.}
    \label{coverage_merge}
\end{figure}

\begin{figure}[ht]
    \begin{minipage}{0.48\linewidth}
    \vspace{3pt}
 	\centerline{\includegraphics[width=\textwidth]{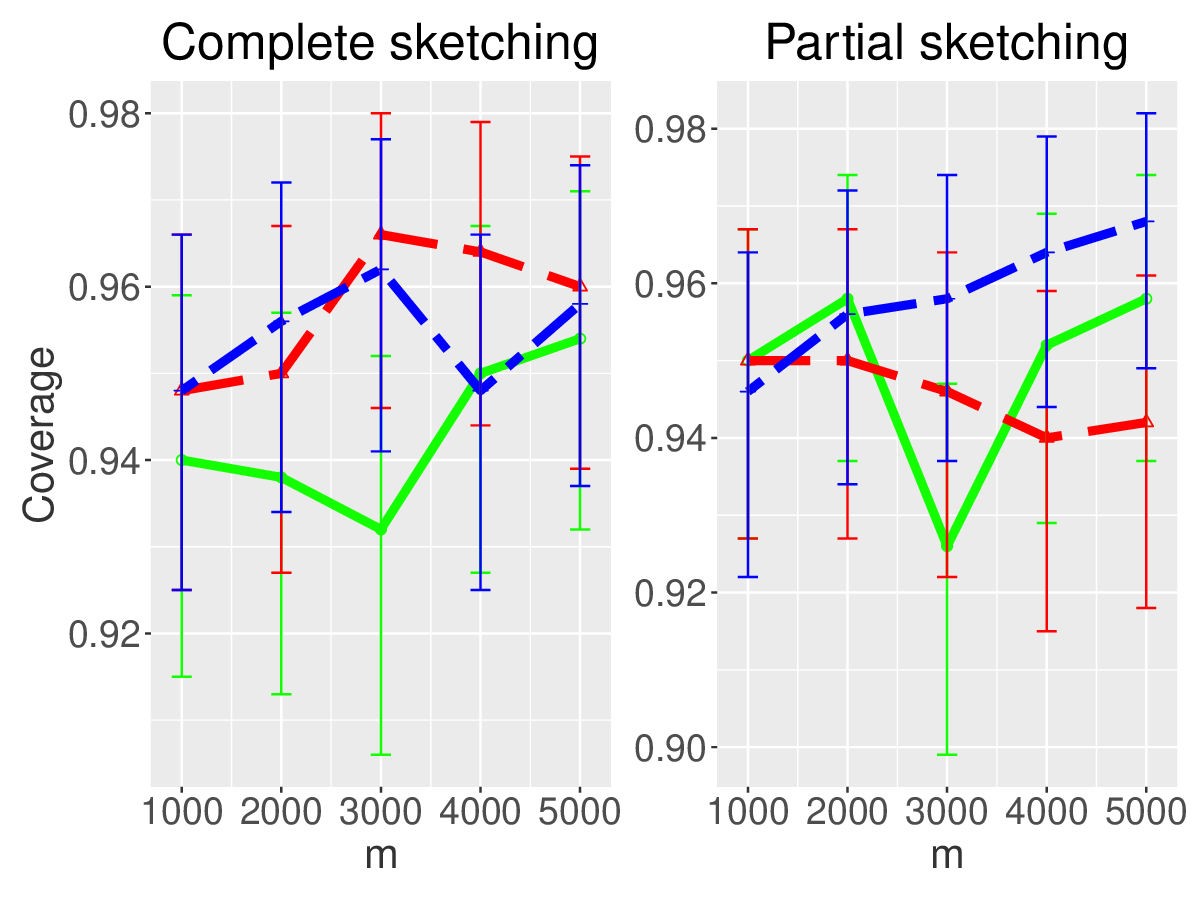}}
 	\vspace{3pt}
 	
    \end{minipage}
    \begin{minipage}{0.48\linewidth}
        \vspace{3pt}
 	\centerline{\includegraphics[width=\textwidth]{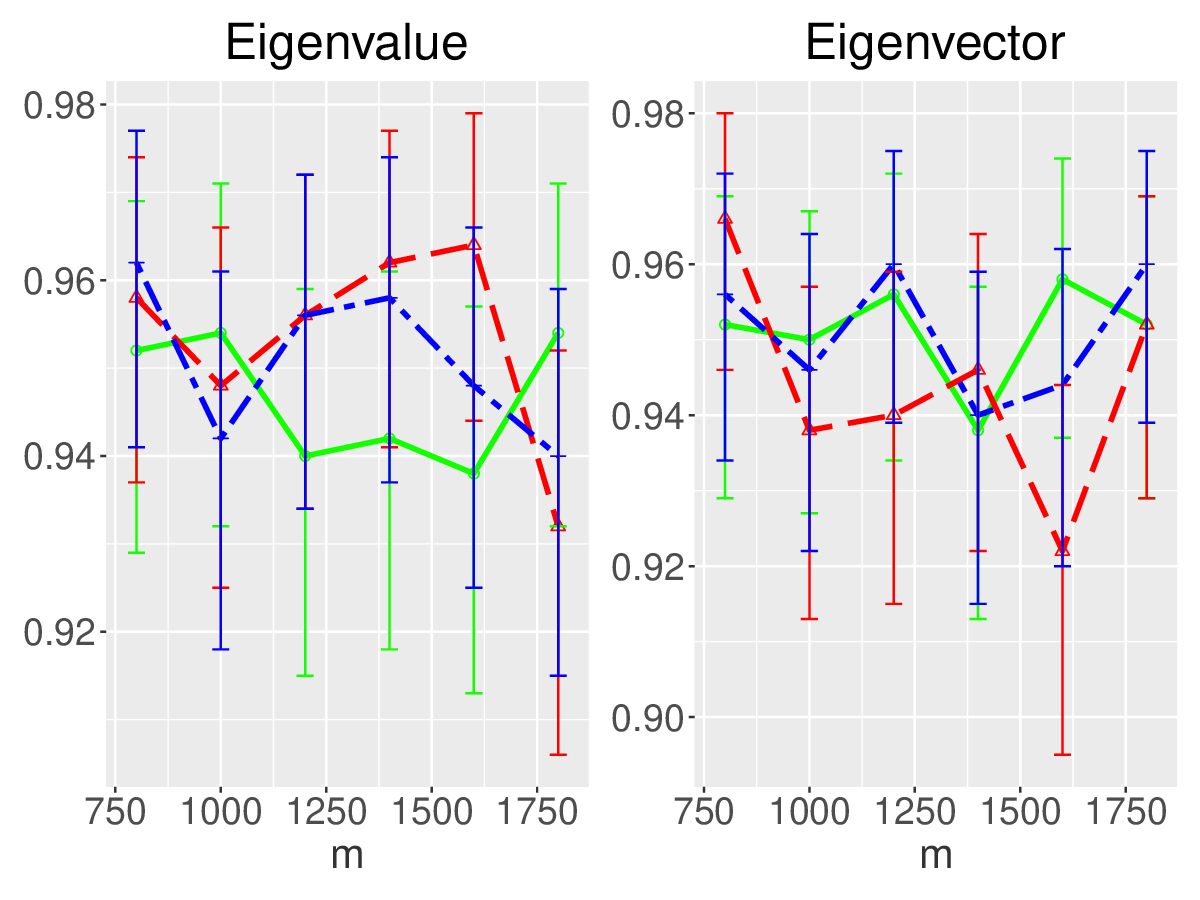}}
 	\vspace{3pt}
    \end{minipage}
    
    \caption{Coverage of 95\% intervals for the first coordinate of the least squares solution using the New York flights dataset; as well as the largest eigenvalue and associated eigenvector 
    from the Million Songs dataset. 
    Also shown are 95\% Clopper-Pearson intervals for the coverage, where we use the same
    protocol as in \Cref{n2048p15_Case1_coverage_merge}.}
    \label{realdata_conf_merge_new}
\end{figure}

In \Cref{var_merge},
we show the empirical variances of sketched PCA solutions over 500 Monte Carlo trials 
using various sketching methods and datasets.

\begin{figure}[ht]
    \begin{minipage}{0.26\linewidth}
 	\vspace{3pt}
 	\centerline{\includegraphics[width=\textwidth]{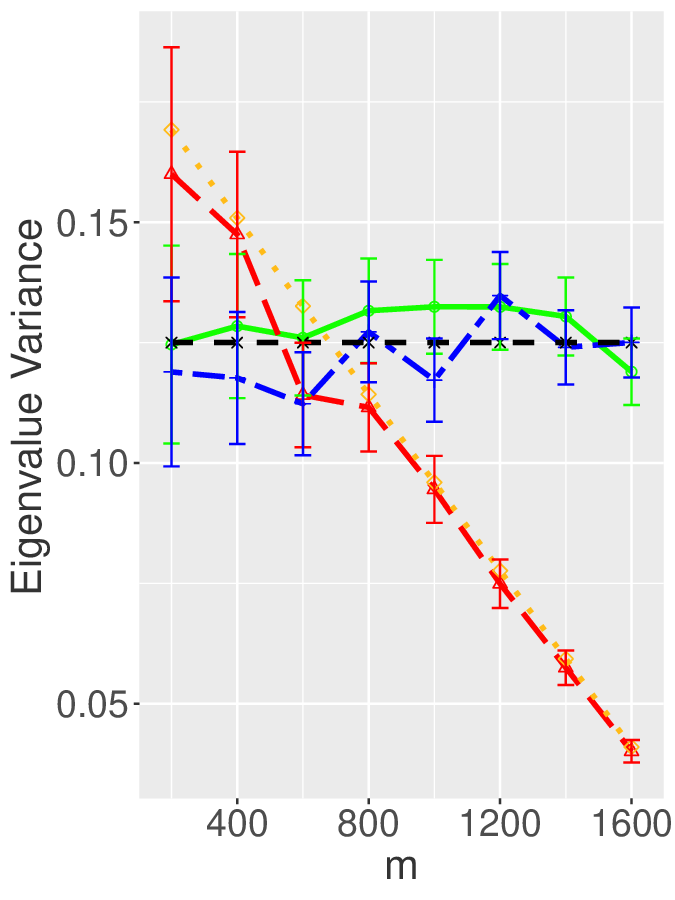}}
 	\vspace{3pt}
 	\centerline{\includegraphics[width=\textwidth]{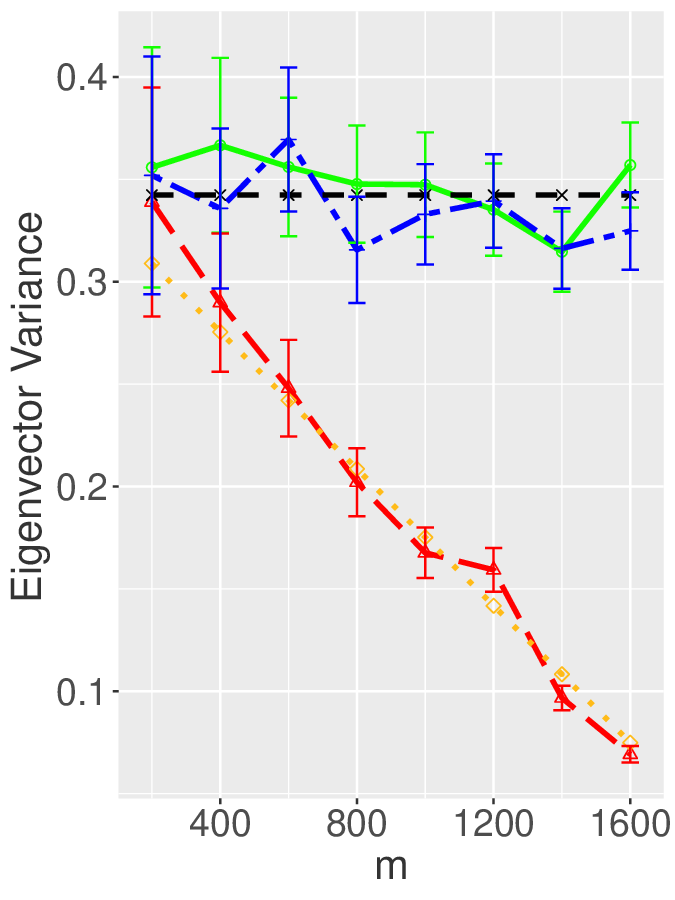}}
 	\vspace{3pt}
 	\centerline{(a) $k=2$, Case 1}
    \end{minipage}
    \begin{minipage}{0.23\linewidth}
 	\vspace{3pt}
 	\centerline{\includegraphics[width=\textwidth]{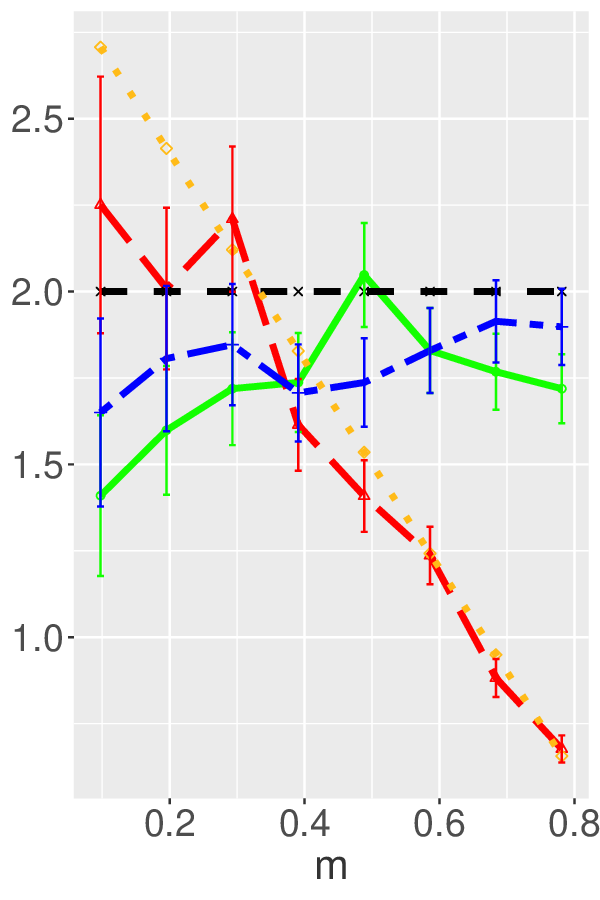}}
 	\vspace{3pt}
 	\centerline{\includegraphics[width=\textwidth]{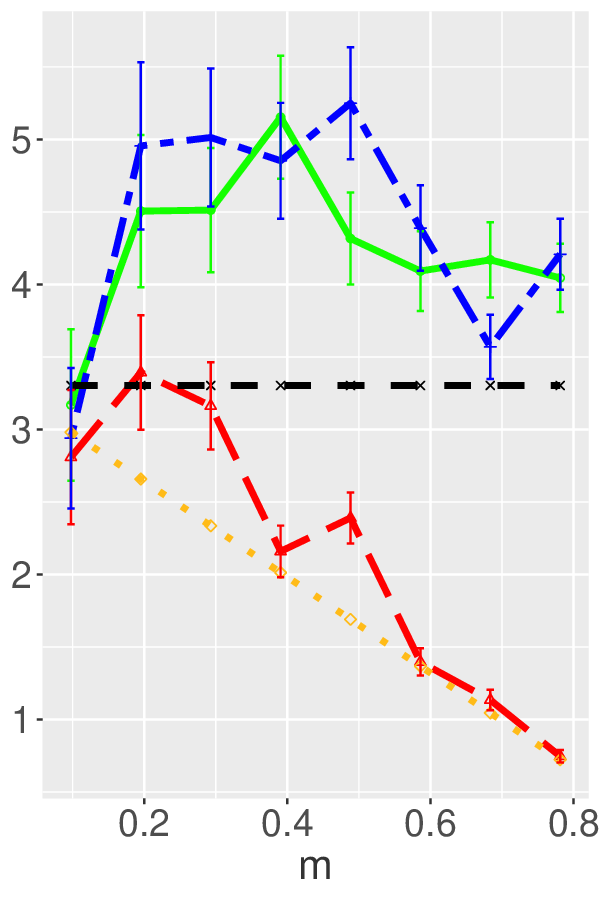}}
 	\vspace{3pt}
 	\centerline{(b) $k=1$, Case 2}
    \end{minipage}
    \begin{minipage}{0.23\linewidth}
 	\vspace{3pt}
 	\centerline{\includegraphics[width=\textwidth]{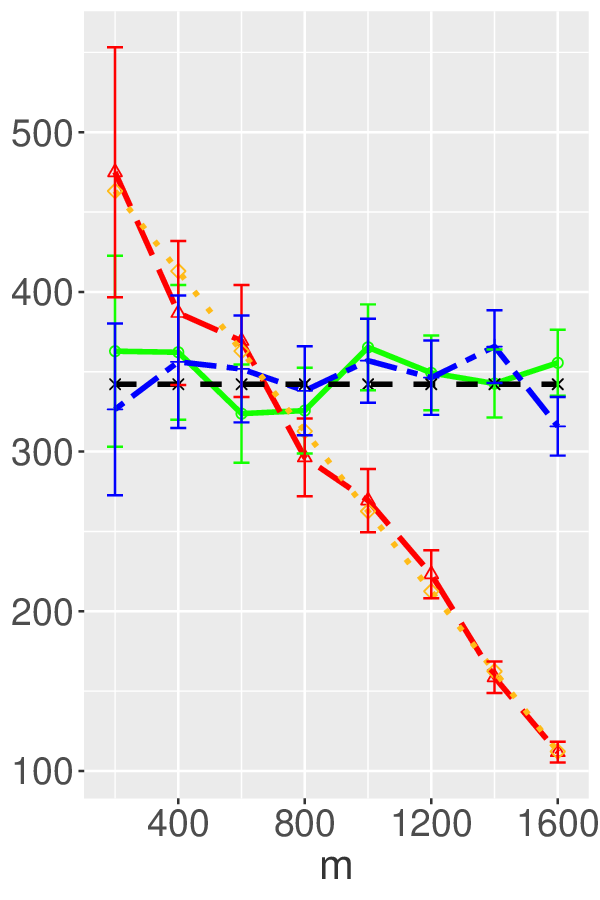}}
 	\vspace{3pt}
 	\centerline{\includegraphics[width=\textwidth]{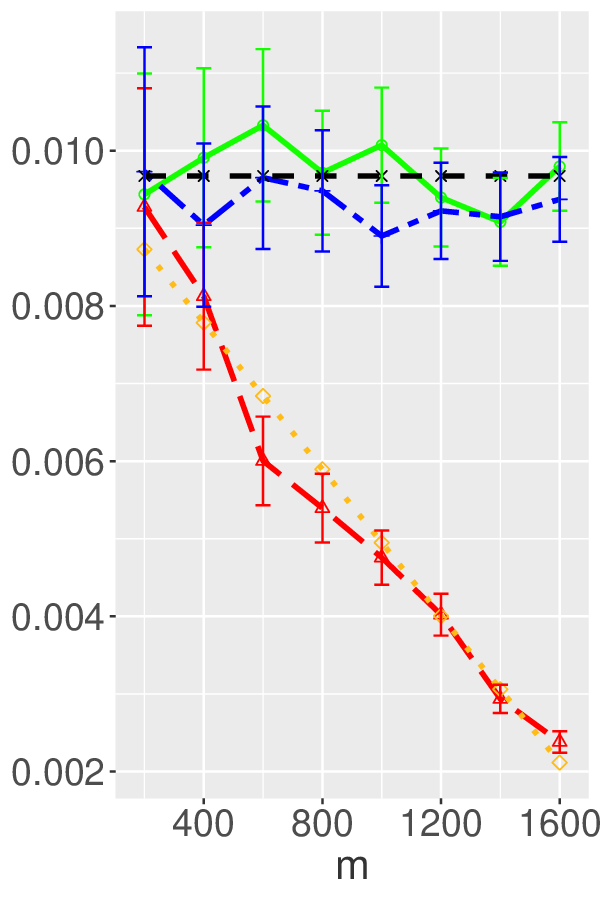}}
 	\vspace{3pt}
 	\centerline{(c) $k=1$, Case 3}
    \end{minipage}
    \begin{minipage}{0.23\linewidth}
 	\vspace{3pt}
 	\centerline{\includegraphics[width=\textwidth]{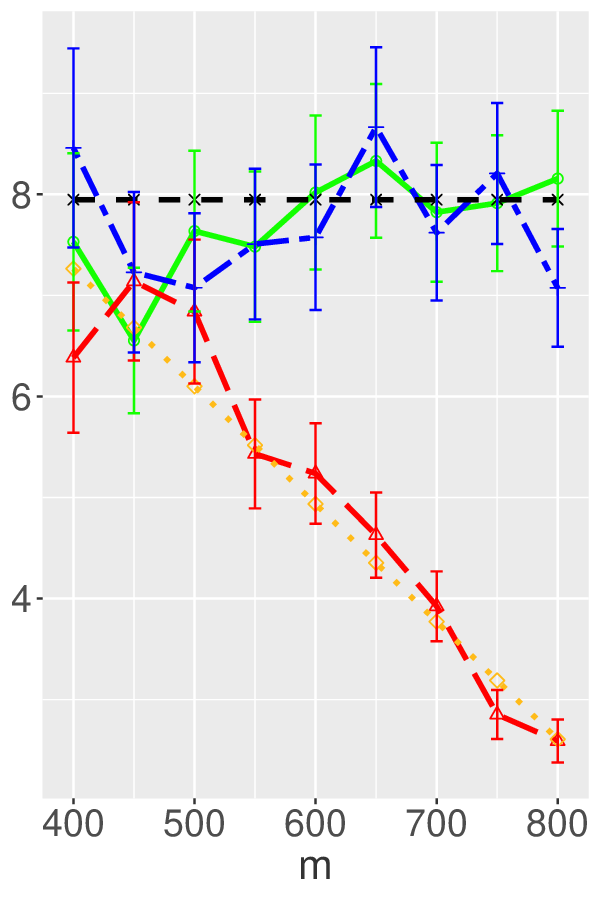}}
 	\vspace{3pt}
 	\centerline{\includegraphics[width=\textwidth]{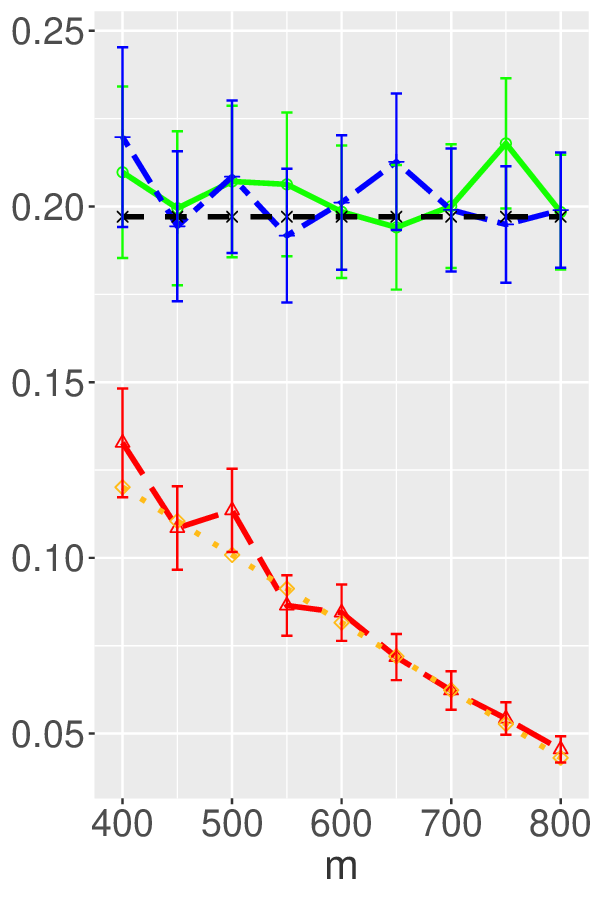}}
 	\vspace{3pt}
 	\centerline{(d) $k=1$, HGDP Dataset}
    \end{minipage}
    \caption{The variance of $\sqrt{m}\hat{\Lambda}_{m,n,k}$ and $\sqrt{m}c^\top\hat{v}_{m,n,k}$ for various cases, using the same protocol as in \Cref{n2048p15_Case1_var_coverage_merge}: 
    (a) $k=2$, Case 1. (b) $k=1$, Case 2. (c) $k=1$, Case 3. (d) $k=1$, HGDP dataset.}
    \label{var_merge}
\end{figure}

{\small
\setlength{\bibsep}{0.2pt plus 0.3ex}

}

\end{document}